\documentclass[a4paper]{amsart}

\title{A Luna \'etale slice theorem for algebraic stacks}

\author[J.~Alper]{Jarod Alper}
\address{Department of Mathematics, University of Washington, Box 354350, Seattle, WA 98195-4350, USA}
\email{jarod@uw.edu}
\author[J.~Hall]{Jack Hall}
\address{Department of Mathematics, University of Arizona, Tucson, AZ 85721-0089, USA}
\email{jackhall@math.arizona.edu}
\address{School of Mathematics \& Statistics, The University of Melbourne, Parkville, VIC, 3010, Australia}
\email{jack.hall@unimelb.edu.au}
\author[D.~Rydh]{David Rydh}
\address{Department of Mathematics, KTH Royal Institute of Technology, SE-100 44 Stockholm, Sweden}
\email{dary@math.kth.se}
\thanks{During the preparation of this article, the first author was partially supported by the Australian Research 
Council grant DE140101519, the National Science Foundation grant DMS-1801976 and by a Humboldt Fellowship. The second author was partially supported by the Australian Research Council grant DE150101799. The third author was partially supported by the Swedish Research Council grants 2011-5599 and 2015-05554.}

\date{November 23, 2020}

        \usepackage{latexsym}
        \usepackage{amssymb}
        \usepackage{amsmath}
        \usepackage{amsfonts}
        \usepackage{amsthm}
        \usepackage[hypertexnames=false]{hyperref}
        \ifpdf
          \usepackage[all,knot,poly,2cell]{xy}
        \else
          \usepackage[all,knot,poly,2cell,dvips]{xy}
        \fi
        \newdir{(}{{}*!/-5pt/@^{(}}
             \UseAllTwocells
        \usepackage{xspace}
        \usepackage{comment}
        \usepackage{fixltx2e}
        \usepackage[toc,page]{appendix}
        \usepackage{url}
        \usepackage{color}
        \usepackage{fixltx2e}
        \usepackage[inline]{enumitem}
        \setlist[enumerate]{font=\normalfont}
        \usepackage{mathtools} 
        
        \usepackage{eucal}
        \usepackage{stmaryrd} 
        \theoremstyle{plain}
        \newtheorem{theorem}{Theorem}[section]
        \newtheorem{corollary}[theorem]{Corollary}
        \newtheorem{lemma}[theorem]{Lemma}
        \newtheorem{question}[theorem]{Question}
        \newtheorem{proposition}[theorem]{Proposition}

        \theoremstyle{definition}
        \newtheorem{definition}[theorem]{Definition}
        \newtheorem{example}[theorem]{Example}
        
        \newtheorem*{example*}{Example}

        \theoremstyle{remark}
        \newtheorem{remark}[theorem]{Remark}
        \newtheorem*{remark*}{Remark}  

\numberwithin{equation}{section}

\newcommand{\inj}{\hookrightarrow} 
\newcommand{\surj}{\twoheadrightarrow} 
\newcommand{\stX}{\mathcal{X}} 
\newcommand{\stY}{\mathcal{Y}} 
\newcommand{\stZ}{\mathcal{Z}} 
\newcommand{\stW}{\mathcal{W}} 
\newcommand{\stG}{\mathcal{G}} 
\newcommand{\sE}{\mathcal{E}} 
\newcommand{\sF}{\mathcal{F}} 
\newcommand{\sI}{\mathcal{I}} 
\newcommand{\idp}{\mathfrak{p}} 
\newcommand{\idm}{\mathfrak{m}} 
\newcommand{\AR}[1]{\mathrm{(AR)_{#1}}} 
\DeclareMathOperator{\Gr}{Gr} 

\newcommand{\fm}{\mathfrak{m}}
\newcommand{\GG}{\mathbb{G}}
\newcommand{\PP}{\mathbb{P}}
\newcommand{\red}{\mathrm{red}}
\newcommand{\op}{\mathrm{op}} 
\renewcommand\AA{\mathbb{A}}
\newcommand\NN{\mathbb{N}}
\newcommand\ZZ{\mathbb{Z}}

\newcommand\CC{\mathbb{C}}
\newcommand\fM{\mathfrak{M}}
 
\newcommand\cA{\mathcal{A}}
\newcommand\cC{\mathcal{C}}
\newcommand\cE{\mathcal{E}}
\newcommand\cF{\mathcal{F}}
\newcommand\cG{\mathcal{G}}
\newcommand\cH{\mathcal{H}}
\newcommand\cI{\mathcal{I}}

\newcommand\cL{\mathcal{L}}
\newcommand\cM{\mathcal{M}}
\newcommand\cN{\mathcal{N}}
\newcommand\cW{\mathcal{W}}
\newcommand\cV{\mathcal{V}}
\newcommand\cU{\mathcal{U}}

\newcommand\cX{\mathcal{X}}
\newcommand\cY{\mathcal{Y}}
\newcommand\cZ{\mathcal{Z}}
\newcommand\oh{\mathcal{O}}
\newcommand\co{\colon}
\newcommand\Coh{\mathsf{Coh}}
\newcommand\Mod{\mathsf{Mod}}
\newcommand\Ab{\mathsf{Ab}}
\newcommand\Quot{\mathsf{Quot}}
\DeclareMathOperator{\Sing}{Sing}
\DeclareMathOperator{\Spec}{Spec}
\DeclareMathOperator{\Spf}{Spf}

\DeclareMathOperator{\GL}{GL}
\DeclareMathOperator{\Sym}{Sym}
\newcommand\Sets{\mathsf{Sets}}
\newcommand{\gitq}{/\!\!/}
\newcommand{\ilim}{\varprojlim}
\newcommand{\dlim}{\varinjlim}
\renewcommand{\tilde}{\widetilde}
\DeclareMathOperator{\id}{id}
\newcommand{\iso}{\stackrel{\sim}{\to}}
\renewcommand{\ss}{\mathrm{ss}}

\renewcommand{\hat}{\widehat}
\DeclareMathOperator{\ev}{ev}
\DeclareMathOperator{\Isom}{Isom}

\newcommand{\Gmu}{\pmb{\mu}}
\newcommand{\tensor}{\otimes}
\DeclareMathOperator{\Aut}{Aut}
\newcommand{\itemref}[1]{\eqref{#1}}
\DeclareMathOperator{\Def}{Def}
\newcommand{\QCoh}{\mathsf{QCoh}}
\newcommand{\DCAT}{\mathsf{D}}
\newcommand{\DQCOH}{\DCAT_{\QCoh}}
\DeclareMathOperator{\Hom}{Hom}
\newcommand{\Sch}{\mathsf{Sch}}
\newcommand{\DERF}[1]{\mathsf{{#1}}}
\newcommand{\RDERF}{\DERF{R}}
\newcommand{\LDERF}{\DERF{L}}
\DeclareMathOperator{\Ext}{Ext}
\newcommand{\shv}[1]{\mathcal{{#1}}}

\newcommand{\cplx}[1]{\shv{{#1}}^\bullet}
\newcommand{\QCPBK}[1]{\LDERF {#1}^*_{\QCoh}}

\newcommand{\spref}[1]{\href{http://stacks.math.columbia.edu/tag/#1}{#1}}
\DeclareMathOperator{\coker}{coker}
\DeclareMathOperator{\im}{im}
\newcommand{\normin}[2]{\tilde{#1}_{#2}}

\begin{document}
\setcounter{tocdepth}{1}
\begin{abstract} We prove that every algebraic stack, locally of finite type over an algebraically closed field with affine stabilizers, is \'etale-locally a quotient stack in a neighborhood of a point with linearly reductive stabilizer group.  The proof uses an equivariant version of Artin's algebraization theorem proved in the appendix.   We provide numerous applications of the main theorems.
\end{abstract}
\maketitle
\section{Introduction}

Quotient stacks form a distinguished class of algebraic stacks which provide intuition for 
the geometry of general algebraic stacks. Indeed, equivariant algebraic geometry has a 
long history with a wealth of tools at its disposal. Thus, it has long been desired---and 
more recently believed \cite{alper-quotient,alper-kresch}---that certain algebraic 
stacks are locally quotient stacks. This is fulfilled by the main result of this article:

\begin{theorem}\label{T:field}
Let $\cX$ be a quasi-separated algebraic stack, locally 
of finite type over an algebraically closed field $k$, with affine stabilizers.  Let $x \in \cX(k)$ be a point and $H \subseteq G_x$ be a subgroup scheme of the stabilizer such that $H$ is linearly reductive and $G_x / H$ is smooth (resp.\ \'etale).  Then there exists an affine scheme $\Spec A$ with an action of $H$, a $k$-point $w \in \Spec A$ fixed by $H$, and a smooth (resp.\ \'etale) morphism
$$f\colon \bigl([\Spec A/H],w\bigr) \to (\cX,x)$$
such that $BH \cong f^{-1}(BG_x)$;  in particular, $f$ induces the given inclusion $H \to G_x$ on stabilizer group schemes at $w$. In addition, if $\cX$ has affine diagonal, then the morphism $f$ can be arranged to be affine.
\end{theorem}

This justifies the philosophy that quotient stacks of 
the form $[\Spec A / G]$, where $G$ is a linearly reductive group, are the building blocks
 of algebraic stacks near points with linearly reductive stabilizers. 
 
In the case of smooth algebraic stacks, we can provide a more refined description (Theorem \ref{T:smooth}) which  
resolves the algebro-geometric counterpart to the Weinstein 
conjectures~\cite{weinstein_linearization}---now known as Zung's Theorem 
\cite{zung_proper_grpds,MR2776372,MR3185351,MR3259039}---on the linearization of 
proper Lie groupoids in differential geometry.  Before we state the second theorem, we introduce the following notation: 
if $\cX$ is an algebraic stack over a field $k$ and $x\in \cX(k)$ is a closed
point with stabilizer group scheme $G_x$, then we let $N_x$ denote the normal
space to $x$ viewed as a $G_x$-representation. If $\cI \subseteq \oh_{\cX}$
denotes the sheaf of ideals defining $x$, then $N_{x} = (\cI/\cI^2)^\vee$.  If
$G_x$ is smooth, then $N_{x}$ is identified with the tangent space of $\cX$ at
$x$; see Section \ref{S:tangent}.

\begin{theorem}\label{T:smooth}
Let $\cX$ be a quasi-separated algebraic stack, locally 
of finite type over an algebraically closed field $k$, with affine stabilizers.  Let $x \in |\cX|$ be a smooth and closed point with linearly reductive stabilizer
  group $G_x$. Then there exists an affine and \'etale morphism $(U,u) \to (N_x \gitq G_x,0)$, where  $N_{x} \gitq G_x$ denotes the GIT quotient, and a cartesian diagram
 $$\xymatrix{
 \bigl([N_{x}/G_x],0\bigr) \ar[d] & \bigl([W / G_x],w\bigr)\ar[r]^-f \ar[d] \ar[l]	& (\cX,x) \\
 (N_{x} \gitq G_x,0) & (U,u) \ar[l]	\ar@{}[ul]|\square				& 
}$$
such that $W$ is affine and $f$ is \'etale and induces an isomorphism
    of stabilizer groups at $w$. In addition, if $\cX$ has affine diagonal, then the morphism $f$ can be
  arranged to be affine.
\end{theorem}

In particular, this theorem implies that $\cX$ and $[N_{x} /G_x]$ have a common \'etale 
neighborhood of the form $[\Spec A / G_x]$. 

The main techniques employed in the proof of Theorem \ref{T:smooth} are
\begin{enumerate}
\item deformation theory,
\item coherent completeness,
\item Tannaka duality, and
\item Artin approximation.
\end{enumerate}
Deformation theory produces an isomorphism between the $n$th infinitesimal neighborhood $\cN^{[n]}$ of 
$0$ in $\cN = [N_x/G_x]$ and the $n$th infinitesimal neighborhood $\cX^{[n]}_x$ of $x$ in 
$\cX$. It is not at all obvious, however, that the system of morphisms $\{f^{[n]} \colon \cN^{[n]} \to \cX\}$ algebraizes. We establish algebraization in two steps. 

The first step is effectivization. To accomplish this, we introduce \emph{coherent completeness}, a key concept of the article. Recall that if $(A,\fm)$ is a complete local ring, then 
$\Coh(A) = \ilim_n \Coh(A/\fm^{n+1})$. Coherent completeness (Definition \ref{D:cc}) is a generalization of this, 
which is more refined than the formal GAGA results of \cite[III.5.1.4]{EGA} and 
\cite{geraschenko-brown} (see \S\ref{A:gms_app}). What we prove in \S\ref{S:cc} is the 
following.
\begin{theorem} \label{key-theorem}
Let $G$ be a linearly reductive affine group scheme over a field $k$. Let $\Spec A$ be a noetherian affine scheme with an action of~$G$, and let $x \in \Spec A$ be a closed point fixed by $G$.   Suppose that $A^{G}$ is a complete local ring. Let $\cX = [\Spec A / G]$ and let $\cX^{[n]}_x$ be the $n$th infinitesimal neighborhood of $x$.  Then the natural functor
\begin{equation} \label{eqn-coh}
\Coh(\cX)  \to  \ilim_n \Coh\bigl(\cX^{[n]}_x\bigr)
\end{equation}
is an equivalence of categories.
\end{theorem}
Tannaka duality for algebraic stacks with affine stabilizers was recently established by the 
second two authors \cite[Thm.~1.1]{hallj_dary_coherent_tannakian_duality} (also see 
Theorem \ref{T:tannakian}). This proves that morphisms between algebraic stacks $\cY \to \cX$ are equivalent to symmetric monoidal functors $\Coh(\cX) \to \Coh(\cY)$. Therefore, to prove Theorem \ref{T:smooth}, we can  combine Theorem \ref{key-theorem} with Tannaka duality (Corollary \ref{C:tannakian}) and the above deformation-theoretic observations to show that the morphisms $\{f^{[n]} \colon \cN^{[n]} \to \cX\}$ effectivize to $\hat{f} \colon \hat{\cN} \to \cX$, where $\hat{\cN} = \cN \times_{N_x\gitq G_x} \Spec \hat{\oh}_{N_x\gitq G_x,0}$. 
The morphism $\hat{f}$ is then algebraized using Artin approximation \cite{artin-approx}. 

The techniques employed in the proof of Theorem \ref{T:field} are similar, but the methods are more involved. Since we no longer assume that $x \in \cX(k)$ is a non-singular point, we cannot expect an \'etale or smooth morphism $\cN^{[n]}\to \cX_x^{[n]}$ where $\cN=[N_x/H]$. Using Theorem \ref{key-theorem} and Tannaka duality, however, we can produce a closed substack $\hat{\cH}$ of $\hat{\cN}$ and a formally versal morphism $\hat{f} \colon \hat{\cH} \to \cX$. To algebraize $\hat{f}$, we apply an equivariant version of Artin algebraization (Corollary \ref{C:equivariant-algebraization}), which we believe is of independent interest.

For tame stacks with finite inertia, Theorem \ref{T:field} is one of the main results of \cite{tame}. The structure of algebraic stacks with infinite stabilizers has been poorly understood until the present article. 
For algebraic stacks with infinite stabilizers that are not---or are not known to be---quotient stacks, Theorems \ref{T:field} and \ref{T:smooth} were only known when $\cX = \fM_{g,n}^{\ss}$ is the 
moduli stack of semistable curves. This is the central result of \cite{alper-kresch}, where it is also shown that $f$ can be arranged to be representable. For certain quotient stacks, Theorems \ref{T:field} and \ref{T:smooth} can be obtained 
using traditional methods in equivariant algebraic geometry, see \S\ref{A:luna} for details. 

\subsection{Some remarks on the hypotheses}
We mention here several examples illustrating the necessity of some of the hypotheses of Theorems \ref{T:field} and \ref{T:smooth}.

\begin{example} \label{ex1}
Some reductivity assumption of the stabilizer $G_x$ is necessary in Theorem \ref{T:field}.  For instance,  consider the group scheme $G= \Spec k[x,y]_{xy+1} \to \AA^1 = \Spec k[x]$ (with multiplication defined by $y \mapsto xyy' + y + y'$), where the generic fiber is $\GG_m$ but the fiber over the origin is $\GG_a$.  Let $\cX = BG$ and $x \in |\cX|$ be the point corresponding to the origin.  There does not exist an \'etale morphism $([W/ \GG_a], w) \to (\cX, x)$, where $W$ is an algebraic space over $k$ with an action of $\GG_a$.
\end{example}

\begin{example} \label{ex2}
It is essential to require that the stabilizer groups are affine in a neighborhood of $x \in |\cX|$.  For instance, let $X$ be a smooth curve and let $\cE \to X$ be a group scheme whose generic fiber is a smooth elliptic curve but the fiber over a point $x \in X$ is isomorphic to $\GG_m$.  Let $\cX = B\cE$.  There is no \'etale morphism $([W/ \GG_m], w) \to (\cX, x)$, where $W$ is an affine $k$-scheme with an action of $\GG_m$.
\end{example}

\begin{example} \label{ex3}
In the context of Theorem \ref{T:field}, it is not possible in general to find a Zariski-local quotient presentation of the form $[\Spec A/G_x]$.  Indeed, if $C$ is the projective nodal cubic curve with $\GG_m$-action, then there is no Zariski-open $\GG_m$-invariant affine neighborhood of the node.  If we view $C$ ($\GG_m$-equivariantly) as the $\ZZ/2\ZZ$-quotient of the union of two $\PP^1$'s glued along two nodes, then after removing one of the nodes, we obtain a (non-finite) \'etale morphism $[\Spec(k[x,y]/xy)/\GG_m] \to [C/\GG_m]$ where $x$ and $y$ have weights $1$ and $-1$.  This is in fact the unique such quotient presentation (see Remark \ref{R:uniqueness}).
\end{example}

The following two examples illustrate that in Theorem \ref{T:field} it is not always possible to obtain a quotient presentation $f\colon [\Spec A/G_x]\to \cX$, such that $f$ is representable or separated without additional hypotheses; see also Question \ref{Q:represent}.

\begin{example} \label{ex5}
Consider the
non-separated affine line as a group scheme $G \to \AA^1$ whose generic fiber
is trivial but the fiber over the origin is $\ZZ/2\ZZ$.  Then $BG$ admits an
\'etale neighborhood $f \co [\AA^1/(\ZZ/2\ZZ)] \to BG$ which induces an isomorphism
of stabilizer groups at $0$, but $f$ is not representable in a neighborhood. 
\end{example}

\begin{example} \label{ex4}
Let $\mathcal{L}og$ (resp.\ $\mathcal{L}og^{\mathrm{al}}$) be the algebraic stack of log structures (resp.\ aligned log structures) over $\Spec k$ introduced in \cite{olssonloggeometry} (resp.\ \cite{acfw}).  Let $r\geq 2$ be an integer and let $\NN^r$ be the free log structure on $\Spec k$.  There is an \'etale neighborhood $[\Spec k[\NN^r] / (\GG_m^r \rtimes S_r)] \to \mathcal{L}og$ of $\NN^r$ which is not representable.  Note that $\mathcal{L}og$ does not have separated diagonal.
Similarly, there is an \'etale neighborhood $[\Spec k[\NN^r] / \GG_m^r] \to \mathcal{L}og^{\mathrm{al}}$ of $\NN^r$ (with the standard alignment) which is representable but not separated.  Because $[\Spec k[\NN^r] / \GG_m^r] \to \mathcal{L}og^{\mathrm{al}}$ is inertia-preserving, $\mathcal{L}og^{\mathrm{al}}$ has affine inertia and hence separated diagonal; however, the diagonal is not affine.
  In both cases, this is the unique such quotient presentation (see Remark \ref{R:uniqueness}).
\end{example}

\subsection{Generalizations}
Using similar arguments, one can in fact establish a generalization of Theorem \ref{T:field} to the relative and mixed characteristic setting.  This requires developing some background material on deformations of linearly reductive group schemes, a more general version of Theorem \ref{key-theorem} and a generalization of the formal functions theorem for good moduli spaces.  To make this article more accessible, we have decided to postpone the relative statement until the follow-up article \cite{ahr2}.

If $G_x$ is not reductive, it is possible that one could find an \'etale
neighborhood $([\Spec A/\GL_n],w)\to (\cX,x)$. However, this is not known even if
$\cX=B_{k[\epsilon]}G_\epsilon$ where $G_\epsilon$ is a deformation of
a non-reductive algebraic group~\cite{mathoverflow_groups-over-dual-numbers}.

In characteristic $p$, the linearly reductive hypothesis in Theorems \ref{T:field} and \ref{T:smooth} is quite restrictive.  Indeed, a smooth affine group scheme $G$ over an algebraically closed field $k$ of characteristic $p$ is linearly reductive if and only if $G^0$ is a torus and $|G/G^0|$ is coprime to $p$ \cite{nagata}.  We ask however:

\begin{question} \label{Q:geom-red} Does a variant of Theorem \ref{T:field} remain true if ``linearly reductive" is replaced with ``reductive"?
\end{question}

We remark that if $\cX$ is a Deligne--Mumford stack, then the conclusion of Theorem \ref{T:field} holds. We also ask:
\begin{question}\label{Q:represent} If $\cX$ has separated (resp.\ quasi-affine) diagonal, then  can the morphism $f$ in Theorems \ref{T:field} and \ref{T:smooth} be chosen to be representable (resp.\ quasi-affine)?
\end{question} 

If $\cX$ does not have separated diagonal, then the morphism $f$ cannot
necessarily be chosen to be representable; see Examples \ref{ex5} and \ref{ex4}.
We answer Question \ref{Q:represent} affirmatively when $\cX$ has affine diagonal (Proposition \ref{P:refinement}) or is a quotient stack (Corollary \ref{C:refinement:quot_stack}), 
or when $H$ is
diagonalizable (Proposition \ref{P:rep_diagonalizable}).

\subsection{Applications} \label{S:intro-applications}
Theorems \ref{T:field}  and \ref{T:smooth}  yield a number of applications to old and new problems.

\subsubsection*{Immediate consequences}\label{A:immediate}
Let $\cX$ be a quasi-separated algebraic stack, locally of finite type over an algebraically closed field $k$ with affine stabilizers, and let $x \in \cX(k)$ be a point with linearly reductive stabilizer $G_x$.
\begin{enumerate}
\item\label{I:immediate-embedding} There is an \'etale neighborhood of $x$ with a closed embedding into a smooth algebraic stack.
\item There is an \'etale-local description of the cotangent complex $L_{\cX/k}$ of $\cX$ in terms of the cotangent complex $L_{\cW/k}$ of $\cW=[\Spec A/G_x]$.  If $x \in |\cX|$ is a smooth point (so that $\cW$ can be taken to be smooth) and $G_x$ is smooth, then $L_{\cW/k}$ admits an explicit description.  If $x$ is not smooth but $G_x$ is smooth, then the $[-1,1]$-truncation of $L_{\cW/k}$ can be described explicitly by appealing to \itemref{I:immediate-embedding}.
\item For any representation $V$ of $G_x$, there exists a vector bundle over an \'etale neighborhood of $x$ extending $V$.
\item If $G_x$ is smooth, then there are $G_x$-actions on the formal miniversal deformation space $\hat{\Def}(x)$ of $x$ and its versal object, and the $G_x$-invariants of $\hat{\Def}(x)$ is the completion of a finite type $k$-algebra.  This observation is explicitly spelled out in Remark~\ref{R:miniversal-completion-finite-type}.
\item Any specialization $y \rightsquigarrow x$  of $k$-points is realized by a morphism $[\AA^1/\GG_m] \to \cX$. This follows by applying the Hilbert--Mumford criterion to an \'etale quotient presentation constructed by Theorem \ref{T:field}.
\end{enumerate}

\subsubsection*{Local applications} The following consequences of Theorems \ref{T:field} and \ref{T:smooth} to the local geometry of algebraic stacks will be detailed in Section~\ref{S:local_applications}: 
\begin{enumerate}
\item  a generalization of Sumihiro's theorem on torus actions to Deligne--Mumford stacks 
  (\S\ref{A:sumihiro}), confirming an expectation of Oprea \cite[\S2]{oprea};
\item a generalization of Luna's \'etale slice theorem to non-affine schemes 
  (\S\ref{A:luna});
\item the existence of equivariant miniversal deformation spaces for singular curves (\S\ref{A:mv_curve}), generalizing~\cite{alper-kresch};
\item the \'etale-local quotient structure of a good moduli space (\S\ref{A:gms_app}), which in particular establishes formal GAGA for good moduli space morphisms, resolving a conjecture of Geraschenko--Zureick-Brown \cite[Conj.\ 32]{geraschenko-brown};
\item the existence of coherent completions of algebraic stacks at points with linearly reductive stabilizer (\S\ref{A:coherent-completion});
\item a criterion for \'etale-local equivalence of algebraic stacks (\S\ref{A:etale}), 
  extending Artin's corresponding results for schemes \cite[Cor.~2.6]{artin-approx};

\item the resolution property holds \'etale-locally for algebraic stacks near a point whose stabilizer has linearly reductive connected component (\S\ref{A:resolution-property-etale}), which in particular provides a characterization of toric Artin stacks in terms of stacky fans  \cite[Thm.~6.1]{GS-toric-stacks-2}.

\end{enumerate}

\subsubsection*{Global applications} In Section~\ref{S:global_applications}, we provide the following global applications:
\begin{enumerate}
\item compact generation of derived categories of algebraic stacks (\S\ref{A:compact-generation});
\item a criterion for the existence of a good moduli space (\S\ref{A:gms}), generalizing 
  \cite{keel-mori,afsw-good};
\item algebraicity of stacks of coherent sheaves, Quot schemes, Hilbert schemes, Hom stacks and equivariant Hom stacks (\S\ref{A:algebraicity});
  \item  a short proof of Drinfeld's results \cite{drinfeld} on algebraic spaces with a $\GG_m$-action and a generalization to Deligne--Mumford stacks with $\GG_m$-actions (\S\ref{A:drinfeld}); and
\item Bia\l ynicki-Birula decompositions for separated Deligne--Mumford stacks 
  (\S\ref{A:BB}). 
\end{enumerate}
We also note that Theorem \ref{T:field} was applied recently by Toda to resolve a wall-crossing conjecture for higher rank DT/PT invariants by Toda \cite{toda-hallalg}.

\subsection{Leitfaden} \label{S:leitfaden}
The logical order of this article is as follows. In Section~\ref{S:cc-tannaka} we
prove the key coherent completeness result (Theorem~\ref{key-theorem}). In Appendix~\ref{A:approx} we state Artin approximation and 
prove an equivariant version of Artin algebraization (Corollary \ref{C:equivariant-algebraization}). These techniques are
then used in Section~\ref{S:proof-section} to prove the main local structure theorems (Theorems~\ref{T:field}
and~\ref{T:smooth}). In Sections~\ref{S:local_applications} and \ref{S:global_applications} we give applications to the
main theorems.

\subsection{Notation} \label{S:notation} 
An algebraic stack $\cX$ is quasi-separated if the diagonal and the diagonal of the diagonal are quasi-compact.
An algebraic stack $\cX$ has \emph{affine stabilizers} if for every field $K$
and point $x\colon \Spec K\to \cX$, the stabilizer group $G_x$ is affine.
If $\cX$ is an algebraic stack and $\cZ \subseteq \cX$ is a closed substack, we will denote by $\cX_{\cZ}^{[n]}$  the $n$th nilpotent thickening of $\cZ \subseteq \cX$
(i.e., if $\cI \subseteq \oh_{\cX}$ is the ideal sheaf defining $\cZ$, then $\cX_{\cZ}^{[n]} \to \cX$ is defined by $\cI^{n+1}$).  If $x\in |\stX|$ is a closed
point, then the \emph{$n$th infinitesimal neighborhood of
  $x$} is the $n$th nilpotent thickening of the inclusion of the residual gerbe $\cG_x \to \cX$.

Recall from \cite{alper-good} that a quasi-separated and quasi-compact morphism $f \co \cX \to \cY$ of algebraic stacks is {\it cohomologically affine} if the push-forward functor $f_*$ on the category of quasi-coherent $\oh_{\cX}$-modules is exact.
If $\cY$ has quasi-affine diagonal and $f$ has affine diagonal, then $f$ is cohomologically affine if and only if $\RDERF f_* \colon \DQCOH^+(\cX) \to \DQCOH^+(\cY)$ is $t$-exact, cf.\ \cite[Prop.~3.10~(vii)]{alper-good} and \cite[Prop.~2.1]{hallj_neeman_dary_no_compacts}; this equivalence is false if $\cY$ does not have affine stabilizers \cite[Rem.~1.6]{hallj_dary_alg_groups_classifying}.
If $G \to \Spec k$ is an affine group scheme of finite type, then we say that $G$ is {\it linearly reductive} if $BG \to \Spec k$ is cohomologically affine.
A quasi-separated and quasi-compact morphism $f \co \cX \to Y$ of algebraic stacks is {\it a good moduli space} if $Y$ is an algebraic space, $f$ is cohomologically affine and $\oh_Y \to f_* \oh_{\cX}$ is an isomorphism.  

If $G$ is an affine group scheme of finite type over a field $k$ acting on an algebraic space $X$, then we say that a $G$-invariant morphism $\pi \co X \to Y$ of algebraic spaces is {\it a good GIT quotient} if the induced map $[X/G] \to Y$ is a good moduli space; we often write $Y = X \gitq G$.  In the case that $G$ is linearly reductive, a $G$-equivariant morphism $\pi \co X \to Y$ is a good GIT quotient if and only if $\pi$ is affine and $\oh_Y \to (\pi_* \oh_X)^G$ is an isomorphism.

An algebraic stack $\cX$ is said to have the {\it resolution property}
if every quasi-coherent $\oh_{\cX}$-module of finite type is a
quotient of a locally free $\oh_{\cX}$-module of finite type.
By the Totaro--Gross
theorem \cite{totaro,gross-resolution}, a quasi-compact and
quasi-separated algebraic stack is isomorphic to $[U/\GL_N]$, where
$U$ is a quasi-affine scheme and $N$ is a positive integer, if and
only if the closed points of $\cX$ have affine stabilizers and $\cX$
has the resolution property.
We will only use the easy implication ($\Longrightarrow$) in this article which can
be strengthened as follows. If $G$ is a group scheme of finite type over a
field $k$ acting on a quasi-affine scheme $U$, then $[U/G]$ has the resolution
property. This is a
consequence of the following two simple observations: (1) $BG$ has
the resolution property; and (2) if $f \colon \cX \to \cY$ is
quasi-affine and $\cY$ has the resolution property, then $\cX$ has the
resolution property. 

If $\cX$ is a noetherian algebraic stack, then we denote by $\Coh(\cX)$ the abelian category of coherent $\oh_{\cX}$-modules.

\subsection*{Acknowledgements}  We thank Andrew Kresch for many useful conversations as well as Dragos Oprea and Michel Brion for some helpful comments.  Finally, we thank the referee whose careful reading greatly improved this article.

\section{Coherently complete stacks and Tannaka duality} \label{S:cc-tannaka}

\subsection{Coherently complete algebraic stacks} \label{S:cc}
Motivated by Theorem \ref{key-theorem}, we begin this section with the following key definition. 
\begin{definition} \label{D:cc}
Let $\cX$ be a noetherian algebraic stack with affine stabilizers and let $\cZ \subseteq \cX$ be a closed substack. For a coherent $\oh_\cX$-module $\sF$, the restriction to $\cX_{\cZ}^{[n]}$ is denoted $\sF_n$.
We say that $\cX$ is {\it coherently complete along $\cZ$} if the natural functor
\[
\Coh(\cX)  \to  \ilim_n \Coh\bigl(\cX_{\cZ}^{[n]}\bigr), \qquad \sF\mapsto \{\cF_n\}_{n\geq 0}
\]
is an equivalence of categories.
 \end{definition}
 We now recall some examples of coherently complete algebraic
 stacks. Traditionally, such results have been referred to as ``formal
 GAGA'' theorems. 
 \begin{example}\label{E:cc}
   Let $A$ be a noetherian ring and let $I \subseteq A$ be an
   ideal. Assume that $A$ is $I$-adically complete, that is,
   $A\simeq \varprojlim_n A/I^{n+1}$.  Then $\Spec A$ is coherently
   complete along $\Spec A/I$. More generally if an algebraic stack
   $\cX$ is proper over $\Spec A$, then $\cX$ is coherently complete
   along $\cX \times_{\Spec A} \Spec A /I$. See \cite[III.5.1.4]{EGA}
   for the case of schemes and \cite[Thm.~1.4]{olsson-proper},
   \cite[Thm.~4.1]{conrad-gaga} for algebraic stacks.
\end{example}
\begin{example}\label{E:gzb}
  Let $(R,\fm)$ be a complete noetherian local ring and let
  $\pi\colon \cX \to \Spec R$ be a good moduli space of finite
  type. If $\cX$ has the resolution property (e.g.,
  $\cX \simeq [\Spec B/\GL_n]$, where $B^{\GL_n}=R$), then $\cX$ is
  coherently complete along $\pi^{-1}(\Spec R/\fm)$ \cite[Thm.~1.1]{geraschenko-brown}. 
\end{example}
Note that in the examples above, completeness was always along a
substack that is pulled back from an affine base. Theorem
\ref{key-theorem} is quite different, however, as the following
example highlights.
\begin{example}
  Let $k$ be a field, then the quotient stack $[\AA^1_k/\GG_m]$ has
  good moduli space $\Spec k$. Theorem \ref{key-theorem} implies that
  $[\AA^1_k/\GG_m]$ is coherently complete along the closed point
  $B\GG_m$. In this special case, one can even give a direct proof of
  the coherent completeness (see Proposition \ref{P:A1-complete}).
\end{example}
We now prove Theorem \ref{key-theorem}.
\begin{proof}[Proof of Theorem \ref{key-theorem}]
 Let $\fm \subset A$ be the maximal ideal corresponding to $x$.  A coherent $\oh_{\cX}$-module $\cF$ corresponds to a finitely generated $A$-module $M$ with an action of $G$.
Note that since $G$ is linearly reductive, $M^G$ is a finitely generated $A^G$-module \cite[Thm.\ 4.16(x)]{alper-good}.  We claim that the following two sequences of $A^G$-submodules $\{(\fm^{n} M)^{G} \}$ and $\{ (\fm^G)^n M^{G} \}$ of $M^{G}$ define the same topology, or in other words that 
\begin{equation} \label{E:formal}
M^{G}  \to  \ilim M^{G} / \bigl(\fm^{n} M\bigr)^{G}
\end{equation}
 is an isomorphism of $A^G$-modules.  
 
To this end, we first establish that
\begin{equation} \label{E:intersection}
 \bigcap_{n \ge 0} \bigl(\fm^n M\bigr)^{G} = 0,
 \end{equation} 
 which immediately informs us that \eqref{E:formal} is injective.
 Let $N =  \bigcap_{n \ge 0} \fm^n M$. The Artin--Rees lemma implies that $N=\fm N$ so $N \tensor_A A/\fm = 0$.
 Since $A^G$ is a local ring, $\Spec A$ has a unique closed orbit $\{x\}$.
 Since the support of $N$ is a closed $G$-invariant
 subscheme of $\Spec A$ which does not contain $x$, it follows that $N=0$.

We next establish that \eqref{E:formal} is an isomorphism if $A^G$ is artinian.  In this case, $\{(\fm^n M)^G\}$ automatically satisfies the Mittag-Leffler condition (it is a sequence of artinian $A^G$-modules).  Therefore, taking the inverse limit of the exact sequences $0 \to (\fm^n M)^G \to M^G \to M^G / (\fm^n M)^G \to 0$ and applying \eqref{E:intersection}, yields an exact sequence
$$0 \to 0 \to M^G \to \ilim M^G / (\fm^n M)^G \to 0.$$
Thus, we have established \eqref{E:formal} when $A^G$ is artinian.

To establish \eqref{E:formal} in the general case, let $J = (\fm^G) A \subseteq A$ and observe that
\begin{equation} \label{E:limit1}
M^G = \ilim M^G / \bigl(\fm^G\bigr)^n M^G = \ilim \bigl(M/J^n M\bigr)^G,
\end{equation}
since $G$ is linearly reductive.
For each $n$, 
we know that
\begin{equation} \label{E:limit2}
\bigl(M/J^nM\bigr)^G = \ilim_l M^G / \bigl((J^n + \fm^l)M \bigr)^G
\end{equation}
using the artinian case proved above.  Finally, combining \eqref{E:limit1} and \eqref{E:limit2} together with the observation that $J^n \subseteq \fm^l$ for $n \ge l$, we conclude that
$$\begin{aligned}
M^G & = \ilim_n \bigl(M / J^n M\bigr)^G \\
	& = \ilim_n \ilim_l M^G / \bigl((J^n + \fm^l)M \bigr)^G \\
	& = \ilim_l M^G / \bigl(\fm^l M\bigr)^G.
\end{aligned}$$

  We now show that \eqref{eqn-coh} is fully faithful.  Suppose that $\shv{G}$ and $\shv{F}$ are coherent
  $\oh_{\cX}$-modules, and let $\shv{G}_n$ and $\shv{F}_n$ denote
  the restrictions to $\cX_x^{[n]}$, respectively. We need to show
  that
  \begin{equation*}
    \Hom(\shv{G}, \shv{F}) \to \ilim \Hom(\shv{G}_n, \shv{F}_n)
  \end{equation*}
  is bijective. Since $\cX$ has the resolution
  property (see \S\ref{S:notation}), we can find locally free $\oh_{\cX}$-modules
  $\cE'$ and $\cE$ and an exact sequence
  \[
\cE' \to \cE \to \shv{G} \to 0.
  \]
  This induces a diagram
  \[
  \xymatrix{
    0 \ar[r]		& \Hom(\shv{G}, \shv{F}) \ar[r] \ar[d]		& \Hom(\cE, \shv{F}) \ar[r] \ar[d]				& \Hom(\cE', \shv{F}) \ar[d]\\
    0 \ar[r] & \ilim \Hom(\shv{G}_n, \shv{F}_n) \ar[r] &
    \ilim \Hom(\cE_n, \shv{F}_n) \ar[r] & \ilim
    \Hom(\cE'_n, \shv{F}_n) }
  \]
  with exact rows.  Therefore, it suffices to assume that $\shv{G}$ is
  locally free.  In this case,
  \[
  \Hom(\shv{G}, \shv{F}) = \Hom(\oh_{\cX}, \shv{G}^{\vee} \otimes
  \shv{F}) \quad \text{and} \quad 
  \Hom(\shv{G}_n, \shv{F}_n) = \Hom\bigl(\oh_{\cX_x^{[n]}},
  (\shv{G}_n^{\vee} \otimes \shv{F}_n)\bigr).
  \]
Therefore, we can also assume that $\shv{G} = \oh_{\cX}$ and we need to verify that the map
  \begin{equation}\label{E:coh-ff}
    \Gamma(\cX,\shv{F}) \to \ilim
    \Gamma\bigl(\cX_x^{[n]},\shv{F}_n\bigr)
  \end{equation}
  is an isomorphism. But 
  $\Gamma\bigl(\cX_x^{[n]},\shv{F}_n\bigr)=
  \Gamma(\cX,\shv{F})/\Gamma(\cX,\fm_x^{n+1}\shv{F})$
  since $G$ is linearly reductive,
  so the map \eqref{E:coh-ff} is identified with the isomorphism \eqref{E:formal}, and the full faithfulness of \eqref{eqn-coh} follows.

  We now prove that the functor
  \eqref{eqn-coh} is essentially surjective. Let $\{\cF_n\} \in \ilim
  \Coh(\cX_x^{[n]})$ be a compatible system of coherent sheaves.
  Since $\cX$ has the
  resolution property (see \S\ref{S:notation}), there is a vector bundle $\shv{E}$ on $\cX$
  together with a surjection $\varphi_0\colon \shv{E} \to \shv{F}_0$. We
  claim that $\varphi_0$ lifts to a compatible system of morphisms
  $\varphi_n \colon \shv{E} \to \shv{F}_n$ for every $n>0$. 
  It suffices to show that for $n>0$, the natural map $\Hom(\cE, \cF_{n+1}) \to \Hom(\cE, \cF_n)$ is surjective. But this is clear: $\Ext_{\oh_{\cX}}^{1}(\cE, \fm^{n+1} \cF_{n+1}) = 0$ since $\cE$ is locally free and $G$ is linearly reductive. It follows that we obtain an induced morphism of systems
  $\{\varphi_n\} \colon \{\cE_n\} \to
  \{\cF_n\}$ and, by Nakayama's Lemma, each $\varphi_n$ is surjective.

  The system of
  morphisms $\{\varphi_n\}$ admits an adic kernel $\{\shv{K}_n\}$ (see \cite[\S3.2]{hallj_dary_coherent_tannakian_duality}, which is a generalization of \cite[Tag \spref{087X}]{stacks-project} to stacks).  
Note that, in general, $\shv{K}_n \neq \ker \varphi_n$ and $\shv{K}_n$ is
  actually the ``stabilization'' of $\ker \varphi_n$ (in the sense of the
  Artin--Rees lemma). Applying the procedure above to $\{\shv{K}_n\}$,
  there is another vector bundle $\cH$ and a morphism of systems
  $\{\psi_n\} \colon \{\cH_n\} \to \{\cE_n\}$ such that
  $\coker(\psi_n) \cong \shv{F}_n$. By the full faithfulness of
  \eqref{eqn-coh}, the morphism $\{\psi_n\}$ arises from a unique
  morphism $\psi \colon \cH \to \cE$. Letting
  $\tilde{\shv{F}} = \coker \psi$, the universal property of cokernels
  proves that there is an isomorphism
  $\tilde{\shv{F}}_n \cong \shv{F}_n$; the result follows.
\end{proof}
\begin{remark} \label{R:explicit}
In this remark, we show that with the hypotheses of Theorem \ref{key-theorem} the coherent $\oh_{\cX}$-module $\cF$ extending a given system $\{\cF_n\} \in \ilim \Coh(\cX_x^{[n]})$ can in fact be constructed explicitly.  Let $\Gamma$ denote the set of irreducible representations of $G$ with $0 \in \Gamma$ denoting the trivial representation. 
For $\rho \in \Gamma$, we let $V_{\rho}$ be the corresponding irreducible representation.    
For any $G$-representation~$V$, we set
$$V^{(\rho)} = \bigl(V \otimes V_{\rho}^{\vee}\bigr)^G \otimes V_\rho.$$
Note that $V = \bigoplus_{\rho \in \Gamma} V^{(\rho)}$ and that $V^{(0)} = V^G$ is the subspace of invariants.  In particular, 
there is a decomposition $A = \bigoplus_{\rho \in \Gamma} A^{(\rho)}$.  The data of a coherent $\oh_{\cX}$-module $\cF$ is equivalent to a finitely generated $A$-module $M$ together with a $G$-action, 
i.e., an $A$-module $M$ with a decomposition $M = \oplus_{\rho \in \Gamma} M^{(\rho)}$, where each $M^{(\rho)}$ is a direct sum of copies of the irreducible representation $V_\rho$, such that the $A$-module structure on $M$ is compatible with the decompositions of $A$ and $M$.  Given a coherent $\oh_{\cX}$-module $\cF = \widetilde{M}$ and a representation $\rho \in \Gamma$, then $M^{(\rho)}$ is a finitely generated $A^G$-module and 
$$M^{(\rho)}  \to  \ilim \bigl(M/ \fm^{k} M\bigr)^{(\rho)}$$
is an isomorphism (which follows from \eqref{E:formal}).  

Conversely, given a system of 
$\{\cF_n = \widetilde{M}_n\} \in \ilim \Coh(\cX_x^{[n]})$ where each $M_n$ is a finitely generated $A / \fm^{n+1}$-module with a $G$-action, then the extension $\cF = \widetilde{M}$ can be constructed explicitly by defining:
$$M^{(\rho)} := \ilim M_n^{(\rho)} \qquad \text{ and } \qquad M := \bigoplus_{\rho \in \Gamma} M^{(\rho)}.$$
One can show directly that each $M^{(\rho)}$ is a finitely generated $A^G$-module, $M$ is a finitely generated $A$-module with a $G$-action, and $M/ \fm^{n+1} M = M_n$.
\end{remark}
 \begin{remark}  
An argument similar to the proof of the essential surjectivity of \eqref{eqn-coh} shows that every vector bundle on $\cX$ is the pullback of a $G$-representation under the projection $\pi \co \cX \to BG$.  Indeed,  given a vector bundle $\cE$ on $\cX$, we obtain by restriction a vector bundle $\cE_0$ on $BG$. 
The surjection $\pi^* \cE_0 \to \cE_0$ lifts to a map $\pi^* \cE_0 \to \cE$ since $\cX$ is cohomologically affine.  By Nakayama's Lemma, the map $\pi^* \cE_0 \to \cE$ is a surjection of vector bundles of the same rank and hence an isomorphism.

In particular, suppose that $G$ is a diagonalizable group scheme.  Then using the notation of Remark \ref{R:explicit}, every irreducible $G$-representation $\rho \in \Gamma$ is one-dimensional so that a $G$-action on $A$  corresponds to a $\Gamma$-grading $A = \bigoplus_{\rho \in \Gamma} A^{(\rho)}$, and an $A$-module with a $G$-action corresponds to a $\Gamma$-graded $A$-module.  
Therefore, if $A = \bigoplus_{\rho \in \Gamma} A^{(\rho)}$ is a $\Gamma$-graded noetherian $k$-algebra with $A^{(0)}$ a complete local $k$-algebra, then every finitely generated projective $\Gamma$-graded $A$-module is free.  When $G = \GG_m$ and $A^G=k$, this is the well known statement (e.g., \cite[Thm.~19.2]{eisenbud}) that every finitely generated projective graded module over a Noetherian graded $k$-algebra $A = \bigoplus_{d \ge 0} A_d$ with $A_0 = k$ is free.
\end{remark}

\subsection{Tannaka duality}
The following Tannaka duality theorem proved by the second and third author is crucial in our argument.

\begin{theorem}\cite[Thm.~1.1]{hallj_dary_coherent_tannakian_duality} \label{T:tannakian}
Let $\cX$ be an excellent stack and let $\cY$ be a noetherian algebraic stack with affine stabilizers.  Then the natural functor
$$\Hom(\cX, \cY) \to \Hom_{r\otimes, \simeq}\bigl(\Coh(\cY), \Coh(\cX)\bigr)$$
 is an equivalence of categories, where $\Hom_{r\otimes, \simeq}(\Coh(\cY), \Coh(\cX))$ denotes the category whose objects are right exact monoidal functors $\Coh(\cY) \to \Coh(\cX)$ and morphisms are natural isomorphisms of functors. 
\end{theorem}

We will apply the following consequence of Tannaka duality:

\begin{corollary} \label{C:tannakian}
Let $\cX$ be an excellent algebraic stack with affine stabilizers and let $\cZ \subseteq \cX$ be a closed substack.  Suppose that $\cX$ is coherently complete along $\cZ$. 
If $\cY$ is a noetherian algebraic stack with affine stabilizers, then the natural functor
$$\Hom(\cX, \cY) \to \ilim_n \Hom\bigl(\cX_{\cZ}^{[n]}, \cY\bigr)$$
is an equivalence of categories.
\end{corollary}

\begin{proof}
There are natural equivalences
\begin{align*}
\Hom(\cX, \cY) 	& \simeq \Hom_{r\otimes, \simeq}\bigl( \Coh(\cY), \Coh(\cX)\bigr)  & & \text{(Tannaka duality)}\\
				& \simeq \Hom_{r\otimes, \simeq}\bigl( \Coh(\cY), \ilim \Coh\bigl(\cX_{\cZ}^{[n]}\bigr) \bigr)  & & \text{(coherent completeness)}\\
				& \simeq \ilim \Hom_{r\otimes, \simeq}\bigl( \Coh(\cY), \Coh\bigl(\cX_{\cZ}^{[n]}\bigr) \bigr)  & & \\
				& \simeq \ilim \Hom\bigl(\cX_{\cZ}^{[n]}, \cY\bigr) & & \text{(Tannaka duality)}.\qedhere
\end{align*}
\end{proof}

 \section{Proofs of Theorems \ref{T:field} and \ref{T:smooth}} \label{S:proof-section}
 
 \subsection{The normal and tangent space of an algebraic stack} \label{S:tangent}
Let $\cX$ be a quasi-separated algebraic stack, locally of finite type over a field $k$, with affine stabilizers. Let $x \in \cX(k)$ be a closed point.  Denote by $i \co BG_x \to \cX$ the closed immersion of the residual gerbe of $x$, and by $\cI$ the corresponding ideal sheaf.  The {\it normal space to $x$} is $N_x := (\cI/\cI^2)^{\vee} = (i^* \cI)^{\vee}$ viewed as a $G_x$-representation. The {\it tangent space $T_{\cX,x}$ to $\cX$ at $x$} is the $k$-vector space of equivalence classes of pairs $(\tau, \alpha)$ consisting of morphisms $\tau \co \Spec k[\epsilon]/(\epsilon^2) \to \cX$ and 2-isomorphisms $\alpha \co x \to \tau|_{\Spec k}$.  The stabilizer $G_x$ acts linearly on the tangent space $T_{\cX,x}$ by precomposition on the 2-isomorphism.  
If $G_x$ is smooth, then there is an identification $T_{\cX,x} \cong N_x$ of $G_x$-representations.  Moreover, if $\cX = [X/G]$ is a quotient stack where $G$ is a smooth group scheme and $x \in X(k)$ (with $G_x$  not necessarily smooth), then $N_x$ is identified with the normal space  $T_{X,x} / T_{G \cdot x, x}$ to the orbit $G \cdot x$ at $x$.

 \subsection{The smooth case} \label{S:smooth}
 We now prove Theorem \ref{T:smooth} even though it follows directly from Theorem \ref{T:field} coupled with Luna's fundamental lemma \cite[p.~94]{luna}.
We feel that since the proof of Theorem \ref{T:smooth} is more transparent and less technical than Theorem \ref{T:field}, digesting the proof first in this case will make the proof of Theorem \ref{T:field} more accessible.

\begin{proof}[Proof of Theorem \ref{T:smooth}]
We may replace $\cX$ with an open neighborhood of $x$ and thus assume that
$\cX$ is noetherian. Define the quotient stack $\cN= [N_x/G_x]$, where $N_x$ is viewed as an affine scheme via $\Spec(\Sym N_x^{\vee})$.

Since $G_x$ is linearly reductive, we claim that there are compatible
isomorphisms $\cX_x^{[n]} \cong \cN^{[n]}$. 
To see this, first note that we can lift $\cX_x^{[0]}=BG_x$ to a unique morphism
$t_n\colon \cX_x^{[n]}\to BG_x$ for all $n$. Indeed, the obstruction to a lift from
$t_n\colon \cX_x^{[n]}\to BG_x$ to $t_{n+1}\colon \cX_x^{[n+1]}\to BG_x$ is an
element of the group $\Ext_{\oh_{BG_x}}^{1}(L_{BG_x/k}, \cI^{n+1}/\cI^{n+2})$
\cite{olsson-defn}, which is zero because $BG_x$ is cohomologically affine and
 $L_{BG_x/k}$ is a perfect complex supported in degrees $[0,1]$ as $BG_x\to
\Spec k$ is smooth.

In particular, $BG_x=\cX_x^{[0]}\inj \cX_x^{[1]}$ has a retraction. This implies
that $\cX_x^{[1]} \cong \cN^{[1]}$ since both are trivial deformations by the
same module. Since $\cN\to BG_x$ is smooth, the obstructions to lifting the
morphism $\cX_x^{[1]}\cong \cN^{[1]}\inj \cN$ to $\cX_x^{[n]}\to \cN$ for every $n$
vanish as
$H^1(BG_x,\Omega_{\cN/BG_x}^\vee\otimes \cI^{n+1}/\cI^{n+2})=0$.  We have induced
isomorphisms $\cX_x^{[n]} \cong \cN^{[n]}$ by
Proposition~\ref{P:closed/iso-cond:infinitesimal}.

Let $\cN  \to  N = N_x \gitq G_x$ be the good moduli space and denote by $0 \in N$ the image of the origin.     Set $\hat{\cN} := \Spec \hat{\oh}_{N,0} \times_N \cN$.  Since $\hat{\cN}$ is coherently complete (Theorem \ref{key-theorem}), we may apply Tannaka duality (Corollary \ref{C:tannakian}) to find a morphism $\hat{\cN} \to \cX$ filling in the diagram
\vspace{.2cm}
$$
\xymatrix{
\cX_x^{[n]} \cong \cN^{[n]} \ar[r] \ar@/^1.6pc/[rrr]	& \hat{\cN} \ar[r] \ar[d]	\ar@/^1pc/@{-->}[rr]					& \cN \ar[d] & \cX\\
& \Spec \hat{\oh}_{N,0} \ar[r]	\ar@{}[ur]|\square				& N.
}
$$

Let us now consider the functor $F \co (\Sch/N)^\op \to \Sets$ which assigns to a
morphism $S \to N$ the set of morphisms $S \times_N \cN \to \cX$ modulo
2-isomorphisms.  This functor is locally of finite presentation and we have an
element of $F$ over $\Spec \hat{\oh}_{N,0}$.  By Artin approximation
\cite[Cor.~2.2]{artin-approx} (cf.\ Theorem~\ref{T:artin-approximation}), there
exist an \'etale morphism $(U,u) \to
(N,0)$ where $U$ is an affine scheme and a morphism $f\colon (\cW,w):=(U
\times_N \cN, (u,0) ) \to (\cX,x)$ agreeing with $(\hat{\cN},0) \to (\cX,x)$ to
first order.  Since $\cX$ is smooth at $x$, it follows by
Proposition~\ref{P:closed/iso-cond:infinitesimal} that $f$ restricts to isomorphisms
$f^{[n]}\colon \cW_w^{[n]}\to \cX_x^{[n]}$ for every $n$, hence that $f$ is \'etale at $w$.
This establishes the theorem after shrinking $U$ suitably; the final statement
follows from Proposition \ref{P:refinement}
below.
\end{proof}

\subsection{The general case} \label{S:proof}

We now prove Theorem \ref{T:field} by a similar method to the proof in the smooth case but using equivariant Artin algebraization (Corollary \ref{C:equivariant-algebraization}) in place of Artin approximation.

\begin{proof}[Proof of Theorem \ref{T:field}]
We may replace $\cX$ with an open neighborhood of $x$ and thus assume that
$\cX$ is noetherian and that $x \in |\cX|$ is a closed point.

Let $\cN := [N_x / H]$ and let $N$ be the GIT quotient $N_x \gitq H$; then the induced morphism $\cN \to N$ is a good moduli space. Further, let $\hat{\cN} := \Spec \hat{\oh}_{N,0} \times_N \cN$, where $0$ denotes the image of the origin. 
Let $\eta_0 \co BH \to BG_x = \cX_x^{[0]}$ be the morphism induced from the inclusion $H \subseteq G_x$; this is a smooth (resp.\ \'etale) morphism.  We first prove by induction that there are compatible $2$-cartesian diagrams
\[
\xymatrix{\cH_n \ar[d]_{\eta_n} \ar@{(->}[r] & \cH_{n+1} \ar[d]^{\eta_{n+1}} \\
\cX_x^{[n]} \ar@{(->}[r] \ar@{}[ur]|\square& \cX_x^{[n+1]},}
\]
where $\cH_0 = BH$ and the vertical maps are
smooth (resp.\ \'etale).  Indeed, given $\eta_n \co \cH_n \to \cX_x^{[n]}$, by \cite{olsson-defn}, the obstruction to
the existence of $\eta_{n+1}$ is an element of $
\Ext^2_{\oh_{BH}}(\Omega_{BH/BG_x},\eta_0^*(\cI^{n+1}/\cI^{n+2}))$, but this group vanishes as $H$ is linearly reductive and $\Omega_{BH/BG_x}$ is a vector bundle.

Let $\tau_0 \co \cH_0 = BH \hookrightarrow \cN$ be the inclusion of the origin. Since $H$ is linearly reductive, the
deformation $\cH_0\inj \cH_1$ is a trivial extension with ideal $N_x^\vee$ and hence we
have
an isomorphism $\tau_1\colon \cH_1\cong \cN^{[1]}$ (see proof of smooth
case). Using linear reductivity of $H$ once again and deformation theory, we
obtain compatible morphisms $\tau_n \colon \cH_n \to \cN$ extending $\tau_0$
and $\tau_1$. These are closed immersions by
Proposition~\ref{P:closed/iso-cond:infinitesimal}~\itemref{PI:closed:infinitesimal}.

The closed immersion $\tau_n\colon \cH_n \inj \cN$ factors through a
closed immersion $i_n\colon \cH_n \inj \cN^{[n]}$. 
Since $\hat{\cN}$ is coherently complete, the inverse system of epimorphisms $\oh_{\cN^{[n]}} \to i_{n,*}\oh_{\cH_n}$, in the category $\varprojlim_n \Coh(\cN^{[n]})$, lifts uniquely to an epimorphism $\oh_{\hat{\cN}} \to \oh_{\hat{\cH}}$ in the category $\Coh(\hat{\cN})$. This defines a closed immersion $i\colon \hat{\cH} \to \hat{\cN}$ rendering the following square $2$-cartesian for all $n$:
\[
\xymatrix{
\cH_n \ar@{(->}[d]_{i_n} \ar@{(->}[r]	& \hat{\cH} \ar@{(->}[d]^i\\
\cN^{[n]} \ar@{(->}[r] 		& \hat{\cN}.
}
\]
Since $\hat{\cH}$ also is coherently complete, Tannaka duality (Corollary \ref{C:tannakian}) yields a morphism $\eta\co \hat{\cH} \to \cX$ such that the following square is
$2$-commutative for all $n$:
\[
\xymatrix{
\cH_n \ar[d]_{\eta_n} \ar@{(->}[r]	& \hat{\cH} \ar[d]^{\eta} \\
\cX_x^{[n]} \ar[r]					& \cX.
}
\]
The morphism $\eta$ is formally versal (resp.\ universal) by Proposition \ref{P:formal-versality-criterion}. We may therefore apply Corollary \ref{C:equivariant-algebraization} to obtain a stack $\cW=[\Spec A/H]$ together with a closed point $w\in |\cW|$, a morphism $f\co (\cW,w)\to (\cX,x)$
of finite type, a flat morphism $\varphi\co \hat{\cH}\to \cW$, identifying $\hat{\cH}$ with the
completion of $\cW$ at $w$, and a $2$-isomorphism $f\circ\varphi\cong \eta$. In particular, $f$ is smooth (resp.\ \'etale)
at $w$. Moreover, $(f\circ \varphi)^{-1}(\cX_x^{[0]})=\cH_0$ so we have a
flat morphism $BH=\cH_0\to f^{-1}(BG_x)$ which equals the inclusion of the
residual gerbe at $w$. It follows that $w$ is an isolated point in the fiber
$f^{-1}(BG_x)$. We can thus replace $\cW$ with an open neighborhood of $w$
such that $\cW\to \cX$ becomes smooth (resp.\ \'etale) and
$f^{-1}(BG_x) = BH$. Since $w$ is a closed point of $\cW$, we may further shrink
$\cW$ so that it remains cohomologically affine (see Lemma~\ref{L:shrink} below).

The final statement follows from Proposition \ref{P:refinement} below.
\end{proof}

 \subsection{The refinement}  The results in this section can be used to show
  that under suitable hypotheses, the quotient presentation $f\colon \cW \to \cX$ in Theorems \ref{T:field} and \ref{T:smooth}
can be arranged to be affine (Proposition \ref{P:refinement}), quasi-affine (Corollary \ref{C:refinement:quot_stack}), and representable (Proposition \ref{P:rep_diagonalizable}).

 The following trivial lemma will be frequently applied to a good moduli space
 morphism $\pi \colon \cX \to X$. Note that any closed subset $\cZ\subseteq
 \cX$ satisfies the assumption in the lemma in this case.
 \begin{lemma}\label{L:shrink}
   Let $\pi\colon \cX \to X$ be a closed morphism of topological spaces and
   let $\cZ\subseteq \cX$ be a closed subset. Assume that every open
   neighborhood of $\cZ$ contains $\pi^{-1}(\pi(\cZ))$. If $\cZ \subseteq \cU$
   is an open neighborhood of $\cZ$, then there exists an open neighborhood
   $U' \subseteq X$ of $\pi(\cZ)$ such that $\pi^{-1}(U') \subseteq \cU$.
 \end{lemma} 
\begin{proof}
Take $U'=X\setminus \pi(\cX\setminus \cU)$.
\end{proof}
We now come to our main refinement result.
\begin{proposition} \label{P:refinement}
 Let $f \co \cW \to \cX$ be a morphism of noetherian algebraic stacks such that $\cW$ is cohomologically affine with affine diagonal. Suppose $w \in |\cW|$ is a closed point such that $f$ induces an injection of stabilizer groups at $w$.  
 If $\cX$ has affine diagonal, then there exists a cohomologically affine open neighborhood $\cU \subseteq \cW$ of $w$ such that $f|_{\cU}$ is affine.
\end{proposition}
\begin{proof}
  By shrinking $\cW$, we may assume that $\Delta_{\cW/\cX}$ is
  quasi-finite and after further shrinking, we may arrange so that
  $\cW$ remains cohomologically affine (Lemma~\ref{L:shrink}). Let
  $p\colon V \to \stX$ be a smooth surjection, where $V$ is affine;
  then $p$ is affine because $\stX$ has affine diagonal. Take
  $f_V \colon \cW_V \to V$ to be the pullback of $f$ along
  $p$. Then $\cW_V \to \cW$ is affine, and so $\cW_V$ is
  cohomologically affine.  Since $\cW_V$ also has quasi-finite and
  affine diagonal, $f_V$ is separated
  \cite[Thm.~8.3.2]{alper-adequate}. By descent, $f$ is separated. In
  particular, the relative inertia of $f$,
  $I_{\cW/\cX} \to \cW$, is finite. By Nakayama's Lemma, there
  is an open substack $\cU$ of $\cW$, containing $w$, with trivial
  inertia relative to $\cX$. Thus $\cU\to \cX$ is quasi-compact,
  representable and separated. Shrinking $\cU$ further, $\cU$
  also becomes cohomologically affine. Since
  $\cX$ has affine diagonal, it follows that $f$ is also
  cohomologically affine. By Serre's Criterion
  \cite[Prop.~3.3]{alper-good}, $f$ is affine.
\end{proof}
\begin{corollary} \label{C:refinement:quot_stack} 
  Let $S$ be a noetherian scheme.  Let
  $f \co \cW \to \cX$ be a morphism of noetherian algebraic stacks over $S$
  such that $\cW$ is cohomologically affine with affine
  diagonal. Suppose $w \in |\cW|$ is a closed point such that $f$
  induces an injection of stabilizer groups at $w$.  If $\cX=[X/G]$
  where $X$ is an algebraic space and $G$ is an affine flat group
  scheme of finite type over $S$, then there exists a cohomologically affine open neighborhood
  $\cU \subseteq \cW$ of $w$ such that $f|_{\cU}$ is quasi-affine.
\end{corollary}
\begin{proof}
  Consider the composition $\cW \to [X/G] \to BG$. By Proposition
  \ref{P:refinement}, we may suitably shrink $\cW$ so that the
  composition $\cW \to [X/G] \to BG$ becomes affine. Since $X$ is a
  noetherian algebraic space, it has quasi-affine diagonal; in
  particular $[X/G] \to BG$ has quasi-affine diagonal. It follows
  immediately that $\cW \to [X/G]$ is quasi-affine.
\end{proof}
\begin{proposition}\label{P:rep_diagonalizable}
  Let $S$ be a noetherian scheme. Let $f\colon \cW \to \cX$ be a morphism of 
  locally noetherian algebraic stacks over $S$. Assume that $\cX$ has separated
  diagonal and that $\cW = [W/H]$, where $W$ is affine over $S$ and $H$ is 
  of multiplicative type over $S$. If $w\in W$ is fixed by $H$ and $f$ induces an injection 
  of stabilizer groups at $w$, then there exists an $H$-invariant affine open $U$ of $w$ in 
  $W$ such that $[U/H] \to \cX$ is representable.
\end{proposition}

\begin{remark}
The separatedness of the diagonal is essential; see Examples \ref{ex5} and \ref{ex4}.
\end{remark}
\begin{proof}
  There is an exact sequence of groups over $\cW$:
  \[
  \xymatrix{1 \ar[r] & I_{\cW/\cX} \ar[r] & I_{\cW/S} \ar[r] & I_{\cX/S}\times_{\cX} \cW.}
  \]
  Since $f$ induces an injection of stabilizer groups at $w$, it follows that
  $(I_{\cW/\cX})_w$ is trivial. Also, since $I_{\cX/S} \to \cX$ is separated,
  $I_{\cW/\cX} \to I_{\cW/S}$ is a closed immersion.

  Let $I=I_{\cW/\cX}\times_{\cW} W$ and pull the inclusion 
  $I_{\cW/\cX} \to I_{\cW/S}$ back to $W$.  Since $I_{\cW/S}\times_{\cW} W \to H\times_S W$ 
  is a closed immersion, it follows that $I \to H\times_S W$ is a closed immersion. Since $H\to S$ is of multiplicative type and $I_w$ is trivial,
  it follows that $I\to W$ is trivial in a neighborhood of $w$~\cite[Exp.~IX, Cor.~6.5]{sga3ii}. By shrinking 
  this open subset using Lemma \ref{L:shrink}, we obtain the result.  
\end{proof}

\section{Local applications} \label{S:local_applications}
\subsection{Generalization of Sumihiro's theorem on torus actions} \label{A:sumihiro}
In \cite[\S 2]{oprea}, Oprea speculates that every quasi-compact
Deligne--Mumford stack $\cX$ with a torus action has an equivariant \'etale
atlas $\Spec A \to \cX$. He proves this when
$\cX=\overline{\cM}_{0,n}(\PP^r,d)$ is the moduli space of stable maps and the action is induced by any
action of $\GG_m$ on $\PP^r$ and obtains some nice applications. We show that Oprea's speculation holds in general.

For group actions on stacks, we follow the conventions of \cite{romagny}.  
Let $T$ be a torus acting on an algebraic stack $\cX$, locally of finite type over a field $k$, via $\sigma \co T \times \cX \to \cX$.
 Let $\cY = [\cX /T]$.  Let $x \in \cX(k)$ be a point with image $y \in \cY(k)$.  There is an exact sequence
\begin{equation} \label{E:stab}
\xymatrix{1 \ar[r] & G_x \ar[r] &  G_y \ar[r] &  T_x \ar[r] & 1},
\end{equation}
where the stabilizer $T_x \subseteq T$ is defined by the fiber product
\begin{equation} \label{D:stab}
\begin{split} 
\xymatrix{
T_x\times B G_x \ar[r]^-{\sigma_x} \ar[d]				& B G_x \ar[d] \\
T\times B G_x \ar[r]^-{\sigma|_x}	\ar@{}[ur]|\square		& \cX
} \end{split}
\end{equation}
and $\sigma|_x \co T\times B G_x \xrightarrow{\id\times\iota_x} T \times \cX \xrightarrow{\sigma} \cX$.

 Observe that $G_y= \Spec k \times_{BG_x} T_x$.  The exact sequence \eqref{E:stab} is trivially split if and only if the induced action $\sigma_x$ of $T_x$ on $BG_x$ is trivial. The sequence is split if and only if the action $\sigma_x$ comes from a group homomorphism $T \to \Aut(G_x)$.
\begin{theorem} \label{T:sumi1}
Let $\cX$ be a quasi-separated algebraic (resp.\ Deligne--Mumford) stack with affine stabilizers, locally of finite type over an algebraically closed field $k$.  Let $T$ be a torus acting on $\cX$.  Let $x \in \cX(k)$ be a point such that $G_x$ is smooth and the exact sequence \eqref{E:stab} is split (e.g., $\cX$ is an algebraic space). There exists a $T$-equivariant smooth (resp.\ \'etale) neighborhood $(\Spec A,u) \to (\cX,x)$ that induces an isomorphism of stabilizers at $u$. 
\end{theorem}

\begin{proof}
Let $\cY = [\cX/T]$ and $y \in \cY(k)$ be the image of $x$.  
As the sequence \eqref{E:stab} splits, we can consider $T_x$ as a subgroup of $G_y$.  
By applying Theorem \ref{T:field} to $\cY$ at $y$ with respect to the subgroup $T_x \subseteq G_y$, we obtain a smooth (resp.\ \'etale) morphism $f \co [W/T_x] \to \cY$,
where $W$ is an affine scheme with an action of $T_x$, which induces the given inclusion $T_x \subseteq G_y$ at stabilizer groups at a preimage $w \in [W/T_x]$ of $y$. Consider the cartesian diagram
$$\xymatrix{
[W/T_x] \times_{\cY} \cX \ar[r] \ar[d]		& \cX \ar[d]\ar[r] & \Spec k\ar[d] \\
[W/T_x] \ar[r]						& \cY\ar[r] & BT
}$$
The map $[W/T_x]\to \cY\to BT$ induces the injection $T_x\inj T$ on stabilizers
groups at $w$. Thus, by
Proposition~\ref{P:refinement},
there is an
open neighborhood $\cU\subseteq [W/T_x]$ of $w$ such that $\cU$ is
cohomologically affine and $\cU\to BT$ is affine. The fiber product
$\cX\times_\cY \cU$ is thus an affine scheme $\Spec A$ and the induced map
$\Spec A\to \cX$ is $T$-equivariant. If $u\in \Spec A$ is a closed point above
$w$ and $x$, then the map $\Spec A\to \cX$ induces an isomorphism $T_x\to T_x$
of stabilizer groups at $u$.
\end{proof}

In the case that $\cX$ is a normal scheme, Theorem \ref{T:sumi1} was proved by
Sumihiro~\cite[Cor.~2]{sumihiro}, \cite[Cor.~3.11]{sumihiro2}; then $\Spec A
\to \cX$ can be taken to be an open neighborhood.  The nodal cubic with a
$\GG_m$-action provides an example where an \'etale neighborhood is needed:
there does not exist a
$\GG_m$-invariant affine open neighborhood of the node. Theorem~\ref{T:sumi1} was also known if
$\cX$ is a quasi-projective scheme \cite[Thm.~1.1(iii)]{brion-linearization} or
if $\cX$ is a smooth, proper, tame and irreducible Deligne--Mumford stack,
whose generic stabilizer is trivial and whose coarse moduli space is a scheme
\cite[Prop.~3.2]{skowera}.

\begin{remark} \label{R:not-split}
The theorem above fails when \eqref{E:stab} does not split because an equivariant, stabilizer-preserving, affine, and \'etale neighborhood induces a splitting. For a simple example when \eqref{E:stab} does not split, consider the Kummer exact sequence $1 \to \Gmu_n \to \GG_m \xrightarrow{n}  \GG_m \to 1$ for some invertible $n$. This gives rise to a $T=\GG_m$ action on the Deligne--Mumford stack $\cX=B\Gmu_n$ with stack quotient $\cY=[\cX/T]=B\GG_m$ such that  \eqref{E:stab} becomes the Kummer sequence and hence does not split. The action of $\GG_m$ on $B\Gmu_n$ has the following explicit description:  for $t \in \GG_m(S) = \Gamma(S, \oh_S)^\times$ and $(\cL, \alpha) \in B\Gmu_n(S)$ (where $\cL$ is a line bundle on $S$ and $\alpha \co \cL^{\tensor n} \to \oh_S$ is an isomorphism), then $t \cdot (\cL, \alpha) = (\cL, t \circ \alpha)$.

There is nevertheless an \'etale presentation $\Spec k \to B \Gmu_n$ which is equivariant under $\GG_m \xrightarrow{n} \GG_m$.  The following theorem shows that such \'etale presentations exist more generally.\end{remark}

\begin{theorem} \label{T:sumi2}
Let $\cX$ be a quasi-separated Deligne--Mumford stack, locally of finite type over an algebraically closed field $k$. Let $T$ be a torus acting on $\cX$.  If $x\in \cX(k)$, then there exist a reparameterization $\alpha \co T \to T$ and an \'etale neighborhood $(\Spec A, u) \to (\cX,x)$ that is equivariant with respect to $\alpha$.
\end{theorem}

\begin{proof}
In the exact sequence \eqref{E:stab}, $G_x$ is \'etale and $T_x$ is diagonalizable. This implies that
$(G_y)^0$ is diagonalizable. Indeed, first note that we have exact sequences:
\[
\xymatrix{1 \ar[r] & G_x\cap (G_y)^0 \ar[r] & (G_y)^0 \ar[r] & (T_x)^0 \ar[r] & 1\phantom{.}\\
1 \ar[r] & G_x\cap (G_y)^0 \ar[r] & (G_y)^0_\red \ar[r] & (T_x)^0_\red \ar[r] & 1.}
\]
The second sequence shows that $(G_y)^0_\red$ is a torus 
(as it is connected, reduced and surjects onto a torus with finite kernel) and, consequently,
that $G_x\cap (G_y)^0$ is diagonalizable. It then follows that $(G_y)^0$ is
diagonalizable from the first sequence~\cite[Exp.~XVII, Prop.~7.1.1~b)]{sga3ii}.

Theorem \ref{T:field} produces an \'etale neighborhood $f\colon ([\Spec
  A/(G_y)^0],w) \to (\cY,y)$ such that the induced morphism on stabilizers
groups is $(G_y)^0 \to G_y$. Replacing $\cX\to \cY$ with the pull-back along $f$,
we may thus assume that $G_y$ is connected and diagonalizable.

\newcommand{\tor}{\mathrm{tor}}%
If we let $G_y=D(N)$, $T_x=D(M)$ and $T=D(\ZZ^r)$, then we have a surjective
map $q\colon \ZZ^r\to M$ and an injective map $\varphi\colon M\to N$. The
quotient $N/M$ is torsion but without $p$-torsion, where $p$ is the
characteristic of $k$. Since all torsion of $M$ and $N$ is $p$-torsion, we
have that $\varphi$ induces an isomorphism of torsion subgroups. We can thus
find splittings of $\varphi$ and $q$ as in the diagram
\[
\xymatrix@C+10mm{
\mathllap{\ZZ^r=}\ZZ^s\oplus M/M_\tor\ar@{(->}[r]^{\alpha=\id\oplus\varphi_2}\ar@{->>}[d]^{q=q_1\oplus\id}
  & \ZZ^s\oplus N/N_\tor\mathrlap{=\ZZ^r}\ar@{->>}[d]^{q'=\varphi_1 q_1\oplus\id} \\
\mathllap{M=}M_\tor\oplus M/M_\tor\ar@{(->}[r]^{\varphi=\varphi_1\oplus \varphi_2}
  & N_\tor\oplus N/N_\tor\mathrlap{=N.}
}
\]
The map $q'$ corresponds to an embedding $G_y\inj T$ and the map
$\alpha$ to a reparameterization $T\to T$. After reparameterizing the
action of $T$ on $\cX$ via $\alpha$, the surjection $G_y\surj T_x$
becomes split. The result now follows from Theorem \ref{T:sumi1}.
\end{proof}

 We can also prove:

\begin{theorem} \label{T:sumi3}
Let $X$ be a quasi-separated algebraic space, locally of finite type over an algebraically closed field $k$. Let $G$ be an affine group scheme of finite type over $k$ acting on $X$.  Let $x \in X(k)$ be a point with linearly reductive stabilizer $G_x$.  Then there exists an affine scheme $W$ with an action of $G$ and a $G$-equivariant \'etale neighborhood $(W,w) \to (X,x)$ that induces an isomorphism of stabilizers at $w$.
\end{theorem}

\begin{proof}
By Theorem \ref{T:field}, there exists an \'etale neighborhood $f\colon (\cW,w) \to 
([X/G],x)$ such that $\cW$ is cohomologically affine, $f$ induces an isomorphism of 
stabilizers at $w$, and $w$ is a closed point.  By
Proposition~\ref{P:refinement}, we can assume after shrinking $\cW$ that 
the composition $\cW \to [X/G] \to BG$ is affine.  It follows that $W = \cW \times_{[X/G]} X$ is affine and that 
$W \to X$ is a $G$-equivariant \'etale neighborhood of $x$. If we also let
$w\in W$ denote the unique preimage of $x$, then $G_w=G_x$.
\end{proof}

Theorem \ref{T:sumi3} is a partial generalization of another result of
Sumihiro~\cite[Lem.~8]{sumihiro}, \cite[Thm.~3.8]{sumihiro2}. He proves
the existence of an open $G$-equivariant covering by quasi-projective
subschemes when $X$ is a normal scheme and $G$ is connected.

\subsection{Generalization of Luna's \'etale slice theorem} \label{A:luna}

We now provide a refinement of Theorem \ref{T:field} in the case that $\cX = [X/G]$ is a quotient stack, generalizing Luna's \'etale slice theorem.

\begin{theorem}\label{T:luna}
Let $X$ be a quasi-separated algebraic space, locally of finite type over an algebraically closed field $k$. Let $G$ be an affine smooth group scheme acting on $X$.  Let $x \in X(k)$ be a point with linearly reductive stabilizer $G_x$.   Then there exists an affine scheme $W$ with an action of $G_x$ which fixes a point $w$, and an unramified $G_x$-equivariant morphism $(W,w) \to (X,x)$ such that $\tilde{f} \co W \times^{G_x} G \to X$ is \'etale.\footnote{Here, $W \times^{G_x} G$ denotes the quotient $(W \times G) / G_x$.  Note that there is an identification of GIT quotients  $(W \times^{G_x} G) \gitq G \cong W \gitq G_x$.}

 If $X$ admits a good GIT quotient $X \to X \gitq G$, then it is possible to arrange that the induced morphism $W \gitq G_x \to X \gitq G$ is \'etale and $W \times^{G_x} G \cong W \gitq G_x \times_{X \gitq G} X$.
 
Let $N_x = T_{X,x} / T_{G \cdot x, x}$ be the normal space to the orbit at $x$; this inherits a natural linear action of $G_x$.  If $x \in X$ is smooth, then it can be arranged that there is an \'etale $G_x$-equivariant morphism $W \to N_x$ such that $W \gitq G_x \to N_x \gitq G_x$ is \'etale and 
 $$\xymatrix{
N_x \times^{G_x} G \ar[d] &  W \times^{G_x} G\ar[r]^-{\tilde f} \ar[d] \ar[l]	& X \\
N_x \gitq G_x & W \gitq G_x \ar[l]	\ar@{}[ul]|\square				& 
}$$
is cartesian.
\end{theorem} 
\begin{proof}
By applying Theorem \ref{T:sumi3}, we can find an affine scheme $X'$ with an action of $G$ and a $G$-equivariant, \'etale morphism $X' \to X$.  This reduces the theorem to the case when $X$ is affine, which was established in \cite[p.~97]{luna}, cf.\ Remark~\ref{R:luna}.
\end{proof}

\begin{remark}\label{R:luna}
The theorem above follows from Luna's \'etale slice theorem \cite{luna} if $X$ is affine. In this case, Luna's \'etale slice theorem is stronger than Theorem \ref{T:luna} as it asserts additionally that $W \to X$ can be arranged to be a locally closed immersion (which is obtained by choosing a $G_x$-equivariant section of $T_{X,x} \to N_x$ and then restricting to an open subscheme of the inverse image of $N_x$ under a $G_x$-equivariant \'etale morphism $X \to T_{X,x}$).
Note that while \cite{luna} assumes that $\mathrm{char}(k) = 0$ and $G$ is reductive, the argument goes through unchanged in arbitrary characteristic if $G$ is smooth, and $G_x$ is smooth and linearly reductive.  Moreover, with minor modifications, the argument in \cite{luna} is also valid if $G_x$ is not necessarily smooth.
\end{remark}

\begin{remark}\label{R:luna-alper-kresch}
More generally, if $X$ is a normal scheme, it is shown in \cite[\S 2.1]{alper-kresch} that $W \to X$ can be arranged to be a locally closed immersion.  However, when $X$ is not normal or is not a scheme, one cannot always arrange $W \to X$ to be a locally closed immersion and therefore we must allow unramified ``slices" in the theorem above.
\end{remark}

\subsection{Existence of equivariant versal deformations for curves} \label{A:mv_curve}

 By a \emph{curve}, we mean a proper scheme over $k$ of pure dimension one.
 An $n$-pointed curve is a curve $C$ together with $n$ points $p_1,\dots,p_n\in
 C(k)$. The points are not required to be smooth nor distinct. We introduce
 the following condition on $(C,\{p_j\})$:

 \begin{enumerate}[label=($\dagger$),ref=$\dagger$]
   \item\label{cond:g1-marked} every connected component of $C$ of arithmetic
     genus $1$ contains a point $p_j$.
 \end{enumerate}
 
 \begin{theorem} \label{T:curves}
 Let $k$ be an algebraically closed field and let $(C,\{p_j\})$ be an
 $n$-pointed reduced curve over $k$ satisfying \itemref{cond:g1-marked}.
 Suppose that a
 linearly reductive group scheme $H$ acts on $C$.
 If $\Aut(C,\{p_j\})$ is smooth, then there exist an affine scheme $W$ of finite type
 over $k$ with an action of $H$ fixing a point $w \in W$
 and a miniversal deformation
 $$\xymatrix{
 \cC \ar[d]	 				&  C \ar[l] \ar[d]\\
 W\ar@/^/[u]^{s_j} 			&  \Spec k\ar@/_/[u]_{p_j} \ar[l]_{w}\ar@{}[ul]|\square
 }$$
 of $C \cong \cC_w$ such that there exists an 
 action of $H$ on the total family $(\cC,\{s_j\})$ compatible with the actions of $H$ on $W$
 and $C$.
 \end{theorem}
 
 The theorem above was proven for Deligne--Mumford semistable curves in \cite{alper-kresch}.

\begin{proof}
The stack $\mathcal{U}_n$ parameterizing all $n$-pointed proper schemes of
dimension $1$ is algebraic and quasi-separated~\cite[App.~B]{smyth_towards-a-classification-modular-compactifications}.
The substack $\mathfrak{M}_n\subset \mathcal{U}_n$ parameterizing reduced
$n$-pointed curves is open and the substack $\mathfrak{M}^\dagger_n\subset
\mathfrak{M}_n$, parameterizing reduced $n$-pointed curves satisfying
\itemref{cond:g1-marked} is open and closed.

We claim that $\mathfrak{M}^\dagger_n$ has affine stabilizers. To see this, let
$(C,\{p_j\})$ be an $n$-pointed curve satisfying \itemref{cond:g1-marked}.
The stabilizer of $(C,\{p_j\})$ is a closed subgroup $\Aut(C,
\{p_j\})\subseteq \Aut(\tilde{C}, Z)$ where $\eta\colon \tilde{C}\to C$ is the normalization, $Z=\eta^{-1}(\Sing C\cup
\{p_1,p_2,\dots,p_n\})$ with the reduced structure and $\Aut(\tilde{C},Z)$ denotes the automorphisms of $\tilde{C}$ that maps $Z$ onto $Z$. Since $\Aut(\pi_0(\tilde{C}))$ is finite, it is
enough to prove that $\Aut(\tilde{C}_i,Z\cap \tilde{C}_i)$ is affine for every component $\tilde{C}_i$ of $\tilde{C}$. This holds
since \itemref{cond:g1-marked} guarantees that either $g(\tilde{C}_i)\neq 1$ or
$Z\cap \tilde{C}_i\neq \emptyset$.

Since $H$ is linearly reductive and $\Aut(C,\{p_j\})/H$ is smooth, Theorem \ref{T:field} provides an affine scheme $W$ with an action of $H$, a $k$-point $w \in W$ fixed by $H$ and a smooth map $[W/H] \to \mathfrak{M}^\dagger_n$ with $w$ mapping to $(C, \{p_j\})$.  This yields a family of $n$-pointed curves $\cC \to W$ with an action of $H$ on $\cC$ compatible with the action on $W$ and $C \cong \cC_w$.  The map $W \to \mathfrak{M}^\dagger_n$ is smooth and adic at $w$. Indeed, it is flat by construction and the fiber at $(C,\{p_j\})$ is $\Spec k\to B\Aut(C,\{p_j\})$ which is smooth. In particular, the tangent space of $W$ at $w$ coincides with the tangent space of $\mathfrak{M}^\dagger_n$ at $(C,\{p_j\})$; that is, $\cC \to W$ is a miniversal deformation of $C$ (see Remark~\ref{R:miniversal}).
\end{proof}

\begin{remark}
From the proof, it is clear that Theorem \ref{T:curves} is valid for pointed curves such that every deformation has an affine automorphism group.  It was pointed out to us by Bjorn Poonen that if $(C, \{p_j\})$ is an $n$-pointed curve  and no connected component of $C_{\red}$ is a smooth unpointed curve of genus 1, then $\Aut(C, \{p_j\})$ is an affine group scheme over $k$.  It follows that Theorem \ref{T:curves} is valid for pointed curves $(C, \{p_j\})$ satisfying the property that for every deformation $(C', \{p'_j\})$ of $(C, \{p_j\})$, there is no connected component of $C'_{\red}$ which is a smooth unpointed curve of genus 1.
\end{remark}

In a previous version of this article, we erroneously claimed that if no
connected component of $C_{\red}$ is a smooth unpointed curve of genus 1, then
this also holds for every deformation of $C$. The following example shows that
this is not the case.

\begin{example}
Let $S=\AA^1_\CC=\Spec \CC[t]$, let $S'=\Spec \CC[t,x]/(x^2-t^2)$ and let
$\cC=E\times_\CC S'$ where $E$ is a smooth genus $1$ curve and $\cC\to S'\to S$
is the natural map. Then the fiber over $t=0$ is $E\times_\CC \Spec
\CC[x]/(x^2)$ and the fiber over $t=1$ is $E\amalg E$. Choosing a point
$p\in E$ and a section of $S'\to S$, e.g., $x=t$, gives a section $s$ of $\cC$
which only passes through one of the two components over $t=1$. In particular,
the fiber of $(\cC,s)$ over $t=1$ does not have affine automorphism
group. Alternatively, one could in addition pick a pointed curve $(C,c)$ of
genus at least $1$ and glue $\cC$ with $C\times S$ along $s$ and $c$. This
gives an unpointed counterexample.
\end{example}

\begin{remark}  If $\cC \to S$ is a family of pointed curves such that there is no connected component of any fiber whose reduction is a smooth unpointed curve of genus 1, then the automorphism group scheme $\Aut(\cC/S) \to S$ of $\cC$ over $S$ has affine fibers but need not be affine (or even quasi-affine).  This even fails for families of Deligne--Mumford semistable curves; see \cite[\S4.1]{alper-kresch}.
\end{remark}

\subsection{Good moduli spaces} \label{A:gms_app}
In the following result, we determine the \'etale-local structure of good moduli space morphisms. 
 \begin{theorem}  \label{T:consequences-gms} Let $\cX$ be a noetherian algebraic stack over an algebraically closed field $k$. Let $\cX \to X$ be a good moduli space with affine diagonal.
If $x\in \cX(k)$ is a
   closed point, then there exists an affine scheme $\Spec A$ with an action of $G_x$  and a cartesian diagram
   $$\xymatrix{
  	[\Spec A / G_x] \ar[r] \ar[d]				& \cX \ar[d]^{\pi} \\
	\Spec A \gitq G_x \ar[r]\ar@{}[ur]|\square	& X
  }$$ such that $\Spec A \gitq G_x \to X$ is an \'etale neighborhood of $\pi(x)$. 
\end{theorem}

In the proof of Theorem \ref{T:consequences-gms} we will use the following minor variation of \cite[Thm.~6.10]{alper-quotient}.  We provide a direct proof here for the convenience of the reader.
\begin{proposition}[Luna's fundamental lemma] \label{P:luna}
Let $f \co \cX \to \cY$ be an \'etale, separated and representable morphism of noetherian algebraic stacks such that there is a commutative diagram
$$\xymatrix{
\cX \ar[r]^f \ar[d]_{\pi_{\cX}}		& \cY \ar[d]^{\pi_{\cY}} \\
X \ar[r]						& Y
}$$
where $\pi_{\cX}$ and $\pi_{\cY}$ are good moduli spaces. Let $x \in |\cX|$ be a closed point.  If $f(x) \in |\cY|$ is closed and $f$ induces an isomorphism of stabilizer groups at $x$, then there exists an open neighborhood $U \subset X$ of $\pi_{\cX}(x)$ such that $U \to X \to Y$ is \'etale and $\pi_{\cX}^{-1}(U) \cong U \times_Y \cY$.
\end{proposition}

\begin{proof}
By Zariski's main theorem \cite[Cor.~16.4(ii)]{MR1771927}, there is a factorization $\cX \hookrightarrow \widetilde{\cX} \to \cY$, where $\cX \hookrightarrow \widetilde{\cX}$ is an open immersion and $\widetilde{\cX} \to \cY$ is finite. There is a good moduli space $\pi_{\widetilde{\cX}} \co \widetilde{\cX} \to \widetilde{X}$ such that the induced morphism $\widetilde{X}\to Y$ is finite \cite[Thm.\ 4.16(x)]{alper-good}. Note that $x\in |\widetilde{\cX}|$ is closed. By Lemma~\ref{L:shrink} we may thus replace $\cX$ with an open neighborhood $\cU \subset \cX$ of $x$ that is saturated with respect to $\pi_{\widetilde{\cX}}$ (i.e., $\cU =  \pi_{\widetilde{\cX}}^{-1}(\pi_{\widetilde{\cX}}(\cU))$). Then $X\to \widetilde{X}$ becomes an open immersion so that $X\to Y$ is quasi-finite.

Further, the question is \'etale local on $Y$. Hence, we may assume that $Y$ is the spectrum of a strictly henselian local ring with closed point $\pi_{\cY}(f(x))$.   Since $Y$ is henselian, after possibly shrinking $X$ further, we can arrange that $X \to Y$ is finite with $X$ the spectrum of a local ring with closed point $\pi_{\cX}(x)$. Then $X\to \widetilde{X}$ is a closed and open immersion, hence so is $\cX\to \widetilde{\cX}$. It follows that $\cX\to \cY$ is finite.
As $f$ is stabilizer-preserving at $x$ and $Y$ is strictly henselian, $f$ induces an isomorphism of residual gerbes at $x$.  We conclude that $f$ is a finite, \'etale morphism of degree $1$, hence an isomorphism.
\end{proof}

\begin{proof}[Proof of Theorem \ref{T:consequences-gms}]
By Theorem \ref{T:good-finite-type}, $\cX \to X$ is of finite type.  We may assume that $X=\Spec R$, where $R$ is a noetherian $k$-algebra. By
  noetherian approximation along $k \to R$, there is a finite type $k$-algebra $R_0$ and an 
  algebraic stack $\cX_0$ of finite type over $\Spec R_0$ with affine diagonal such that $\cX \simeq \cX_0 \times_{\Spec R_0} \Spec R$. We may 
  also arrange that the image $x_0$ of $x$ in $\cX_0$ is closed with linearly reductive 
  stabilizer $G_x$. We now apply Theorem \ref{T:field} to find a pointed affine \'etale 
  $k$-morphism $f_0 \colon ([\Spec A_0/G_x],w_0) \to (\cX_0,x_0)$ that induces an 
  isomorphism of stabilizers at~$w_0$. Pulling this back along $\Spec R \to \Spec R_0$, we 
  obtain an affine \'etale morphism $f \colon [\Spec A/G_x] \to \cX$ inducing an 
  isomorphism of stabilizers at all points lying over the preimage of $w_0$. The result now 
  follows from Luna's fundamental lemma for stacks (Proposition \ref{P:luna}).
\end{proof}

The following corollary answers negatively a question of Geraschenko--Zureick-Brown~\cite[Qstn.\ 32]{geraschenko-brown}:  does there exist an algebraic stack, with affine diagonal and good moduli space a field, that is not a quotient stack?  In the equicharacteristic setting, this result also settles a conjecture of theirs: formal GAGA holds for good moduli spaces with affine diagonal~\cite[Conj.\ 28]{geraschenko-brown}. The general case will be treated in forthcoming work \cite{ahr2}.
\begin{corollary}\label{C:gb-c28}
Let $\cX$ be a noetherian algebraic stack over a field $k$ (not assumed to be algebraically closed) with affine diagonal. Suppose that there exists a good moduli space $\pi \colon \cX \to \Spec R$, where $(R,\fm)$ is a complete local ring.
\begin{enumerate}
\item \label{C:gb-c28:res} Then $\cX\cong[\Spec B/\GL_n]$; in particular, $\cX$  has the resolution property; and
\item \label{C:gb-c28:fGAGA} the natural functor
  \[
  \Coh(\cX) \to \ilim \Coh\bigl( \cX \times_{\Spec R} \Spec
  R/\fm^{n+1}\bigr)
  \]
  is an equivalence of categories.
\end{enumerate}
\end{corollary}
\begin{proof}
    By \cite[Thm.~1]{geraschenko-brown}, we have \itemref{C:gb-c28:res}$\implies$\itemref{C:gb-c28:fGAGA}; thus, it suffices to prove \itemref{C:gb-c28:res}. 

  If $R/\fm 
  =k$ and $k$ is algebraically closed, then $\cX=[\Spec A/G_x]$ by Theorem 
  \ref{T:consequences-gms}. Embed $G_x \subseteq \GL_{N,k}$ for some $N$.
  Then $\cX=[U/\GL_{N,k}]$ where $U=(\Spec A \times \GL_{N,k})/G_x$ is an
  algebraic space. Since $U$ is affine over $\cX$ it is cohomologically affine,
  hence affine by Serre's criterion \cite[Prop.~3.3]{alper-good}.
  In this case, \itemref{C:gb-c28:res} holds even if $R$ is not
  complete but merely henselian.

  If $R/\fm=k$ and $k$ is not algebraically closed, then we proceed as follows. Let 
  $\overline{k}$ be an algebraic closure of $k$. By \cite[$0_{\mathrm{III}}$.10.3.1.3]{EGA}, 
  $\overline{R}=R\otimes_k \overline{k}=\varinjlim_{k \subseteq k' \subseteq \overline{k}} R\tensor_k k'$ is a noetherian local ring with maximal ideal $\overline{\fm} = 
  \fm\overline{R}$ and residue field $\overline{R}/\overline{\fm}\cong \overline{k}$, and the induced 
  map $R/\fm \to \overline{R}/\overline{\fm}$ coincides with $k \to \overline{k}$. Since each $R\otimes_k k'$ is henselian, $\overline{R}$ is henselian.
  By the case considered above, there is a vector bundle $\overline{\shv{E}}$ on
  $\cX_{\overline{k}}$ such that the associated frame bundle is an algebraic space (even
  an affine scheme). Equivalently, for every geometric point $y$ of $\cX$,
  the stabilizer $G_y$ acts faithfully on $\overline{\shv{E}}_y$, cf.\
  \cite[Lem.~2.12]{ehkv}.

  We can find a finite extension
  $k \subseteq k' \subseteq \overline{k}$ and a vector bundle $\shv{E}$ on
  $\cX_{k'}$ that pulls back to $\overline{\shv{E}}$. If $p\colon \cX_{k'}\to \cX$
  denotes the natural map, then $p_*\shv{E}$ is a vector bundle and the counit map
  $p^*p_*\shv{E}\to \shv{E}$ is surjective. In particular, the stabilizer
  actions on $p_*\shv{E}$ are also faithful so the frame bundle $U$ of
  $p_*\shv{E}$ is an algebraic space
  and $\cX=[U/\GL_{N'}]$ is a quotient stack. 
  Since $\cX$ is cohomologically affine
  and $U \to [U/\GL_{N'}]$ is affine, $U$ is affine by Serre's criterion
  \cite[Prop.~3.3]{alper-good}.
  
  In general, let $K=R/\fm$. Since $R$ is a complete $k$-algebra, it admits a coefficient 
  field; thus, it is also a $K$-algebra. We are now free to replace $k$ with 
  $K$ and the result follows.
\end{proof}
\begin{remark}  If $k$ is algebraically closed, then in Corollary \ref{C:gb-c28}\itemref{C:gb-c28:res} above, $\cX$ is in fact isomorphic to a quotient stack $[\Spec A / G_x]$ where $G_x$ is the stabilizer of the unique closed point.  If in addition $\cX$ is smooth, then $\cX \cong [N_x/G_x]$ where $N_x$ is the normal space to $x$ (or equivalently the tangent space of $\cX$ at $x$ if $G_x$ is smooth).
\end{remark}
\subsection{Existence of coherent completions}  \label{A:coherent-completion}
Recall that a \emph{complete local stack} is an excellent local stack $(\cX,x)$
with affine stabilizers such that $\cX$ is coherently complete along the
residual gerbe $\cG_x$ (Definition~\ref{D:complete-local-stack}).

The \emph{coherent completion} of a noetherian stack $\cX$ at a point $x$
is a complete local stack $(\hat{\cX}_x,\hat{x})$ together with a morphism $\eta\colon (\hat{\cX}_x,\hat{x}) \to (\cX,x)$ inducing isomorphisms of $n$th infinitesimal neighborhoods of $\hat{x}$ and $x$. If $\cX$ has affine stabilizers, then the pair $(\hat{\cX}_x,\eta)$ is unique up to unique $2$-isomorphism by Tannaka duality (Corollary~\ref{C:tannakian}).

The next 
result asserts that the coherent completion always exists under very mild hypotheses.  

\begin{theorem} \label{T:complete}
Let $\cX$ be a quasi-separated algebraic stack with affine stabilizers, locally of finite type over an algebraically closed field $k$. For any point $x \in \cX(k)$ with linearly reductive stabilizer $G_x$, the coherent completion $\hat{\cX}_x$ exists.
\begin{enumerate}
  \item \label{T:complete:excellent}
    The coherent completion is an excellent quotient stack
    $\hat{\cX}_x=[\Spec B/G_x]$, and is
    unique up to unique $2$-isomorphism. The invariant ring $B^{G_x}$ is
    the completion of an algebra of finite type over $k$ and $B^{G_x}\to B$
    is of finite type.
  \item \label{T:complete:etale-pres}
    If $f \co (\cW, w) \to (\cX,x)$ is an \'etale morphism, such that $\cW=[\Spec A/G_x]$, the point $w\in |\cW|$ is closed and $f$ induces an isomorphism of stabilizer groups at $w$; then $\hat{\cX}_x = \cW \times_W \Spec \hat{\oh}_{W,\pi(w)}$, where $\pi \co \cW \to W = \Spec A^{G_x}$ is the morphism to the GIT quotient.
  \item \label{T:complete:gms}
	If $\pi \co \cX \to X$ is a good moduli space with affine diagonal, then $\hat{\cX}_x = \cX \times_X \Spec \hat{\oh}_{X,\pi(x)}$.
\end{enumerate}
\end{theorem}

\begin{proof}
Theorem \ref{T:field} gives an \'etale morphism $f \co (\cW, w) \to (\cX,x)$, where $\cW=[\Spec A/G_x]$ and $f$ induces an isomorphism of stabilizer groups at the closed point $w$.  The main statement and Parts \eqref{T:complete:excellent} and \eqref{T:complete:etale-pres} follow by taking $\hat{\cX}_x = \cW \times_W \Spec \hat{\oh}_{W,\pi(w)}$ and $B=A\otimes_{A^{G_x}} \widehat{A^{G_x}}$.
Indeed, $\hat{\cX}_x=[\Spec B/G_x]$ is coherently complete by Theorem~\ref{key-theorem} and $B$ is excellent since it is of finite type over the complete local ring $B^{G_x}=\widehat{A^{G_x}}$.  Part \eqref{T:complete:gms} follows from \eqref{T:complete:etale-pres} after applying Theorem~\ref{T:consequences-gms}.
\end{proof}

\begin{remark} \label{R:miniversal-completion-finite-type}
With the notation of Theorem~\ref{T:complete}~\itemref{T:complete:etale-pres}, observe that if $G_x$ is smooth, then $(\Spec A,w)\to (\cX,x)$ is smooth and adic so the formal miniversal deformation space of $x$ is $\hat{\Def}(x) = \Spf \hat{A}$ where $\hat{A}$ denotes the completion of $A$ at the $G_x$-fixed point $w$ (see Remark~\ref{R:miniversal}).  The stabilizer $G_x$ acts on $\Spf \hat{A}$ and its versal object, and it follows from Theorem \ref{key-theorem} that there is an identification $\hat{A}^{G_x} = \widehat{A^{G_x}}$. In particular, $\hat{A}^{G_x}$ is the completion of a $k$-algebra of finite type.
\end{remark}

\begin{remark} \label{R:uniqueness}
If there exists an \'etale neighborhood $f \co \cW=[\Spec A/G_x] \to \cX$ of $x$ such that $A^{G_x} = k$, then the pair $(\cW, f)$ is unique up to unique 2-isomorphism.   This follows from Theorem \ref{T:complete} as $\cW$ is the coherent completion of $\cX$ at $x$. 
\end{remark}

The \emph{henselization of $\cX$ at $x$} is the stack
$\cX^h_x=\cW\times_W \Spec (A^{G_x})^h$. This stack also satisfies a
universal property (initial among pro-\'etale neighborhoods of the residual
gerbe at $x$) and will be treated in forthcoming work \cite{ahr2}.

\subsection{\'Etale-local equivalences}  \label{A:etale}
Before we state the next result, let us recall that if $\cX$ is an algebraic stack, and $x \in \cX(k)$ is a point, then a formal miniversal deformation space of $x$ is a one-point affine formal scheme $\hat{\Def}(x)$, together with a formally smooth morphism $\hat{\Def}(x) \to \cX$ that is an isomorphism on tangent spaces at $x$, see Remark~\ref{R:miniversal-completion-finite-type}.  If the stabilizer group scheme $G_x$ is smooth and linearly reductive, then $\hat{\Def}(x)$ inherits an action of $G_x$.

\begin{theorem} \label{T:etale}
  Let $\cX$ and $\cY$ be quasi-separated algebraic stacks with affine stabilizers, locally of finite type over an algebraically closed field $k$.  Suppose $x \in \cX(k)$ and $y \in \cY(k)$ are points with smooth linearly reductive stabilizer group schemes $G_x$ and $G_y$, respectively.  Then the following are equivalent:
  \begin{enumerate}
  	\item\label{TI:etale:miniversal}
      There exist an isomorphism $G_x \to G_y$ of group schemes and an isomorphism $\hat{\Def}(x) \to \hat{\Def}(y)$ of formal miniversal deformation spaces which is equivariant with respect to $G_x \to G_y$.
  	\item\label{TI:etale:completion}
	  There exists an isomorphism $\hat{\cX}_x \to \hat{\cY}_y$.
  	\item\label{TI:etale:etale}
	  There exist an affine scheme $\Spec A$ with an action of $G_x$, a point $w \in \Spec A$ fixed by $G_x$, and a diagram of
      \'etale morphisms
	$$\xymatrix{	
			& [\Spec A /G_x] \ar[ld]_f \ar[rd]^g \\
		\cX	&  & \cY
	}$$
	such that $f(w) = x$ and $g(w) = y$, and both $f$ and $g$ induce isomorphisms of stabilizer groups at $w$.
  \end{enumerate}
If additionally $x \in |\cX|$ and $y \in |\cY|$ are smooth, then the conditions above are equivalent to the existence of an isomorphism $G_x \to G_y$ of group schemes and an isomorphism $T_{\cX,x} \to T_{\cY,y}$ of tangent spaces which is equivariant under $G_x \to G_y$.
\end{theorem}

\begin{remark}  If the stabilizers $G_x$ and $G_y$ are not smooth, then the theorem above remains true (with the same argument) if the formal miniversal deformation spaces are replaced with formal completions of equivariant flat adic presentations (Definition \ref{D:adic}) and the tangent spaces are replaced with normal spaces.  Note that the composition $\Spec A \to [\Spec A/ G_x] \to \cX$ produced by Theorem \ref{T:field} is a $G_x$-equivariant flat adic presentation.
\end{remark}

\begin{proof}[Proof of Theorem \ref{T:etale}]
The implications \itemref{TI:etale:etale}$\implies$\itemref{TI:etale:completion}$\implies$\itemref{TI:etale:miniversal} are immediate.  We also have \itemref{TI:etale:miniversal}$\implies$\itemref{TI:etale:completion} as $\cX_x^{[n]} = [\hat{\Def}(x)^{[n]} / G_x]$ and $\cY_y^{[n]} = [\hat{\Def}(y)^{[n]} / G_y]$.  We now show that \itemref{TI:etale:completion}$\implies$\itemref{TI:etale:etale}. We are given an isomorphism $\alpha \co\hat{\cX}_x \iso \hat{\cY}_y$. Let
$f \co (\cW,w)\to (\cX,x)$ be an \'etale neighborhood as in
Theorem~\ref{T:field}, that is, $\cW = [\Spec A/G_x]$ and $f$ induces an isomorphism of stabilizer groups at the closed point $w$. Let $W=\Spec A^{G_x}$ denote the good moduli space of
$\cW$ and let $w_0$ be the image of $w$. Then $\hat{\cX}_x=\cW\times_W \Spec
\hat{\oh}_{W,w_0}$.  The functor $F\co (T\to W)\mapsto \Hom(\cW\times_W T,\cY)$
is locally of finite presentation. Artin approximation applied to $F$ and
$\alpha\in F(\Spec \hat{\oh}_{W,w_0})$ thus gives an \'etale morphism
$(W',w')\to (W,w)$ and a morphism $\varphi\co \cW':=\cW\times_W W'\to \cY$ such
that $\varphi^{[1]}\co \cW'^{[1]}_{w'}\to \cY_y^{[1]}$ is an isomorphism.
Since $\hat{\cW'}_{w'}\cong \hat{\cX}_x\cong \hat{\cY}_y$, it
follows that $\varphi$ induces an isomorphism $\hat{\cW'}\to \hat{\cY}$ by
Proposition~\ref{P:closed/iso-cond:complete}~\itemref{PI:iso:complete}. After replacing $W'$ with an open
neighborhood we thus obtain an \'etale morphism $(\cW',w')\to (\cY,y)$.
The final statement is clear from Theorem \ref{T:smooth}.
\end{proof}

\subsection{The resolution property holds \'etale-locally}\label{A:resolution-property-etale}
In~\cite[Def.~2.1]{rydh-noetherian}, an algebraic stack $\cX$
is said to be of \emph{global type} (resp.\ \emph{s-global type}) if there
is a representable (resp.\ representable and separated) \'etale surjective
morphism $p\colon [V/\GL_n]\to \cX$ of finite presentation where $V$ is
quasi-affine. That is, the resolution property holds for $\cX$
\'etale-locally. We will show that if $\cX$ has linearly reductive stabilizers
at closed points and affine diagonal, then $\cX$ is of s-global type. We begin
with a more precise statement.

\begin{theorem}\label{T:global-type}
Let $\cX$ be a quasi-separated algebraic stack, of finite type over
a perfect (resp.\ arbitrary) field $k$, with
affine stabilizers. Assume that for every closed point $x\in |\cX|$, the unit
component $G_x^0$ of the stabilizer group scheme $G_x$ is linearly reductive.
Then there exists
\begin{enumerate}
\item a finite field extension $k'/k$;
\item a linearly reductive group scheme $G$ over $k'$;
\item a finitely generated $k'$-algebra $A$ with an action of $G$; and
\item an \'etale (resp.\ quasi-finite flat) surjection
  $p\colon [\Spec A/G] \to \cX$.
\end{enumerate}
Moreover, if $\cX$ has affine diagonal, then $p$ can be arranged to be affine.
\end{theorem}
\begin{proof}
First assume that $k$ is algebraically closed. Since $\cX$ is quasi-compact,
Theorem~\ref{T:field} gives an \'etale surjective morphism
$q\co [U_1/G_1]\amalg\dots\amalg [U_n/G_n]\to \cX$ where $G_i$ is a linearly
reductive group scheme over $k$ acting on an affine scheme $U_i$. If we let
$G=G_1\times G_2\times\dots\times G_n$ and let $U$ be the disjoint union of the
$U_i\times G/G_i$, we obtain an \'etale surjective morphism $p\co [U/G]\to \cX$.
If $\cX$ has affine diagonal, then we can assume that $q$, and hence $p$, are
affine.

For general $k$, write the algebraic closure $\overline{k}$ as a union of its
finite subextensions $k'/k$. A standard limit argument gives a solution over
some $k'$ and we compose this with the \'etale (resp.\ flat) morphism
$\cX_{k'}\to \cX$.
\end{proof}

\begin{corollary}\label{C:s-global-type}
Let $\cX$ be an algebraic stack with affine diagonal and of finite type over a
field $k$ (not necessarily algebraically closed). Assume that for every closed
point $x\in |\cX|$, the unit component $G_x^0$ of the stabilizer group scheme
$G_x$ is linearly reductive. Then $\cX$ is of s-global type.
\end{corollary}
\begin{proof}
By Theorem~\ref{T:global-type} there is an affine,
quasi-finite and faithfully flat morphism $\cW\to \cX$ of finite
presentation where $\cW=[\Spec A/G]$ for a linearly reductive group scheme
$G$ over $k'$. If we choose an embedding $G\inj \GL_{n,k'}$, then we can
write $\cW=[\Spec B/\GL_n]$, see proof of Corollary~\ref{C:gb-c28}.
By~\cite[Prop.~2.8~(iii)]{rydh-noetherian}, it follows that $\cX$ is of
s-global type.
\end{proof}

Geraschenko and Satriano define generalized toric Artin stacks in terms of generalized stacky fans.
They establish that over an algebraically closed field of characteristic $0$, an Artin stack $\cX$ with finite quotient singularities is toric if and only if it has affine
diagonal, has an open dense torus $T$ acting on the stack, has linearly
reductive stabilizers, and $[\cX/T]$ is of global
type~\cite[Thm.~6.1]{GS-toric-stacks-2,GS-toric-stacks-2-erratum}. If $\cX$ has linearly reductive
stabilizers at closed points, then so has
$[\cX/T]$. Corollary~\ref{C:s-global-type} thus shows that the last condition
is superfluous.

\section{Global applications}\label{S:global_applications}
\subsection{Compact generation of derived categories} \label{A:compact-generation}
For results involving derived categories of quasi-coherent sheaves, perfect (or compact) generation of the unbounded derived category $\DQCOH(\cX)$ continues to be an indispensable tool at one's disposal \cite{neeman_duality,BZFN}.  We prove:  
\begin{theorem}\label{T:compact-generation}
  Let $\cX$ be an algebraic stack of finite type over a field $k$ (not assumed to be algebraically closed) with 
  affine diagonal. If the stabilizer group $G_x$ has linearly reductive
  identity component $G_x^0$ for every closed point of $\cX$, then $\cX$ has the Thomason condition; that is, 
  \begin{enumerate}
  \item $\DQCOH(\cX)$ is compactly generated by a countable set of
    perfect complexes; and
  \item for every open immersion $\cU\subseteq \cX$, there exists a compact and perfect complex $P \in \DQCOH(\cX)$ with support precisely $\cX\setminus \cU$. 
  \end{enumerate}
\end{theorem}

\begin{proof} This follows from
Corollary~\ref{C:s-global-type}
together with \cite[Thm.~B]{perfect_complexes_stacks} (characteristic $0$) and
Theorem~\ref{T:global-type} together with
\cite[Thm.~D]{hallj_dary_alg_groups_classifying} (positive characteristic).
\end{proof}

Theorem \ref{T:compact-generation} was previously only known for stacks with finite stabilizers~\cite[Thm.~A]{perfect_complexes_stacks} or quotients of quasi-projective schemes by a linear action of an affine algebraic group in characteristic $0$ \cite[Cor.~3.22]{BZFN}.

In positive characteristic, the theorem is almost sharp: if the reduced
identity component $(G_x)^0_\red$ is not linearly reductive, i.e., not a torus,
at some point $x$, then $\DQCOH(\cX)$ is not compactly
generated~\cite[Thm.~1.1]{hallj_neeman_dary_no_compacts}.

If $\cX$ is an algebraic stack of finite type over $k$ with affine stabilizers such that either
\begin{enumerate}
\item the characteristic of $k$ is $0$; or
\item \emph{every} stabilizer is linearly reductive;
\end{enumerate}
then $\cX$ is concentrated, that is, a complex of $\oh_{\cX}$-modules with quasi-coherent cohomology is perfect if and only if it is a compact object of $\DQCOH(\cX)$ \cite[Thm.~C]{hallj_dary_alg_groups_classifying}.
If $\cX$ admits a good moduli space $\pi\colon \cX\to X$ with affine diagonal,
then one of the two conditions hold by Theorem~\ref{T:consequences-gms}.
If $\cX$ does not admit a good moduli space and is of positive characteristic,
then it is not sufficient that closed points have linearly reductive
stabilizers as the following example shows.
\begin{example}
Let $\cX=[X/(\GG_m\times \ZZ/2\ZZ)]$ be the quotient of the non-separated affine
line $X$ by the natural $\GG_m$-action and the $\ZZ/2\ZZ$-action that swaps the
origins. Then $\cX$ has two points, one closed with stabilizer group $\GG_m$
and one open point with stabilizer group $\ZZ/2\ZZ$. Thus if $k$ has
characteristic two, then not every stabilizer group is linearly reductive and
there are non-compact perfect
complexes~\cite[Thm.~C]{hallj_dary_alg_groups_classifying}.
\end{example}

\subsection{Characterization of when $\cX$ admits a good moduli space} \label{A:gms}

Using the existence of completions, we can give an intrinsic characterization of those algebraic stacks that admit a good moduli space.

We will need one preliminary definition.  We say that a geometric point $y \co \Spec K \to \cX$ is {\it geometrically closed} if the image of $(y, \mathrm{id}) \co \Spec K \to \cX \otimes_k K$ is a closed point of $|\cX \otimes_k K|$.

\begin{theorem} \label{T:gms} Let $\cX$ be an algebraic stack with affine diagonal, locally of finite type over an algebraically closed field $k$. Then $\cX$ admits a good moduli space if and only if 
\begin{enumerate}
\item\label{TI:gms:unique-closed}
  For every point $y \in \cX(k)$, there exists a unique closed point in the closure $\overline{ \{ y \}}$.
\item\label{TI:gms:cc} 
 For every closed point $x \in \cX(k)$, the stabilizer group scheme $G_x$ is linearly reductive and the morphism $\hat{\cX}_x \to \cX$ from the coherent completion of $\cX$ at $x$ satisfies:
	\begin{enumerate}
	\item\label{TI:gms:stab-pres}
      The morphism $\hat{\cX}_x \to \cX$ is stabilizer preserving at every point; that is, $\hat{\cX}_x \to \cX$ induces an isomorphism of stabilizer groups for every point $\xi \in |\hat{\cX}_x|$.
	\item\label{TI:gms:geom-closed}
      The morphism $\hat{\cX}_x \to \cX$ maps geometrically closed points to geometrically closed points.
	\item\label{TI:gms:injective-on-k}
      The map $\hat{\cX}_x(k) \to \cX(k)$ is injective.
	\end{enumerate}
\end{enumerate}
\end{theorem}

\begin{remark}
The quotient $[\PP^1 / \GG_m]$ (where $\GG_m$ acts on $\PP^1$ via multiplication) does not satisfy \itemref{TI:gms:unique-closed}.
If $\cX=[X/(\ZZ/2\ZZ)]$ is the quotient of the non-separated affine line $X$ by the $\ZZ/2\ZZ$-action which swaps the origins (and acts trivially elsewhere), then the map $\Spec k\llbracket x \rrbracket = \hat{\cX}_0 \to \cX$ from the completion at the origin does not satisfy \itemref{TI:gms:stab-pres}.
If $\cX=[(\AA^2 \setminus 0) / \GG_m]$ where $\GG_m$-acts via $t \cdot (x,y) = (x,ty)$ and $p=(0,1) \in |\cX|$, then the map $\Spec k\llbracket x \rrbracket = \hat{\cX}_{p} \to \cX$  does not satisfy \itemref{TI:gms:geom-closed}.
If  $\cX=[C/\GG_m]$ where $C$ is the nodal cubic curve with a $\GG_m$-action and $p \in |\cX|$ denotes the image of the node, then $[\Spec(k[x,y]/xy) / \GG_m] = \hat{\cX}_{p} \to \cX$ 
does not satisfy \itemref{TI:gms:injective-on-k}. (Here $\GG_m$ acts on coordinate axes via $t \cdot (x,y) = (tx, t^{-1}y)$.)  These pathological examples in fact appear in many natural moduli stacks; see \cite[App.~A]{afsw-good}.
\end{remark}

\begin{remark}
Consider the non-separated affine line as a group scheme $G \to \AA^1$ whose
generic fiber is trivial but the fiber over the origin is $\ZZ/2\ZZ$ and let
$\cX=[\AA^1/G]$. In this case
\itemref{TI:gms:stab-pres} is not satisfied. Nevertheless, the stack
$\cX$ does have a good moduli space $X=\AA^1$ but $\cX\to X$ has
non-separated diagonal.
\end{remark}

\begin{remark}
When $\cX$ has finite stabilizers, then
conditions~\itemref{TI:gms:unique-closed}, \itemref{TI:gms:geom-closed} and
\itemref{TI:gms:injective-on-k} are always
satisfied. Condition~\itemref{TI:gms:stab-pres} is satisfied if and only if the
inertia stack is finite over $\cX$. In this case, the good moduli space of $\cX$
coincides with the coarse space of $\cX$, which exists by~\cite{keel-mori}.
\end{remark}

\begin{proof}[Proof of Theorem \ref{T:gms}]
For the necessity of the conditions, properties of good moduli spaces imply \itemref{TI:gms:unique-closed}  \cite[Thm.\ 4.16(ix)]{alper-good} and that $G_x$ is linearly reductive for a closed point $x \in |\cX|$ \cite[Prop.\ 12.14]{alper-good}.  The rest of \itemref{TI:gms:cc} follows from the explicit description of the coherent completion in Theorem~\ref{T:complete}~\itemref{T:complete:gms}.

For the sufficiency, we claim it is enough to verify:
\begin{enumerate}
\item[(I)] For every closed point $x \in |\cX|$, there exists an affine \'etale morphism
$$f \co (\cX_1, w) \to (\cX, x),$$ such that $\cX_1 = [\Spec A / G_x]$ and for each closed point $w' \in \cX_1$,
\begin{enumerate}
\item $f$ induces an isomorphism of stabilizer groups at $w'$; and
\item  $f(w')$ is closed.
\end{enumerate}
\item[(II)] For every $y \in \cX(k)$, the closed substack $\overline{ \{y\}}$ admits a good moduli space.
\end{enumerate}
This is proven in \cite[Thm.~1.2]{afsw-good} but we provide a quick proof here for the convenience of the reader.  For a closed point $x \in |\cX|$, consider the \v{C}ech nerve of an affine \'etale morphism $f$ satisfying (I)
\[
\xymatrix{ \cdots    \cX_3 \ar@<-1ex>[r] \ar[r] \ar@<1ex>[r]  &  \cX_2 \ar@<.5ex>[r]^{\smash{p_1}} \ar@<-.5ex>[r]_{p_2} & \cX_1 \ar[r]^-{\smash{f}} & \cX_0 = \im(f).
}
\]
Since $f$ is affine, there are good moduli spaces $\cX_i \to X_i$ for each $i \ge 1$, and morphisms $\xymatrix{\cdots  X_3  \ar@<1ex>[r] \ar@<-1ex>[r] \ar[r]  & X_2 \ar@<.5ex>[r] \ar@<-.5ex>[r]  & X_1.}$

We claim that both projections $p_1, p_2 \co \cX_2 \to \cX_1$ send closed points to closed points.  If $x_2 \in |\cX_2|$ is a closed point, to check that $p_j(x_2) \in |\cX_1|$ is also closed for either $j=1$ or $j=2$,  we may replace the $\cX_i$ with the base changes along $\overline{ \{f(p_j(x_2))\}}\inj \cX_0$.  In this case, there is a good moduli space $\cX_0\to X_0$ by (II) and $\cX_1\to \cX_0$ sends closed points to closed points and are stabilizer preserving at closed points by (I).
Luna's fundamental lemma (Proposition \ref{P:luna}) then implies that $\cX_1 \cong \cX_0 \times_{X_0} X_1$ and the claim follows.

Luna's fundamental lemma (Proposition \ref{P:luna}) now applies to $p_j\co \cX_2\to \cX_1$ and says that $X_2 \to X_1$ is \'etale and that $p_j$ is the base change of this map along $\cX_1 \to X_1$.  The analogous fact holds for the maps $X_3 \to X_2$.  The universality of good moduli spaces induces an \'etale groupoid structure $X_2 \rightrightarrows X_1$.  To check that this is an \'etale equivalence relation, it suffices to check that $X_2 \to X_1 \times X_1$ is injective on $k$-points but this follows from the observation the $|\cX_2| \to |\cX_1| \times |\cX_1|$ is injective on closed points.  It follows that there is an algebraic space quotient $X_0 := X_1/X_2$ and a commutative cube 
\[
\xymatrix@=15pt{
				&\cX_2 \ar[rr] \ar[dd] \ar[dl]			&				& \cX_1 \ar[dd] \ar[dl] \\
\cX_1 \ar[rr] \ar[dd] &						& \cX_0 \ar[dd]	& \\
				& X_2 \ar[rr] \ar[dl]				& 				& X_1 \ar[dl]  \\
X_1  \ar[rr]	&							& X_0. 		&
}
\]
Since the top, left, and bottom faces are cartesian, it follows from \'etale descent that the right face is also cartesian and that $\cX_0 \to X_0$ is a good moduli space.  Moreover, each closed point of $\cX_0$ remains closed in $\cX$ by (Ib).  Therefore, if we apply the construction above to find such an open neighborhood $\cX_x$ around each closed point $x$, the good moduli spaces of $\cX_x$ glue to form a good moduli space of~$\cX$.

We now verify condition (I).  Let $x \in \cX(k)$ be a closed point.  By Theorem \ref{T:field}, there
exist a quotient stack $\cW = [\Spec A / G_x]$ with a closed point $w \in |\cW|$ and an affine \'etale morphism $f \co (\cW, w) \to (\cX, x)$ such that $f$ is stabilizer preserving at $w$.
 As the coherent completion of $\cW$ at $w$ is identified with $\hat{\cX}_x$, we have a 2-commutative diagram
\begin{equation} \label{E:completion}
\begin{split}
\xymatrix{
\hat{\cX}_x \ar[d] \ar[rd]	&  \\
\cW \ar[r]^(.4){f}		& \cX.
} \end{split}
\end{equation}
The subset $Q_a\subseteq |\cW|$ consisting of points $\xi \in |\cW|$ such that
$f$ is stabilizer preserving at $\xi$ is constructible.
Since $Q_a$ contains every point in the image of $\hat{\cX}_x \to \cW$ by
hypothesis \itemref{TI:gms:stab-pres}, it follows that $Q_a$ contains a
neighborhood of $w$. Thus after replacing $\cW$ with an open saturated
neighborhood containing $w$ (Lemma~\ref{L:shrink}), we may assume that $f \co
\cW \to \cX$ satisfies condition (Ia).

Let $\pi\co \cW\to W$ be the good moduli space of $\cW$ and consider the
morphism $g=(f,\pi)=\cW\to \cX\times W$. For a point $\xi\in |W|$, let
$\xi^0\in |\cW|$ denote the unique point that is closed in the fiber $\cW_\xi$.
Let $Q_b\subseteq |W|$ be the locus of points $\xi\in |W|$ such that $g(\xi^0)$
is closed in $|(\cX\times W)_\xi|=|\cX_{\kappa(\xi)}|$. This locus is
constructible. Indeed, the subset $\cW^0=\{\xi^0\;:\;\xi\in |W|\}\subseteq
|\cW|$ is easily seen to be constructible; hence so is $g(\cW^0)$ by
Chevalley's theorem. The locus $Q_b$ equals the set of points $\xi\in |W|$ such
that $g(\cW^0)_\xi$ is closed which is constructible by~\cite[IV.9.5.4]{EGA}. The
locus $Q_b$ contains $\Spec \oh_{W,\pi(w)}$ by hypothesis
\itemref{TI:gms:geom-closed} (recall that $\hat{\cX}_x=\cW\times_W \Spec
\hat{\oh}_{W,\pi(w)}$). Therefore, after replacing $\cW$ with an open saturated
neighborhood of $w$, we may assume that $f \co \cW \to \cX$ satisfies condition
(Ib).

For condition (II), we may replace $\cX$ by $\overline{ \{y\} }$. By \itemref{TI:gms:unique-closed}, there is a unique closed point $x \in \overline{ \{y\} }$ and we can find a commutative diagram as in \eqref{E:completion} for $x$.  By \itemref{TI:gms:geom-closed} we can, since $f$ is \'etale, also assume that $\cW$ has a unique closed point. Now let $B=\Gamma(\cW,\oh_{\cW})$; then $B$ is a domain of finite type over $k$ \cite[Thm.~4.16(viii),(xi)]{alper-good}; in particular, $B$ is a Jacobson domain. Since $\cW$ has a unique closed point, so too does $\Spec B$ \cite[Thm.~4.16(iii)]{alper-good}. Hence, $B$ is also local and so $\Gamma(\cW, \oh_{\cW}) = k$. By Theorem \ref{T:complete}\itemref{T:complete:etale-pres}, $\hat{\cX}_x = \cW$.  By hypothesis \itemref{TI:gms:injective-on-k}, $f \co \cW \to \cX$ is an \'etale monomorphism which is also surjective by hypothesis \itemref{TI:gms:unique-closed}.  We conclude that $f \co \cW \to \cX$ is an isomorphism establishing condition (II).
\end{proof}

\subsection{Algebraicity results} \label{A:algebraicity}

In this subsection, we fix a field $k$ (not necessarily algebraically closed), an algebraic space $X$ locally of finite type over $k$, and an algebraic stack $\cX$ of finite type over $X$ with affine diagonal over $X$ such that $\cX \to X$ is a good moduli space. We prove the following algebraicity results.

\begin{theorem}[Stacks of coherent sheaves]\label{T:coh}
  The $X$-stack $\underline{\Coh}_{\cX/X}$, whose objects over $T \to X$ are finitely presented quasi-coherent sheaves on $\cX \times_X T$ flat over $T$, is an algebraic stack, locally of finite type over $X$, with affine diagonal over $X$. 
\end{theorem}

\begin{corollary}[Quot schemes]\label{C:quot}
Let $\cF$ be a quasi-coherent $\oh_{\cX}$-module. The $X$-sheaf $\underline{\Quot}_{\cX/X}(\cF)$, whose objects over $T \to X$ are quotients $p_1^* \cF \to \cG$, where $p_1 \co \cX \times_X T \to \cX$ is the projection and $\cG$ is a finitely presented quasi-coherent $\oh_{\cX \times_X T}$-module that is flat over $T$, is a separated algebraic space over $X$. In addition, if $\cF$ is coherent, then $\underline{\Quot}_{\cX/X}(\cF)$ is locally of finite type over $X$.
\end{corollary}
\begin{corollary}[Hilbert schemes]\label{C:hilb}
 The $X$-sheaf $\underline{\mathrm{Hilb}}_{\cX/X}$, whose objects over $T \to X$ are closed substacks $\cZ \subseteq \cX \times_X T$ such that $\cZ$ is flat and locally of finite presentation over $T$, is a separated algebraic space locally of finite type over $X$.
\end{corollary}

\begin{theorem}[Hom stacks] \label{T:hom}
Let $\cY$ be a quasi-separated algebraic stack, locally of finite type over $X$ with affine stabilizers.
If $\cX \to X$ is flat, then the $X$-stack $\underline{\Hom}_X(\cX, \cY)$, whose objects are pairs consisting of a morphism $T \to X$ of algebraic spaces and a morphism $\cX \times_X T \to \cY$ of algebraic stacks over $X$, is an algebraic stack, locally of finite type over $X$, with quasi-separated diagonal.  If $\cY \to X$ has affine (resp.\ quasi-affine, resp.\ separated) diagonal, then the same is true for $\underline{\Hom}_X(\cX, \cY) \to X$.
\end{theorem}
A general algebraicity theorem for Hom stacks was also considered in \cite{hlp}.  In the setting of Theorem \ref{T:hom}---without the assumption that $\cX \to X$ is a good moduli space---\cite[Thm.~1.6]{hlp} establishes the algebraicity of $\underline{\Hom}_X(\cX, \cY)$ under the additional hypotheses that $\cY$ has affine diagonal and that $\cX \to X$ is cohomologically projective \cite[Def.~1.12]{hlp}.  We note that as a consequence of Theorem \ref{T:consequences-gms}, if $\cX \to X$ is a good moduli space, then $\cX \to X$ is necessarily \'etale-locally cohomologically projective.

We also prove the following, which we have not seen in the literature before.
\begin{corollary}[$G$-equivariant Hom stacks] \label{C:homG}
Let $Z$ and $S$ be quasi-separated algebraic spaces, locally of finite type over $k$. Let $\cX$ be a quasi-separated Deligne--Mumford stack, locally of finite type over $k$. 
Let $G$ be a linearly reductive affine group scheme acting on $Z$ and $\cX$.
Let $Z \to S$ and $\cX \to S$ be
$G$-invariant morphisms.  Suppose that $Z \to S$ is flat and a good GIT quotient.  Then the $S$-groupoid
$\underline{\Hom}_S^{G}(Z, \cX)$, whose objects over $T \to S$ are $G$-equivariant $S$-morphisms $Z \times_S T \to \cX$, 
is a Deligne--Mumford stack, locally of finite type over $S$. In addition, if $\cX$ is an algebraic space, then so is $\underline{\Hom}_S^{G}(Z, \cX)$, and
if $\cX$ has quasi-compact and separated diagonal, then so has $\underline{\Hom}_S^{G}(Z, \cX)$.
\end{corollary}
The results of this section will largely be established using Artin's criterion, as formulated in \cite[Thm.~A]{hallj_openness_coh}. This uses the notion of coherence, in the sense of Auslander \cite{auslander}, which we now briefly recall.

Let $A$ be a ring. An additive functor $F \colon \Mod(A) \to \Ab$ is \emph{coherent} if there is a morphism of $A$-modules $\varphi\colon M \to N$ and a functorial isomorphism:
\[
  F(-) \simeq \coker(\Hom_{A}(N,-) \xrightarrow{\varphi^*} \Hom_{A}(M,-)). 
\]
Coherent functors are a remarkable collection of functors, and form an
abelian subcategory of the category of additive functors
$\Mod(A) \to \Ab$ that is closed under limits \cite{auslander}. It is obvious that coherent functors preserve small (i.e., infinite) products and work
of Krause \cite[Prop.~3.2]{MR2026723} implies that this essentially characterizes them. They frequently arise in algebraic geometry as cohomology and
$\Ext$-functors \cite{hartshorne-coherent}. Fundamental results, such
as cohomology and base change, are very simple consequences of the
coherence of cohomology functors \cite{hallj_coho_bc}.
\begin{example}\label{E:coherent-functor1}
  Let $A$ be a ring and $M$ an $A$-module. Then the functor
  $I \mapsto \Hom_A(M,I)$ is coherent. Taking $M=A$, we see that the
  functor $I \mapsto I$ is coherent.
\end{example}
\begin{example}\label{E:coherent-functor2}
  Let $R$ be a noetherian ring and $Q$ a finitely generated
  $R$-module. Then the functor
  $I \mapsto {Q}\otimes_{R} I$ is
  coherent. Indeed, there is a presentation
  $A^{\oplus r} \to A^{\oplus s} \to {Q} \to 0$ and an
  induced functorial isomorphism
  \[
    {Q}\otimes_{A} I = \coker(I^{\oplus r} \to I^{\oplus s}).
  \]
  Since the category of coherent functors is abelian, the claim
  follows from Example \ref{E:coherent-functor1}. More generally, if
  $Q^\bullet$ is a bounded above complex of coherent
  $R$-modules, then for each $j\in \ZZ$ the functor:
  \[
    I \mapsto H^j({Q}^\bullet\otimes^{\LDERF}_{R} I)
  \]
  is coherent \cite[Ex.~2.3]{hartshorne-coherent}. Such a coherent
  functor is said to \emph{arise from a complex}. An additive functor
  $F \colon \Mod(A) \to \Ab$ is \emph{half-exact} if for every short
  exact sequence $0 \to I'' \to I \to I' \to 0$ the sequence
  $F(I'') \to F(I) \to F(I')$ is exact. Functors arising from
  complexes are half-exact. Moreover, if $F \colon \Mod(A) \to \Ab$ is
  coherent, half-exact and preserves direct limits of $A$-modules,
  then it is the direct summand of a functor that arises from a complex
  \cite[Prop.~4.6]{hartshorne-coherent}.
\end{example}

The following proposition is a variant of \cite[Thm.~C]{hallj_coho_bc} and \cite[Thm.~D]{perfect_complexes_stacks}.
\begin{proposition}\label{P:coh_gms}
  Let $\cplx{F} \in \DQCOH(\cX)$ and $\cplx{G} \in \DCAT_{\Coh}^b(\cX)$. If $X=\Spec R$ is affine, then the 
  functor
  \[
    H_{\cplx{F},\cplx{G}}(-):=\Hom_{\oh_{\cX}}\bigl(\cplx{F},\cplx{G} \tensor_{\oh_\cX}^{\LDERF} \QCPBK{\pi}(-)\bigr) \colon \Mod(R) \to \Ab
  \]
  is coherent.
\end{proposition}
\begin{proof}
  This follows immediately from Theorem \ref{T:compact-generation},
  \cite[Thm.~4.16(x)]{alper-good} and
  \cite[Cor.~4.19]{perfect_complexes_stacks}. The argument is
  reasonably straightforward, however, so we sketch it here.  To this end if 
  $\cplx{F} = \oh_{\cX}$, then the projection formula
  \cite[Prop.~4.11]{perfect_complexes_stacks} implies that
  \begin{align*}
    \Hom_{\oh_{\cX}}\bigl(\oh_{\cX},\cplx{G} \tensor_{\oh_\cX}^{\LDERF} \QCPBK{\pi}(-)\bigr) &= H^0\bigl(\RDERF \Gamma(\cX,\cplx{G} \tensor_{\oh_\cX}^{\LDERF} \QCPBK{\pi}(-))\bigr)\\
    &\simeq H^0\bigl(\RDERF \Gamma(\cX,\cplx{G}) \tensor_{R}^{\LDERF} - \bigr).
  \end{align*}
  Since $\cplx{G} \in \DCAT_{\Coh}^b(\cX)$ and $\cX \to X$ is a good
  moduli space, it follows that
  $\RDERF \Gamma(\cX,\cplx{G}) \in \DCAT^b_{\Coh}(R)$. Indeed
  $H^j(\RDERF \Gamma(\cX,\cplx{G})) \simeq
  \Gamma(\cX,\cH^j(\cplx{G}))$, because $\RDERF\Gamma(\cX,-)$ is $t$-exact on
  quasi-coherent sheaves (see \S\ref{S:notation}); now apply \cite[Thm.~4.16(x)]{alper-good}. By Example \ref{E:coherent-functor2}, the
  functor $H_{\oh_{\cX},\cplx{G}}$ is coherent. If $\cplx{F}$ is
  perfect, then
  $H_{\cplx{F},\cplx{G}} \simeq H_{\oh_{\cX},{\cplx{F}}^\vee
    \tensor^{\LDERF}_{\oh_{\cX}} \cplx{G}}$ is also coherent by the
  case just established. Now let $\mathcal{T} \subseteq \DQCOH(\cX)$
  be the full subcategory consisting of those $\cplx{F}$ such that
  $H_{\cplx{F},\cplx{G}}$ is coherent for every
  $\cplx{G} \in \DCAT^b_{\Coh}(\cX)$. Certainly, $\mathcal{T}$ is a
  triangulated subcategory that contains the perfect complexes. If
  $\{\cplx{F}_\lambda\}_{\lambda\in \Lambda} \subseteq \DQCOH(\cX)$,
  then
  $\prod_{\lambda\in \Lambda} H_{\cplx{F}_\lambda,\cplx{G}} \simeq
  H_{\oplus_\lambda \cplx{F}_\lambda,\cplx{G}}$. Since small products
  of coherent functors are coherent \cite[Ex.~4.9]{hallj_coho_bc}, it
  follows that $\mathcal{T}$ is closed under small coproducts. By
  Thomason's Theorem (e.g.,
  \cite[Cor.~3.14]{perfect_complexes_stacks}),
  $\mathcal{T} = \DQCOH(\cX)$. The result follows.
\end{proof}
The following corollary is a variant of \cite[Thm.~D]{hallj_coho_bc}.
\begin{corollary}\label{C:aff_hom_fund}
  Let $\cF$ be a quasi-coherent $\oh_{\cX}$-module and let $\cG$ be a coherent $\oh_{\cX}$-module. If $\cG$ is flat over $X$, then the $X$-presheaf $\underline{\Hom}_{\oh_{\cX}/X}(\cF,\cG)$ whose objects over $T \xrightarrow{\tau} X$ are homomorphisms $\tau_{\cX}^*\cF \to \tau_{\cX}^*\cG$ of $\oh_{\cX\times_X T}$-modules (where $\tau_{\cX} \co \cX \times_X T \to \cX$ is the projection) 
  is representable by an affine $X$-scheme. 
\end{corollary}
\begin{proof}
  We argue exactly as in the proof of \cite[Thm.~D]{hallj_coho_bc}, but
  using Proposition \ref{P:coh_gms} in place of
  \cite[Thm.~C]{hallj_coho_bc}. Again, the argument is quite short, so
  we sketch it here for completeness. First, we may obviously reduce
  to the situation where $X=\Spec R$. Next, since $\shv{G}$ is flat
  over $X$, it follows that
  $\shv{G} \tensor^{\LDERF}_{\oh_{\cX}} \LDERF \pi^*_{\QCoh}(-)
  \simeq \shv{G} \tensor_{\oh_{\cX}} \pi^*(-)\simeq \shv{G}
  \tensor_R (-)$. The functor
  $H(-)= \Hom_{\oh_{\cX}}(\shv{F},\shv{G} \tensor_{\oh_{\cX}}
  \pi^*(-))$ is also coherent and left-exact (Proposition \ref{P:coh_gms}). But coherent functors
  preserve small products \cite[Ex.~4.8]{hallj_coho_bc}, so the
  functor above preserves all limits. It follows from the
  Eilenberg--Watts Theorem \cite[Thm.~6]{MR0118757} (also see
  \cite[Ex.~4.10]{hallj_coho_bc} for further discussion) that there is
  an $R$-module $Q$ and an isomorphism of functors
  $H(-) \simeq \Hom_{R}(Q,-)$. Finally, consider an $R$-algebra $C$
  and let $T=\Spec C$; then there are functorial isomorphisms:
  \begin{align*}
    \underline{\Hom}_{\oh_{\cX}/X}(\shv{F},\shv{G})(T \to X) &= \Hom_{\oh_{\cX\times_X T}}(\tau_{\cX}^*\shv{F},\tau_{\cX}^*\shv{G})\\
                                                               &\simeq \Hom_{\oh_{\cX}}(\shv{F},(\tau_{\cX})_*\tau_{\cX}^*\shv{G})\\
                                                               &\simeq \Hom_{\oh_{\cX}}(\shv{F},\shv{G} \otimes_R C)\\
                                                               &\simeq \Hom_{R}(Q,C)\\
                                                               &\simeq \Hom_{R\mathrm{-Alg}}(\Sym^{\bullet}_RQ,C)\\
                                                               &\simeq \Hom_X(T,\Spec (\Sym^{\bullet}_RQ)).
  \end{align*}
  Hence, 
  $\underline{\Hom}_{\oh_{\cX}/X}(\cF,\cG)$ is represented by the
  affine $X$-scheme $\Spec (\Sym_R^{\bullet} Q)$.
\end{proof}

\begin{proof} [Proof of Theorem \ref{T:coh}]
  The argument is very similar to the proof of
  \cite[Thm.~8.1]{hallj_openness_coh}. We may assume that $X=\Spec R$
  is affine. If $C$ is an $R$-algebra, it will be convenient to let
  $\cX_C = \cX\times_X \Spec C$. We also let
  $\Coh^{\mathrm{flb}}(\cX_C)=\underline{\Coh}_{\cX/X}(\Spec C)$;
  that is, it denotes the full subcategory of $\Coh(\cX_C)$ with
  objects those coherent sheaves that are $C$-flat.

  Note that $R$ is of finite type over a field $k$, so $X$ is an
  excellent scheme.  We may now use the criterion of
  \cite[Thm.~A]{hallj_openness_coh}. There are six conditions to
  check.
  \begin{enumerate}
  \item {[Stack]} $\underline{\Coh}_{\cX/X}$ is a stack for the \'etale
    topology. This is immediate from \'etale descent of quasi-coherent
    sheaves.
  \item {[Limit preservation]} If $\{A_j\}_{j\in J}$ is a direct system
    of $R$-algebras with limit $A$, then the natural functor
    $\varinjlim_j \Coh^{\mathrm{flb}}(\cX_{A_j}) \to
    \Coh^{\mathrm{flb}}(\cX_A)$ is an equivalence of
    categories. This is immediate from standard limit results \cite[\S\S IV.8, IV.11]{EGA}. 
  \item {[Homogeneity]} Given a diagram of $R$-algebras
    $[B \to A \leftarrow A']$, where $A' \to A$ is surjective with
    nilpotent kernel, then the natural functor:
    \[
      \Coh^{\mathrm{flb}}(\cX_{B \times_A A'}) \to {\Coh}^{\mathrm{flb}}(\cX_B) \times_{{\Coh}^{\mathrm{flb}}(\cX_A)}\Coh^{\mathrm{flb}}(\cX_{A'})
    \]
    induces an equivalence of categories. This is just a strong version of Schlessinger's
    conditions, which is proved in \cite[Lem.~8.3]{hallj_openness_coh}
    (see \cite[p.~166]{hallj_openness_coh} for further discussion).
  \item {[Effectivity]} If $(A,\mathfrak{m})$ is an
    $\mathfrak{m}$-adically complete noetherian local ring, then the natural functor:
    \[
      \Coh^{\mathrm{flb}}(\cX_A) \to \varprojlim_n \Coh^{\mathrm{flb}}(\cX_{A/\mathfrak{m}^{n+1}})
    \]
    is an equivalence of categories. This is immediate from Corollary \ref{C:gb-c28}\itemref{C:gb-c28:fGAGA} and the local criterion of flatness. 
  \item {[Conditions on automorphisms and deformations]} If $A$ is a
    finite type $R$-algebra and
    $\shv{F} \in \Coh^{\mathrm{flb}}(\cX_A)$, then the infinitesimal
    automorphism and deformation functors associated to $\shv{F}$ are
    coherent. It is established in
    \cite[\S8]{hallj_openness_coh} that as additive functors from $\Mod(A) \to \Ab$:
    \begin{align*}
      \Aut_{\underline{\Coh}_{\cX/X}}(\shv{F},-) &= \Hom_{\oh_{\cX_A}}(\shv{F},\shv{F} \otimes_{\oh_{\cX_A}} \pi_A^*(-))\quad \mbox{and}\\
      \Def_{\underline{\Coh}_{\cX/X}}(\shv{F},-) &= \Ext^1_{\oh_{\cX_A}}(\shv{F},\shv{F} \otimes_{\oh_{\cX_A}} \pi_A^*(-)).
    \end{align*}
    By Proposition \ref{P:coh_gms}, the functors above are coherent. 
  \item {[Conditions on obstructions]} If $A$ is a finite type $R$-algebra and $\shv{F} \in \Coh^{\mathrm{flb}}(\cX_A)$, then there is an integer $n$ and a coherent $n$-step obstruction theory for $\shv{F}$. The obstruction theory is described in \cite[\S8]{hallj_openness_coh}. If $\cX \to X$ is flat, then there is the usual $1$-step obstruction theory
    \[
      \mathrm{O}^2(\shv{F},-) = \Ext^2_{\oh_{\cX_A}}(\shv{F},\shv{F} \otimes_{\oh_{\cX_A}} \pi_A^*(-)).
    \]
    If $\cX \to X$ is not flat, then $\mathrm{O}^2(\shv{F},-)$ is the second step in a $2$-step obstruction theory, whose primary obstruction lies in:
    \[
      \mathrm{O}^1(\shv{F},-) = \Hom_{\oh_{\cX_A}}(\shv{F} \otimes_{\oh_{\cX_A}} \shv{T}or_1^{X}(A,\oh_{\cX}),\shv{F} \otimes_{\oh_{\cX_A}} \pi_A^*(-)).
    \]
    By Proposition \ref{P:coh_gms}, these functors are coherent.
  \end{enumerate}
  Hence, $\underline{\Coh}_{\cX/X}$ is an algebraic stack that is
  locally of finite presentation over $X$. Corollary
  \ref{C:aff_hom_fund} now implies that the diagonal is affine. This completes the proof. 
\end{proof}

Corollaries \ref{C:quot} and \ref{C:hilb} follow immediately from Theorem \ref{T:coh}. Indeed, the natural functor
$\underline{\mathrm{Quot}}_{\cX/X}(\cF)\to \underline{\Coh}_{\cX/X}$ is quasi-affine by Corollary \ref{C:aff_hom_fund} and Nakayama's Lemma (see \cite[Lem.~2.6]{lieblich-coherent} for details). 

\begin{proof}[Proof of Theorem \ref{T:hom}]
  This only requires small modifications to the proof of
  \cite[Thm.~1.2]{hallj_dary_coherent_tannakian_duality}, which again
  uses Artin's criterion as formulated in
  \cite[Thm.~A]{hallj_openness_coh}. In more detail: We may assume
  that $X=\Spec R$ is affine; in particular, $X$ is excellent so we
  may apply \cite[Thm.~A]{hallj_openness_coh}. As in the proof of
  Theorem \ref{T:coh}, there are six conditions to check. The
  conditions (1) [Stack], (2) [Limit preservation] and (3)
  [Homogeneity] are largely routine, and ultimately rely upon
  \cite[\S9]{hallj_openness_coh}. Condition (4) [Effectivity] follows
  from an idea due to Lurie, and makes use of Tannaka
  duality. Indeed, if $(A,\mathfrak{m})$ is a noetherian and
  $\mathfrak{m}$-adically complete local $R$-algebra, then the
  effectivity condition corresponds to the natural functor:
  \[
    \Hom_{\Spec A}(\cX_A,\cY_A) \to \varprojlim_n \Hom_{\Spec
      A}(\cX_{A/\mathfrak{m}^{n+1}}, \cY_{A})
  \]
  being an equivalence. The effectivity thus
  follows from coherent completeness (Corollary \ref{C:gb-c28}) and
  Tannaka duality (Corollary \ref{C:tannakian}). Conditions (5)
  and (6) on the coherence of the automorphisms, deformations, and
  obstructions follows from \cite{olsson-defn}, Proposition
  \ref{P:coh_gms}, and the discussion in
  \cite[\S9]{hallj_openness_coh} describing the $2$-term obstruction
  theory. A somewhat subtle point is that we do not deform the
  morphisms directly, but their graph, because \cite{olsson-defn} is
  only valid for representable morphisms. This proves that the stack
  $\underline{\Hom}_{X}(\cX,\cY)$ is algebraic and locally of finite
  presentation over $X$. The conditions on the diagonal follow from
  Corollary \ref{C:aff_hom_fund}, together with some standard
  manipulations of Weil restrictions.
\end{proof}
\begin{proof} [Proof of Corollary \ref{C:homG}]
 A $G$-equivariant morphism $Z\to \cX$ is equivalent to a morphism
 of stacks $[Z/G]\to [\cX/G]$ over $BG$. This gives the $2$-cartesian diagram
 \begin{equation*}
 \xymatrix{
  \underline{\Hom}_S^G(Z,\cX)\ar[r]\ar[d] & \underline{\Hom}_S\bigl([Z/G],[\cX/G]\bigr)\ar[d] \\
  S\ar[r] & \underline{\Hom}_S\bigl([Z/G],BG\bigr)\ar@{}[ul]|\square
 }
 \end{equation*}
 where the bottom map is given by the structure map $[Z/G]\to BG$
 and the right map is given by postcomposition with the structure map
 $[\cX/G]\to BG$. By Theorem \ref{T:hom}, the stacks
 $\underline{\Hom}_S\bigl([Z/G],[\cX/G]\bigr)$ and
 $\underline{\Hom}_S\bigl([Z/G],BG\bigr)$ are algebraic and locally
 of finite type over $S$. The latter always has quasi-affine diagonal and the former has quasi-affine diagonal when $\cX$ has separated diagonal. In particular,
 the bottom map is always quasi-affine. It follows that
 $\underline{\Hom}_S^G(Z,\cX)$ is always an algebraic stack locally
 of finite type over $S$ and has quasi-affine (in particular, quasi-compact and separated) diagonal whenever $\cX$ has quasi-affine diagonal; since $\cX$ is Deligne--Mumford and quasi-separated, this is equivalent to it having separated diagonal. Clearly, $\underline{\Hom}_S^G(Z,\cX)$ has no
 non-trivial infinitesimal automorphisms, hence is a Deligne--Mumford stack. Similarly, if $\cX$ is an algebraic space, then $\underline{\Hom}_S^G(Z,\cX)$ has no
 non-trivial automorphisms, hence is an algebraic space. 
\end{proof}

\subsection{Deligne--Mumford stacks with $\GG_m$-actions} \label{A:drinfeld}
Let $\cX$ be a quasi-separated Deligne--Mumford stack, locally of finite type over a field $k$ (not assumed to be algebraically closed), with an action of $\GG_m$.
Define the following stacks on $\Sch/k$:
\[
\begin{aligned}
\cX^0 		& := \underline{\Hom}^{\GG_m}(\Spec k, \cX) & \quad & \text{(the `fixed' locus)\footnotemark} \\
\cX^+ 		& := \underline{\Hom}^{\GG_m}(\AA^1,\cX) & \quad & \text{(the attractor)}
\end{aligned}
\]
where $\GG_m$ acts on $\AA^1$ by multiplication, and define the stack $\tilde{\cX}$ on $\Sch/\AA^1$ by
\[
\tilde{\cX} 	 := \underline{\Hom}_{\AA^1}^{\GG_m}(\AA^2 , \cX \times \AA^1),
\]
where $\GG_m$ acts on $\AA^2$ via $t \cdot (x,y) = (tx, t^{-1} y)$ and acts on $\AA^1$ trivially, and the morphism $\AA^2 \to \AA^1$ is defined by $(x,y) \mapsto xy$.

\footnotetext{If $\cX$ is an algebraic space, this is the fixed locus.  If $\cX$ is a Deligne--Mumford stack, we will define the {\it fixed locus} $\cX^{\GG_m}$ after allowing reparameterizations of the action; see Definition \ref{D:fixed-locus}.}

\begin{theorem}  \label{T:drinfeld}
With the hypotheses above, $\cX^0$ and $\cX^+$ are quasi-separated Deligne--Mumford stacks, locally of finite type over $k$.  Moreover, the natural morphism $\cX^0 \to \cX$ is a closed immersion, and the natural morphism $\ev_0\co \cX^+ \to \cX^0$ obtained by restricting to the origin is affine. In addition, $\tilde{\cX}$ is a Deligne--Mumford stack, locally of finite type over $k$, which is quasi-separated whenever $\cX$ has quasi-compact and separated diagonal (e.g., an algebraic space).
\end{theorem}

\begin{remark} When $\cX$ is an algebraic space, then $\cX^0$, $\cX^+$ and $\tilde{\cX}$ are algebraic spaces and the above result is due to Drinfeld   \cite[Prop.~1.2.2, Thm.~1.4.2 and Thm.~2.2.2]{drinfeld}.
\end{remark}

The algebraicity of $\cX^0$, $\cX^+$ and $\tilde{\cX}$ follows directly from Corollary \ref{C:homG}.  To establish the final statements, we will need to establish several preliminary results.

\begin{proposition} \label{P:A1-complete}  If $S$ is a noetherian affine scheme, then $[\AA^1_S / \GG_m]$ is coherently complete along $[S / \GG_m]$, where $S\inj \AA^1_S$ is the zero section.
\end{proposition}

\begin{proof}  
Let $A=\Gamma(S,\oh_S)$; then 
$\AA^1_S = \Spec A[t]$ and $V(t) = [S/\GG_m]$. If $\cF \in\Coh([\AA^1_S/\GG_m])$, then 
we claim that there exists an integer $n\gg 0$ such that the natural surjection 
$\Gamma(\cF) \to \Gamma(\cF/t^n\cF)$ is bijective. Now every coherent sheaf on  
$[\AA^1_S/\GG_m]$ is a quotient of a finite direct sum of coherent sheaves of the form 
$p^*\cE_l$, where $\cE_l$ is the weight $l$ representation of $\GG_m$ and $p\co 
[\AA^1_S/\GG_m] \to [S/\GG_m]$ is the natural map. It is enough to prove that
$\Gamma(p^*\cE_l) \to \Gamma(p^*\cE_l/t^n p^*\cE_l)$ is bijective, or equivalently,
that $\Gamma((t^n) \otimes p^*\cE_l) = 0$. But $(t^n) = p^*\cE_n$ and
$\Gamma(p^*\cE_{n+l})=0$ if $n+l>0$, hence for all $n\gg 0$.
We conclude that $\Gamma(\cF) \to \ilim_n \Gamma(\cF/t^n\cF)$ is bijective.
What remains can be proven analogously to Theorem \ref{key-theorem}. 
\end{proof}

\begin{proposition} \label{P:tannakian2}
Let $W$ be an excellent algebraic space over a field $k$ and let $G$ be an algebraic group acting on $W$.  Let $Z \subseteq W$ be a $G$-invariant closed subspace. Suppose that $[W/G]$ is coherently complete along $[Z/G]$.  Let $\cX$ be a noetherian algebraic stack over $k$ with affine stabilizers with an action of $G$.
Then the natural map
$$
\Hom^{G}(W, \cX) \to \ilim_n \Hom^{G}\bigl(W_Z^{[n]}, \cX\bigr)
$$
is an equivalence of groupoids.
\end{proposition}

\begin{proof}
As in the proof of Corollary~\ref{C:homG}, we have a cartesian diagram of groupoids
\[
\xymatrix{
  \Hom^G(W,\cX)\ar[r]\ar[d] & \Hom\bigl([W/G],[\cX/G]\bigr)\ar[d] \\
  {*}\ar[r] & \Hom\bigl([W/G],BG\bigr)
}
\]
and a similar cartesian diagram for $W$ replaced with $W_Z^{[n]}$ for any $n$ which gives the cartesian diagram
\[
\xymatrix{
  \ilim_n \Hom^G\bigl(W_Z^{[n]},\cX\bigr)\ar[r]\ar[d] & \ilim_n \Hom\bigl([W_Z^{[n]}/G],[\cX/G]\bigr)\ar[d] \\
  {*}\ar[r] & \ilim_n \Hom\bigl([W_Z^{[n]}/G],BG\bigr).
}
\]
Since $[W/G]$ is coherently complete along $[Z/G]$, it follows by Tannaka duality that the natural maps from the first square to the
second square are isomorphisms.
\end{proof}

\begin{proposition} \label{P:drinfeld-base-change} 
	If $f \co \cX \to \cY$ is an \'etale and representable $\GG_m$-equivariant morphism of quasi-separated Deligne--Mumford stacks of finite type over a field $k$, then
	$\cX^0 = \cY^0 \times_{\cY} \cX$ and $\cX^+ = \cY^+ \times_{\cY^0} \cX^0$.
\end{proposition}

\begin{proof}
  For the first statement, let $x \co S \to \cX$ be a morphism from a scheme $S$ such that the composition $f \circ x \co S \to \cY$ is $\GG_m$-equivariant.  To see that $x$ is $\GG_m$-equivariant, it suffices to base change $f$ by $S \to \cY$ and check that a section $S \to \cX_S$ of $\cX_S \to S$ is necessarily equivariant.  As $\cX_S \to S$ is \'etale and representable, $S \to \cX_S$ is an open immersion, and since any $\GG_m$-orbit in $\cX_S$ is necessarily connected, $S$ is an invariant open of  $\cX_S$.
  
For the second statement,
we need to show that there exists a unique $\GG_m$-equivariant morphism filling in the $\GG_m$-equivariant
diagram
\begin{equation} \label{D:drinfeld}
\begin{split}
\xymatrix{
\Spec k \times S \ar[r] \ar[d]				& \cX \ar[d]^f \\
\AA^1 \times S	\ar[r]	\ar@{-->}[ur]			& \cY
} \end{split} \end{equation}
where $S$ is an affine scheme of finite type over $k$, and the vertical left arrow is the inclusion of the origin.  For each $n \ge 1$, the formal lifting property of \'etaleness yields a unique $\GG_m$-equivariant map $\Spec(k[x]/x^n) \times S \to \cX$ such that 
$$\xymatrix{
\Spec k \times S \ar[r] \ar[d]				& \cX \ar[d]^f \\
\Spec (k[x]/x^n) \times S	\ar[r]	\ar@{-->}[ur]			& \cY
}$$
commutes.   By Propositions \ref{P:A1-complete} and \ref{P:tannakian2}, there exists a unique $\GG_m$-equivariant morphism $\AA^1 \times S \to \cX$ such that \eqref{D:drinfeld} commutes.
\end{proof}

\begin{remark} If $f \co \cX \to \cY$ is not representable, then it is not true in general that $\cX^0 = \cY^0 \times_{\cY} \cX$, e.g., let $f \co B\Gmu_n\to \Spec k$ where $\GG_m$ acts on $B \Gmu_n$ as in Remark \ref{R:not-split}. 
\end{remark}

\begin{remark} It is not true in general that $\cX^+$ is the fiber product of $f \co \cX \to \cY$ along the morphism $\ev_1 \co \cY^+ \to \cY$ defined by $\lambda \mapsto \lambda(1)$.    Indeed, consider the $\GG_m$-equivariant open immersion $\cX=\GG_m \hookrightarrow \AA^1=\cY$, where $\GG_m$ acts by scaling positively.  Then $\cY^+ = \cY$ but $\cX^+$ is empty.  \end{remark}

\begin{proof}[Proof of Theorem \ref{T:drinfeld}]  The algebraicity of $\cX^0$, $\cX^+$ and $\tilde{\cX}$ follows directly from Corollary \ref{C:homG}.  To verify the final statements, we may assume that $k$ is algebraically closed.  For any $\GG_m$-equivariant map $x \co \Spec k \to \cX$, the stabilizer $T_x$ of $x$ (as defined in \eqref{D:stab}) is $\GG_m$ and the map on quotients $B\GG_m \to \cY := [\cX/\GG_m]$ induces a map $\GG_m \to G_y$ on stabilizers providing a splitting of \eqref{E:stab}.  Our generalization of Sumihiro's theorem (Theorem \ref{T:sumi1}) provides an \'etale $\GG_m$-equivariant neighborhood $(\Spec A,u) \to (\cX,x)$.  Proposition \ref{P:drinfeld-base-change} therefore reduces the statements to the case of an affine scheme, which can be established directly; see \cite[\S 1.3.4]{drinfeld}.
\end{proof}

We will now investigate how $\cX^0$ and $\cX^+$ change if we reparameterize the
torus. We denote by $\cX_{\langle d \rangle}$  the Deligne--Mumford stack $\cX$ with $\GG_m$-action induced by the reparameterization $\GG_m \xrightarrow{d} \GG_m$.  For integers $d \, | \, d'$, there are maps $\cX^0_{\langle d \rangle} \to \cX^0_{\langle d' \rangle}$ and $\cX^+_{\langle d \rangle} \to \cX^+_{\langle d' \rangle}$ (defined by precomposing with $\AA^1 \to \AA^1, x \mapsto x^{d'/d}$) that are compatible with the natural maps to $\cX$.

Recall from Theorem~\ref{T:drinfeld} that $\cX^0_{\langle d \rangle}\to \cX$ is a closed immersion. Also recall that for $x \in \cX(k)$,  there is an exact sequence
\begin{equation} \label{E:stab2}
\xymatrix{1 \ar[r] & G_x \ar[r] &  G_y \ar[r] &  T_x \ar[r] & 1,}
\end{equation}
where $y$ is the image of $x$ in $\cY = [\cX/\GG_m]$; and $T_x := \GG_m \times_{\cX} BG_x \subset \GG_m$ is the stabilizer, where $\GG_m  \xrightarrow{\id \times x} \GG_m \times \cX \xrightarrow{\sigma_{\cX}} \cX$ is the restriction of the action map; see also \eqref{D:stab}.

\begin{proposition} \label{P:fixed-point}
Let $\cX$ be a quasi-separated Deligne--Mumford stack, 
locally of finite type over a field $k$, with an action of $\GG_m$.
Let $x\in \cX(k)$ and let $T_x \subset \GG_m$ be its stabilizer.
\begin{enumerate}
	\item \label{P:fixed-point1}
		$x \in \cX^0$ if and only if $T_x = \GG_m$ and \eqref{E:stab2} splits.
	\item \label{P:fixed-point2}
		The following conditions are equivalent:
                \begin{enumerate*}
                \item \label{P:fixed-point2:reparam2a}
                  $x \in \cX^0_{\langle d \rangle}$ for sufficiently
                  divisible integers $d$; 
		\item \label{P:fixed-point2:reparam2b} 
			$T_x = \GG_m$; and
		\item  \label{P:fixed-point2:reparam2c} 
			$\dim G_y = 1$.
	\end{enumerate*}
\end{enumerate}
\end{proposition}
\begin{proof}
For \eqref{P:fixed-point1}, if $x \co \Spec k \to \cX$ is $\GG_m$-equivariant, then clearly $T_x = \GG_m$ and the map on quotients $B\GG_m \to \cY := [\cX/\GG_m]$ induces a map $\GG_m \to G_y$ on stabilizers providing a splitting of \eqref{E:stab2}.  Conversely, a section $\GG_m \to G_y$ providing a splitting of \eqref{E:stab2} induces a section $B \GG_m \to BG_y$ of $BG_y \to B \GG_m$, and taking the base change of the composition $B \GG_m \to BG_y \to \cY \to B\GG_m$ along $\Spec k \to B \GG_m$ induces a unique $\GG_m$-equivariant map $\Spec k \to \cX$.

For \eqref{P:fixed-point2}, it is clear that \itemref{P:fixed-point2:reparam2b} and \itemref{P:fixed-point2:reparam2c} are equivalent, and that they are implied by \itemref{P:fixed-point2:reparam2a}.  On the other hand, if \itemref{P:fixed-point2:reparam2b} holds, then the sequence \eqref{E:stab2} splits after reparameterizing the action by $\GG_m \xrightarrow{d} \GG_m$ for sufficiently divisible integers $d$ (see proof of Theorem~\ref{T:sumi2}). It now follows from \eqref{P:fixed-point1} that $x \in \cX^0_{\langle d \rangle}$.
\end{proof}

\begin{proposition} \label{P:reparam}
Let $\cX$ be a quasi-separated Deligne--Mumford stack, 
locally of finite type over a field $k$, with an action of $\GG_m$.
\begin{enumerate}
	\item  \label{P:reparam1}
	For $d \,| \, d'$, the map $\cX^0_{\langle d \rangle} \to \cX^0_{\langle d' \rangle}$ is an open and closed immersion, and $\cX^+_{\langle d \rangle} = \cX^+_{\langle d' \rangle} \times_{\cX^0_{\langle d' \rangle} } \cX^0_{\langle d \rangle}$.
	\item  \label{P:reparam3}
	If $\cX$ is quasi-compact, then for sufficiently divisible integers $d$ and $d'$, $\cX^0_{\langle d \rangle} = \cX^0_{\langle d' \rangle}$ and $\cX^+_{\langle d \rangle} = \cX^+_{\langle d' \rangle}$.
\end{enumerate}
\end{proposition}

\begin{proof} We may assume $k$ is algebraically closed.
For \eqref{P:reparam1}, since $\cX^0_{\langle d \rangle} \to \cX$ is a closed immersion for all $d$ (Theorem \ref{T:drinfeld}), we see that $\cX^0_{\langle d \rangle} \to \cX^0_{\langle d' \rangle}$ is a closed immersion.   For any $x \in \cX^0(k)$, Theorem \ref{T:sumi1} provides an \'etale $\GG_m$-equivariant morphism $(U, u) \to (\cX, x)$ where $U$ is an affine scheme.  By Proposition \ref{P:drinfeld-base-change}, the maps $\cX^0 \to \cX^0_{\langle d \rangle}$ and $\cX^+ \to \cX^+_{\langle d \rangle}$ pullback to the isomorphisms $U^0 \to U^0_{\langle d \rangle}$ and $U^+ \to U^+_{\langle d \rangle}$. This shows both that $\cX^0 \to \cX^0_{\langle d \rangle}$ is an open immersion and that $\cX^+ = \cX^+_{\langle d \rangle} \times_{\cX^0_{\langle d \rangle} } \cX^0$, which implies \eqref{P:reparam1}.

For \eqref{P:reparam3}, the locus of points in $\cY=[\cX/\GG_m]$ with a positive dimensional stabilizer is a closed substack.  It follows from Proposition \ref{P:fixed-point}\eqref{P:fixed-point2} that the locus of points in $\cX$ contained in $\cX^0_{\langle d \rangle}$ for some $d$ is also closed.  In particular, the substacks $\cX^0_{\langle d \rangle}$ stabilize for sufficiently divisible integers $d$ and by \eqref{P:reparam1} the stacks $\cX^+_{\langle d \rangle}$ also stabilize.
\end{proof}

Proposition \ref{P:reparam}\eqref{P:reparam3} justifies the following definition.

\begin{definition} \label{D:fixed-locus}
Let $\cX$ be a quasi-separated Deligne--Mumford stack, 
locally of finite type over a field $k$, with an action of $\GG_m$.  The {\it fixed locus} is the closed substack of $\cX$ defined as
$$\cX^{\GG_m} := \bigcup_d \cX_{\langle d \rangle}^0.$$
\end{definition}

\begin{remark}
Consider the action of $\GG_m$ on $\cX=B \Gmu_n$ as in Remark \ref{R:not-split}.  Then $\cX^0$ is empty but $\cX_{\langle d \rangle}^0 = \cX$ for all integers $d$ divisible by $n$.  Thus, $\cX^{\GG_m} = \cX$.
\end{remark}

\subsection{Bia\l ynicki-Birula decompositions for Deligne--Mumford stacks} \label{A:BB}

We provide the following theorem establishing the existence of Bia\l ynicki-Birula decompositions for a Deligne--Mumford stack $\cX$.  Our proof relies on the algebraicity of the stacks $\cX^0 = \underline{\Hom}^{\GG_m}(\Spec k, \cX)$ and $\cX^+ = \underline{\Hom}^{\GG_m}(\AA^1,\cX)$ (Theorem \ref{T:drinfeld}) and the existence of $\GG_m$-equivariant \'etale affine neighborhoods (Theorem \ref{T:sumi1}).  In particular, our argument recovers the classical result from \cite[Thm.~4.1]{bb}.   Due to subtleties arising from group actions on stacks, the proof is substantially simpler in the case that $\cX$ is an algebraic space, and the reader may want to consider this special case on a first reading.

\begin{theorem} \label{T:bb}
Let $\cX$ be a separated Deligne--Mumford stack, of finite type over an arbitrary field $k$, with an action of $\GG_m$.  Let $\cX^{\GG_m} = \coprod_i \cF_i$ be the fixed locus (see Definition \ref{D:fixed-locus}) with connected components $\cF_i$.  There exists an affine morphism $\cX_i \to \cF_i$ for each $i$ and a monomorphism $\coprod_i \cX_i \to \cX$.  Moreover,
\begin{enumerate}
\item \label{T:bb-proper}
	If $\cX$ is proper, then $\coprod_i \cX_i \to \cX$ is surjective.
\item \label{T:bb-smooth}
	If $\cX$ is smooth, then $\cF_i$ is smooth and $\cX_i \to \cF_i$ is an affine fibration (i.e., $\cX_i$ is affine space \'etale locally over $\cF_i$).
\item \label{T:bb-locally-closed}
	Let $\cX \to X$ be the coarse moduli space.
\begin{enumerate}
\item \label{T:bb-locally-closed-affine}
	If $X$ is affine, then $\cX_i \hookrightarrow \cX$ is a  closed immersion.
\item \label{T:bb-locally-affine}
	If $X$ has a $\GG_m$-equivariant affine open cover (e.g., $X$ is a normal scheme), then
\begin{enumerate}[label=(\roman*),ref=\roman*]
	\item \label{T:bb-locally-affine1} 
	$\cX_i \hookrightarrow \cX$ is a local immersion 
	(i.e., a locally closed immersion Zariski-locally on the source) and $\cX_i \to \cX \times \cF_i$ is a locally closed immersion; and
	\item \label{T:bb-locally-affine2} 
	if $\cZ \subset \cX_i$ is an irreducible component, then $\cZ \hookrightarrow \cX$ is a locally closed immersion.
\end{enumerate}
\item \label{T:bb-locally-closed-smooth-stack}
	If $\cX$ is smooth and $X$ is a scheme, then $\cX_i \hookrightarrow \cX$ is a locally closed immersion.
\item \label{T:bb-locally-closed-quasi-projective}
	If there exists a $\GG_m$-equivariant locally closed immersion $X \hookrightarrow \PP(V)$ where $V$ is a $\GG_m$-representation (e.g., $\cX$ normal and $X$ is quasi-projective), then $\cX_i \hookrightarrow \cX$ is a locally closed immersion.
\end{enumerate}
\end{enumerate}
\end{theorem}

\begin{remark} If $\cX$ is a smooth scheme and $k$ is algebraically closed, then this statement (except Case \eqref{T:bb-locally-affine}) is the classical Bia\l{}ynicki-Birula decomposition theorem \cite[Thm.~4.1]{bb} (using Sumihiro's theorem \cite[Cor.~2]{sumihiro} ensuring that $\cX$ has a $\GG_m$-equivariant affine open cover).
 If $\cX$ is an algebraic space, then this 
was established in  \cite[Thm.~B.0.3]{drinfeld} (except Case \eqref{T:bb-locally-affine}).  Our formulation of Case \eqref{T:bb-locally-affine}\eqref{T:bb-locally-affine1} was motivated by \cite[Thm.~4.5, p.~69]{hesselink}  and \cite[Prop.~13.58]{milne} and
Case \eqref{T:bb-locally-affine}\eqref{T:bb-locally-affine2} was motivated by \cite[Prop.~7.6]{js-bb}.

Using Drinfeld's results and our Theorem \ref{T:sumi1}, {Jelisiejew} and {Sienkiewicz} establish the theorem above when $\cX$ is an algebraic space as a special case of \cite[Thm.~1.5]{js-bb} and their proof in particular  recovers  the main result of \cite{bb}.  Our proof follows a similar strategy by relying on results of the previous section and Theorem \ref{T:sumi1} to reduce to the affine case.
\end{remark}

\begin{remark} It is not true in general that $\cX_i \hookrightarrow \cX$ is a locally closed immersion.
\begin{enumerate}
\item The condition in \eqref{T:bb-locally-closed-smooth-stack} that $X$ is a scheme is necessary.    Sommese has given an example of a smooth algebraic space $X$ such that $X_i \hookrightarrow X$ is not a locally closed immersion \cite{sommese}. This is based on Hironaka's example of a proper, non-projective, smooth 3-fold.
\item The condition in \eqref{T:bb-locally-closed-quasi-projective} that $X$ is quasi-projective is necessary.  Konarski has provided an example of a normal proper scheme $X$ (a toric variety) such that $X_i \hookrightarrow X$ is not a locally closed immersion \cite{konarski}.
\end{enumerate}

For a smooth Deligne--Mumford stack $\cX$ with a $\GG_m$-action, \cite[Prop.~5]{oprea} states that the existence of a Bia\l ynicki-Birula decomposition with each $\cX_i \hookrightarrow \cX$ locally closed follows from the existence of a $\GG_m$-equivariant, \'etale atlas $\Spec A \to \cX$ (as provided by Theorem \ref{T:sumi2}).  The counterexamples above show that \cite[Prop.~5]{oprea} is incorrect.  Nevertheless, the main result \cite[Thm.\ 2]{oprea} still holds as a consequence of Theorem~\ref{T:bb} since the Deligne--Mumford stack $\overline{\cM}_{0,n}(\PP^r,d)$ of stable maps is smooth and its coarse moduli space is a scheme (Case \eqref{T:bb-locally-closed-smooth-stack}).

Moreover, \cite[Thm.~3.5]{skowera} states the above theorem in the case that $\cX$ is a smooth, proper and tame Deligne--Mumford stack with $X$ a scheme but the proof is not valid as it relies on \cite[Prop.~5]{oprea}.  A similar error appeared in a previous version of our article where it was claimed incorrectly that $\cX_i \hookrightarrow \cX$ is a locally closed immersion for any smooth, proper Deligne--Mumford stack.
\end{remark}

\begin{remark} Let $X$ be a separated scheme of finite type over $k$ with finite quotient singularities and with a $\GG_m$-action.  There is a canonical smooth Deligne--Mumford stack $\cX$ whose coarse moduli space is $X$ (see \cite[\S 4.1]{fantechi-mann-nironi}).  The $\GG_m$-action lifts canonically to $\cX$.  Applying 
Theorem \ref{T:bb}\eqref{T:bb-locally-closed-smooth-stack} to $\cX$ and appealing to Proposition \ref{P:cms}\eqref{P:cms-climmersion-nilimmersion-both}, we can conclude that the components $X_i$ of $X^+$ are locally closed in $X$.
\end{remark}

The following proposition establishes properties of the evaluation map $\ev_1 \co \cX^+ \to \cX, \lambda \mapsto \lambda(1)$ in terms of properties of $\cX$.  We  find it prudent to state a relative version that for a given morphism $f \co \cX \to \cY$ between Deligne--Mumford stacks
establishes properties of the relative evaluation map
$$\ev_f \co \cX^+ \to \cY^+ \times_{\cY} \cX, \lambda \mapsto (f \circ \lambda, \lambda(1))$$
in terms of properties of $f$. In the proof of Theorem~\ref{T:bb} we will only use the absolute case where $\cY=\Spec k$.

\begin{proposition} \label{P:ev1}
Let $f \co \cX \to \cY$ be a $\GG_m$-equivariant morphism of quasi-separated Deligne--Mumford stacks that are
locally of finite type over an arbitrary field $k$.
\begin{enumerate}
	\item \label{P:ev1-unramified}
	$\cX^+ \to \cY^+ \times_{\cY} \cX$ is unramified.		
	\item \label{P:ev1-representable}
	If $f \co \cX \to \cY$ has separated diagonal, then $\cX^+ \to \cY^+ \times_{\cY} \cX$ is representable.
	\item \label{P:ev1-monomorphism}
	 If $f \co \cX \to \cY$ is separated, then $\cX^+ \to \cY^+ \times_{\cY} \cX$ is a monomorphism.
	\item \label{P:ev1-surjective}
	If $f \co \cX \to \cY$ is proper and $\cY$ is quasi-compact, then $\cX_{\langle d \rangle}^+ \to  \cY_{\langle d \rangle}^+ \times_{\cY} \cX$ is surjective for sufficiently divisible integers $d$.
\end{enumerate}
 \end{proposition}
 
\begin{proof}
We may assume that $k$ is algebraically closed.
For \eqref{P:ev1-unramified}, it suffices to show that $\cX^+ \to \cX$ is unramified.  We follow the argument of \cite[Prop.~1.4.11(1)]{drinfeld}.  We need to check that for any $(\AA^1 \xrightarrow{\lambda} \cX) \in \cX^+(k)$, the induced map $T_{\lambda} \cX^+ \to T_{\lambda(1)} \cX$ on tangent spaces is injective.  This map can be identified with the restriction map
\[
\Hom_{\oh_{\AA^1}}^{\GG_m}(\lambda^* \Omega_{\cX}^1, \oh_{\AA^1}) \to \Hom_{\oh_{\AA^1 \setminus 0}}^{\GG_m}((\lambda^* \Omega_{\cX}^1)|_{\AA^1 \setminus 0}, \oh_{\AA^1 \setminus 0}),
\]
which is clearly injective.

For \eqref{P:ev1-representable}, let  $(\lambda \co \AA^1 \to \cX) \in \cX^+(k)$. Automorphisms $\tau_1, \tau_2 \in \Aut_{\cX^+}(\lambda)$ mapping to the same automorphism of $\ev_f(\lambda)$ induce two sections $\tilde{\tau}_1, \tilde{\tau}_2$ of $\Isom_{\cX/\cY}(\lambda) \to \AA^1$ agreeing over $\AA^1 \setminus 0$.  The valuative criterion for separatedness implies that $\tilde{\tau}_1 = \tilde{\tau}_2$ and thus $\tau_1 = \tau_2$.

For \eqref{P:ev1-monomorphism}, since  $\cX^+ \to \cY^+ \times_{\cY} \cX$  is unramified and representable, it is enough to prove that is universally injective. A $k$-point of $\cY^+ \times_{\cY} \cX$ with two preimages in $\cX^+$ corresponds to a $\GG_m$-equivariant $2$-commutative square
$$\xymatrix{
\AA^1 \setminus 0 \ar@{^(->}[d] \ar[r]								&  	\cX \ar[d]^f\\
\AA^1  \ar[r] \ar@<0.5ex>[ur]^{h_1} \ar@<-0.5ex>[ur]_{h_2}	& \cY
}$$
with two $\GG_m$-equivariant lifts $h_1, h_2 \co \AA^1 \to \cX$. We need to produce a $\GG_m$-equivariant $2$-isomorphism $h_1\iso h_2$. As $f \co \cX \to \cY$ is separated, $I:=\Isom_{\cX/\cY}(h_1, h_2) \to \AA^1$ is proper. The 2-isomorphism $h_1|_{\AA^1 \setminus 0} \iso h_2|_{\AA^1 \setminus 0}$ gives a $\GG_m$-equivariant section of $I\to \AA^1$ over $\AA^1 \setminus 0$ and the closure of its graph gives a $\GG_m$-equivariant section of $I\to \AA^1$, i.e., a $\GG_m$-equivariant 2-isomorphism $h_1 \iso h_2$.

For \eqref{P:ev1-surjective}, by Proposition \ref{P:reparam}\eqref{P:reparam3}, it suffices to show that a $\GG_m$-equivariant commutative diagram
$$\xymatrix{
\AA^1 \setminus 0 \ar[r] \ar@{^(->}[d]				& \cX \ar[d]^f \\
\AA^1 \ar[r]^{\lambda} \ar@{-->}[ur]^{\lambda'}						& \cY
}$$
of solid arrows admits a $\GG_m$-equivariant lift $\lambda' \co \AA^1 \to \cX$ after
reparameterizing the action.  Let $x \in \cX(k)$ be the image of $1$ under $\AA^1 \setminus 0 \to \cX$.  After replacing $\cX$ with the closure of $\im(\AA^1 \setminus 0 \to \cX)$ we may assume that $\cX$ is integral  of dimension $\le 1$.  Let $\cX' \to \cX$ be the normalization and choose a preimage $x'$ of $x$.  Since $\cX' \to \cY$ is proper, the induced map $\AA^1 \setminus 0 \to \cX'$, defined by $t \mapsto t \cdot x'$, admits a unique lift $h \co C \to \cX'$ compatible with $\lambda$ after a ramified extension $(C,c) \to (\AA^1,0)$.   Let $x'_0 = h(c) \in \cX'(k)$.

 By Theorem \ref{T:sumi2}, there exists $d > 0$ and a $\GG_m$-equivariant map $(\Spec A,w_0) \to (\cX'_{\langle d \rangle}, x'_0)$ with $w_0$ fixed by $\GG_m$.  If $\dim \cX = 0$, then we may assume $A=k$ and so in this case the composition $\AA^1 \to \Spec k \xrightarrow{x} \cX_{\langle d \rangle}$ gives the desired map.  If $\dim \cX = 1$, then we may assume that $\Spec(A)$ is a smooth and irreducible affine curve with two orbits---one open and one closed.  It follows that $\Spec(A)$ is $\GG_m$-equivariantly isomorphic to $\AA^1$ and the composition $\AA^1 \to \cX'_{\langle d \rangle} \to \cX_{\langle d \rangle}$ gives the desired map.
\end{proof}

\begin{proposition} \label{P:cms}
Let $\pi \co \cX \to \cY$ be a $\GG_m$-equivariant morphism
of quasi-separated 
Deligne--Mumford stacks, 
of finite type over an arbitrary field $k$.  If $\pi$ is proper and quasi-finite (e.g., $\pi$ is a coarse moduli space), then
\begin{enumerate}
	\item \label{P:cms-proper}
          $\cX^+ \to \cY^+$ is proper and
        \item \label{P:cms-climmersion-nilimmersion-both} 
          the maps $\cX^0_{\langle d \rangle} \to \cY^0_{\langle d \rangle} \times_{\cY} \cX$ and $\cX^+_{\langle d\rangle} \to \cY^+_{\langle d \rangle} \times_{\cY} \cX$ are closed immersions for all $d>0$ and nilimmersions for $d$ sufficiently divisible. In particular, $\cX^{\GG_m} \to \cY^{\GG_m} \times_{\cY} \cX$ is a nilimmersion.
\end{enumerate}
\end{proposition}

\begin{proof}
For \eqref{P:cms-proper} it is enough to prove that $\ev_\pi\colon
\cX^+ \to \cY^+\times_\cY \cX$ is proper. First observe that $\cX^+$ and $\cY^+$ are quasi-compact: via $\ev_0$, they are affine over $\cX$ and $\cY$, respectively (Theorem \ref{T:drinfeld}).  Since $\cX$ and $\cY$ are quasi-separated, it follows that $\ev_\pi \colon \cX^+ \to \cY^+\times_{\cY} \cX$ is quasi-compact. We may now use the valuative criteria to verify that $\ev_\pi$ is proper.  Let
\begin{equation} \label{E:cms-proper}
\begin{split}
\xymatrix{
\Spec K \ar[d] \ar[r]	& \cX^+ \ar[d]^{\ev_\pi} \\
\Spec R \ar[r] \ar@{-->}[ur]	& \cY^+\times_\cY \cX
}
\end{split} 
\end{equation}
be a diagram of solid arrows where $R$ is a DVR with fraction field $K$.   This corresponds to a $\GG_m$-equivariant diagram
\begin{equation} \label{E:cms-proper2}
\begin{split}
\xymatrix{
\AA^1_R \setminus 0 \ar[r]	\ar@{^(->}[d]	& \cX \ar[d]^{\pi} \\
\AA^1_R \ar[r]	\ar@{-->}[ur]			& \cY
}
\end{split} 
\end{equation}
of solid arrows, where $0 \in \AA^1_R$ denotes the unique closed point fixed by $\GG_m$. Equivalently, we have a diagram
\begin{equation} \label{E:cms-proper3}
\begin{split}
\xymatrix{
[(\AA^1_R \setminus 0)/\GG_m] \ar[r]	\ar@{^(->}[d]	& [\cX/\GG_m] \ar[d]^{\pi} \\
[\AA^1_R/\GG_m] \ar[r]	\ar@{-->}[ur]			& [\cY/\GG_m]
}
\end{split} 
\end{equation}
of solid arrows. A dotted arrow providing a lift of \eqref{E:cms-proper3} is the same as a $\GG_m$-equivariant dotted arrow providing a lift of \eqref{E:cms-proper2} or a lift in \eqref{E:cms-proper}. Now let $U=[(\AA_R^1\setminus 0)/\GG_m]$ and $S= [\AA^1_R/\GG_m]$. By base change, a dotted arrow providing a lift of \eqref{E:cms-proper3} is the same as a section to the projection $S\times_{[\cY/\GG_m]}[\cX/\GG_m] \to S$ extending the induced section over $U$. Since $S$ is regular and $U$ contains all points of $S$ of codimension $1$, we may apply Lemma \ref{L:extension}\eqref{L:extension-section} to deduce the claim. (It is worthwhile to point out that the existence of lifts in \eqref{E:cms-proper3} is equivalent to the map $[\cX/\GG_m] \to [\cY/\GG_m]$ being $\Theta$-reductive, as introduced in \cite{dhl-instability}, and that Lemma \ref{L:extension}\eqref{L:extension-section}  implies that $[\cX/\GG_m] \to [\cY/\GG_m]$  is $\Theta$-reductive.)

For \itemref{P:cms-climmersion-nilimmersion-both}, Theorem
\ref{T:drinfeld} implies that $\cX^0_{\langle d \rangle} \to \cX$ and
$\cY^0_{\langle d \rangle} \to \cY$ are closed immersions for all $d>0$;
thus,
$\cX^0_{\langle d \rangle} \to \cY^0_{\langle d \rangle} \times_{\cY}
\cX$ is a closed immersion. For $d$ sufficiently divisible, it is now
easily checked using the quasi-finiteness of $\pi$ that the morphisms in question are also surjective. Also, for all $d>0$, the map $\cX_{\langle d \rangle}^+ \to \cY_{\langle d \rangle}^+ \times_{\cY} \cX$ is proper by
 \eqref{P:cms-proper} and a monomorphism by Proposition \ref{P:ev1}\eqref{P:ev1-monomorphism}, and thus a closed immersion.  The surjectivity of $\cX_{\langle d \rangle}^+ \to \cY_{\langle d \rangle}^+ \times_{\cY} \cX$ for sufficiently divisible $d$ follows from Proposition \ref{P:ev1}\eqref{P:ev1-surjective}.
\end{proof}

\begin{lemma} \label{L:extension}
Let $S$ be a regular algebraic stack and let $U \subset S$ be an open substack containing all points of codimension $1$.  Let $f \co \cX \to S$ be a quasi-finite morphism that is relatively Deligne--Mumford.
\begin{enumerate}
	\item \label{L:extension-etale}
	If $f|_U \co f^{-1}(U) \to U$ is \'etale, then $f \co \normin{\cX}{U} \to S$ is \'etale, where $\normin{\cX}{U}$ denotes the normalization of $\cX$ in $f^{-1}(U)$. 
	\item \label{L:extension-section}
	If $f \co \cX \to S$ is proper and $f|_U$ has a section, then $f \co \cX \to S$ has a section.
\end{enumerate}
\end{lemma}

\begin{proof}
For \eqref{L:extension-etale}, as the question is smooth-local on $S$ and \'etale-local on $\cX$, we may assume that $\cX$ and $S$ are irreducible schemes. Now the statement follows from Zariski--Nagata purity \cite[Exp.~X, Cor.~3.3]{sga1}.  For \eqref{L:extension-section}, by Zariski's Main Theorem \cite[Thm.~16.5(ii)]{MR1771927}, we may factor a section $U \to \cX$ as $U \hookrightarrow \cV \to \cX$ where $U \hookrightarrow {\cV}$ is a dense open immersion and ${\cV} \to \cX$ is a finite morphism. Since $\cV\to S$ is proper, quasi-finite and an isomorphism over $U$, it follows that $\normin{{\cV}}{U} \to S$ is proper and \'etale by \eqref{L:extension-etale}.  As $I_{\normin{{\cV}}{U}/S} \to \normin{{\cV}}{U}$ is finite, \'etale and generically an isomorphism, it is an isomorphism and we conclude that $\normin{{\cV}}{U}\to S$ is representable.  Then $\normin{{\cV}}{U} \to S$ is finite, \'etale and generically an isomorphism, thus an isomorphism.  
\end{proof}

\begin{proof}[Proof of Theorem \ref{T:bb}]
After reparameterizing the action by $\GG_m \xrightarrow{d} \GG_m$ for $d$ sufficiently divisible, we may assume that $\cX^0 = \cX^{\GG_m}$ (Proposition \ref{P:reparam}).  Theorem \ref{T:drinfeld} yields quasi-separated Deligne--Mumford stacks $\cX^0 =  \underline{\Hom}^{\GG_m}(\Spec k, \cX)$ and $\cX^+ = \underline{\Hom}^{\GG_m}(\AA^1,\cX)$, locally of finite type over $k$, such that the morphism $\ev_0 \co \cX^+ \to \cX^0$ is affine.  Let $\cX^0 = \coprod_i \cF_i$ be the decomposition into connected components and set $\cX_i := \ev_0^{-1}(\cF_i)$.  Since $\cX$ is separated, $\coprod_i \cX_i \to \cX$ is a monomorphism (Proposition \ref{P:ev1}\eqref{P:ev1-monomorphism}).  This establishes the main part of the theorem. Part \eqref{T:bb-proper} follows from Proposition \ref{P:ev1}\eqref{P:ev1-surjective}.

We now establish \eqref{T:bb-smooth} and \eqref{T:bb-locally-closed} in stages of increasing generality.  If $\cX$ is an affine space with a linear $\GG_m$-action, then it is easy to see that $\cX^+ \to \cX^0$ is a projection of linear subspaces.  If $\cX = \Spec(A)$ is affine and  $A = \bigoplus_{d} A_d$ denotes the induced $\ZZ$-grading, then a direct calculation shows that $\cX^+ = V(\sum_{d < 0} A_d)$ and $\cX^0 = V(\sum_{d \neq 0} A_d)$ are closed subschemes; see \cite[\S 1.3.4]{drinfeld}.

To see \eqref{T:bb-smooth}, we may assume that $k$ is algebraically closed.  When $\cX=\Spec A$ is affine, let $x \in \cX^0(k)$ be a fixed point defined by a maximal ideal $\fm \subset A$.  The surjection $\fm \to \fm/\fm^2$ admits a $\GG_m$-equivariant section which in turn induces a morphism $\cX \to T_{\cX,x} = \Spec(\Sym \fm/\fm^2)$ which is \'etale at $x$, and \eqref{T:bb-smooth} follows from  \'etale descent using the case above of affine space and Proposition \ref{P:drinfeld-base-change}.
In general, Proposition \ref{P:fixed-point} and Sumihiro's theorem (Theorem \ref{T:sumi1}) imply that any point of $\cX^0$ has an equivariant affine \'etale neighborhood and thus Proposition \ref{P:drinfeld-base-change} reduces \eqref{T:bb-smooth} to the case of an affine scheme.

For  \eqref{T:bb-locally-closed}, let $X^+$ and $X_i$ be the coarse moduli spaces of $\cX^+$ and $\cX_i$.  For \eqref{T:bb-locally-closed-affine}, the above discussion shows that since $X$ is affine, $X^+ \to X$ is a closed immersion. Since $\cX^+ \to  X^+ \times_{X} \cX$ is a nilimmersion (Proposition \ref{P:cms}\eqref{P:cms-climmersion-nilimmersion-both}), $\cX^+ \to \cX$ is also a closed immersion.  

For \eqref{T:bb-locally-affine}, by Proposition \ref{P:cms}\eqref{P:cms-climmersion-nilimmersion-both}, we may assume that $\cX=X$.  For any point $x \in X^+$, let $x_0$ be the image of $x$ under $\ev_0 \co X^+ \to X^0$, and choose a $\GG_m$-invariant affine open neighborhood $U \subset X$ of $x_0$. 
This induces a diagram
\begin{equation} \begin{split} \label{E:locally-affine}
\xymatrix{
U^+ \ar@{^(->}[r] \ar@{^(->}[rd]		& \ev_1^{-1}(U) \ar[r] \ar@{^(->}[d]		& U \ar@{^(->}[d] \\
				& X^+ \ar[r]^{\ev_1}				& X .
}
\end{split} \end{equation}
Since $U^+ \to U$ is a closed immersion (as $U$ is affine) and $X^+ \to X$ is separated (it is a monomorphism), $U^+ \to \ev_1^{-1}(U)$ is a closed immersion.  Since $U^+ = X^+ \times_{X^0} U^0$ (Proposition \ref{P:drinfeld-base-change}), $x \in U^+$ and $U^+ \to X^+$ is an open immersion.  In particular, $U^+ \subset \ev_1^{-1}(U)$ is an open and closed subscheme containing $x$. 

For \eqref{T:bb-locally-affine1}, for any $x \in X^+$, we observe from Diagram \eqref{E:locally-affine} that $U^+ \to X^+ \to X$ is a locally closed immersion.  Moreover, $U \times U^0$ is an open neighborhood of $(\ev_1(x),x_0)$.  Since the restriction of $(\ev_1, \ev_0) \co X^+ \to X \times X^0$ over $U \times U^0$ is the closed immersion $U^+ \to U \times U^0$, it follows that $X^+ \to X \times X^0$ is a locally closed immersion.

For \eqref{T:bb-locally-affine2}, let $Z \subset X^+$ be an irreducible component and $x \in Z$.  Then $Z \cap U^+$ is a nonempty open and closed subscheme of the irreducible scheme  $Z \cap \ev_1^{-1}(U)$.  This shows that $Z \cap U^+ = Z \cap \ev_1^{-1}(U)$ and that $Z \cap \ev_1^{-1}(U) \to U$ is a closed immersion.  It follows that $Z \hookrightarrow X^+ \xrightarrow{\ev_1} X$ is a locally closed immersion.

For \eqref{T:bb-locally-closed-smooth-stack}, observe that \eqref{T:bb-smooth} implies that $\cX_i$ is smooth and connected,  thus irreducible.  Since the coarse moduli space $X$ is necessarily normal and thus admits a $\GG_m$-equivariant open affine cover by Sumihiro's theorem \cite[Cor.~2]{sumihiro}, the conclusion follows from \eqref{T:bb-locally-affine}\eqref{T:bb-locally-affine2}.

For \eqref{T:bb-locally-closed-quasi-projective}, it suffices to show that $X_i \hookrightarrow X$ is a locally closed immersion by Proposition \ref{P:cms}\eqref{P:cms-climmersion-nilimmersion-both}.  This statement is easily reduced to the case of $X=\PP(V)$, using a special case of Proposition \ref{P:cms}\eqref{P:cms-climmersion-nilimmersion-both}. For $X=\PP(V)$ a direct calculation shows that each $X_i$ is of the form $\PP(W) \setminus \PP(W')$ for linear subspaces $W' \subset W \subset V$.
\end{proof}

\appendix 

\section{Equivariant Artin algebraization} \label{A:algebraization} \label{A:approx} 

In this appendix, we give an equivariant generalization of Artin's algebraization theorem \cite[Thm.~1.6]{artin-algebraization}.  We follow the approach of \cite{conrad-dejong} using Artin approximation and an effective version of the Artin--Rees lemma.

The main results of this appendix (Theorems \ref{T:approximate-algebraization} and \ref{T:algebraization}) are formulated in greater generality than necessary to prove Theorem \ref{T:field}.  We feel that these results are of independent interest and will have further applications.  In particular, in the subsequent article \cite{ahr2} we will apply the results of this appendix to prove a relative version of Theorem \ref{T:field}.   

\subsection{Good moduli space morphisms are of finite type}
Let $G$ be a group acting on a noetherian ring~$A$.
Goto--Yamagishi~\cite{MR706507} and Gabber~\cite[Exp.~IV,
  Prop.~2.2.3]{MR3309086} have proven that $A$ is finitely generated over $A^G$
when $G$ is either diagonalizable (Goto--Yamagishi) or finite and tame
(Gabber). Equivalently, the good moduli space morphism $\stX=[\Spec A/G]\to
\Spec A^G$ is of finite type. The following theorem generalizes this result to
every noetherian stack with a good moduli space.

\begin{theorem} \label{T:good-finite-type}
Let $\cX$ be a noetherian algebraic stack. If $\pi \co \cX \to X$ is a good moduli space with affine diagonal, then $\pi$ is of finite type.
\end{theorem}

\begin{proof}
We may assume that $X = \Spec A$ is affine.  
Let $p \co U=\Spec B \to \cX$ be an affine presentation.  Then $\pi_* (p_* \oh_{U}) = \tilde{B}$.  We need to show that $B$ is a finitely generated $A$-algebra.  This follows from the following lemma. 
\end{proof}

\begin{lemma}
If $\cX$ is a noetherian algebraic stack and $\pi \co \cX \to X$ is a good moduli space, then $\pi_*$ preserves finitely generated algebras.
\end{lemma}

\begin{proof}
Let $\cA$ be a finitely generated $\oh_{\cX}$-algebra. Write $\cA = \dlim_{\lambda} \cF_{\lambda}$ as a union of its finitely generated $\oh_{\cX}$-submodules. Then $\cA$ is generated as an $\oh_{\cX}$-algebra by $\cF_\lambda$ for sufficiently large $\lambda$; that is,  we have a surjection $\Sym(\cF_\lambda) \to \cA$. 
Since $\pi_*$ is exact, it is enough to prove that $\pi_* \Sym(\cF_\lambda)$ is finitely generated. 
But $C:= \Gamma(\cX, \Sym(\cF_\lambda))$ is a $\ZZ$-graded ring which is noetherian by \cite[Thm.~4.16(x)]{alper-good} since $\underline{\Spec}_{\cX}(\Sym(\cF_\lambda))$ is noetherian and $\Spec(C)$ is its good moduli space. It is well-known that $C$ is then finitely generated over $C_0=\Gamma(\cX, \oh_{\cX})=A$. 
\end{proof}

\subsection{Artinian stacks and adic morphisms}

\begin{definition}
We say that an algebraic stack $\stX$ is \emph{artinian} if it is noetherian
and $|\stX|$ is discrete. We say that a quasi-compact and quasi-separated
algebraic stack $\stX$ is \emph{local} if there exists a unique closed point
$x\in |\stX|$.
\end{definition}

Let $\stX$ be a noetherian algebraic stack and let $x\in |\stX|$ be a closed
point with maximal ideal $\idm_x \subset \oh_{\cX}$. The \emph{$n$th infinitesimal neighborhood of
  $x$} is the closed algebraic stack $\stX_x^{[n]}\inj \stX$ defined by
$\idm_x^{n+1}$. Note that $\stX_x^{[n]}$ is artinian and that $\stX_x^{[0]}=\stG_x$ is
the residual gerbe. A local artinian stack $\stX$ is a local artinian scheme
if and only if $\stX_x^{[0]}$ is the spectrum of a field.
\begin{definition} \label{D:adic}
Let $\stX$ and $\stY$ be algebraic stacks, and let $x\in |\stX|$ and $y\in
|\stY|$ be closed points. If $f\colon (\stX,x)\to (\stY,y)$ is a pointed
morphism, then $\stX_x^{[n]}\subseteq f^{-1}(\stY_y^{[n]})$ and we let
$f^{[n]}\colon \stX_x^{[n]}\to \stY_y^{[n]}$ denote the induced morphism.
We say that $f$ is \emph{adic} if $f^{-1}(\stY_y^{[0]})=\stX_x^{[0]}$.
\end{definition}

Note that $f$ is adic precisely when $f^*\idm_y\to \idm_x$ is surjective.
When $f$ is adic, we thus have that $f^{-1}(\stY_y^{[n]})=\stX_x^{[n]}$ for all
$n\geq 0$. Every closed immersion is adic.

\begin{proposition}
Let $\stX$ be a quasi-separated algebraic stack and let $x\in |\stX|$ be a
closed point. Then there exists an adic flat presentation; that is, there
exists an adic flat morphism of finite presentation $p\colon (\Spec A,v)\to
(\stX,x)$. If the stabilizer group of $x$ is smooth, then there exists an adic
smooth presentation.
\end{proposition}
\begin{proof}
The question is local on $\stX$ so we can assume that $\stX$ is quasi-compact.
Start with any smooth presentation $q\colon V=\Spec A\to \stX$. The fiber
$V_x=q^{-1}(\stG_x)=\Spec A/I$ is smooth over the residue field $\kappa(x)$ of
the residual gerbe. Pick a closed point $v\in V_x$ such that
$\kappa(v)/\kappa(x)$ is separable. After replacing $V$ with an open
neighborhood of $v$, we may pick a regular sequence
$\overline{f}_1,\overline{f}_2,\dots,\overline{f}_n\in A/I$ that generates
$\idm_v$. Lift this to a sequence $f_1,f_2,\dots,f_n\in A$ and let $Z\inj V$ be
the closed subscheme defined by this sequence. The sequence is transversely
regular over $\stX$ in a neighborhood $W\subseteq V$ of $v$. In particular,
$U=W\cap Z\to V\to \stX$ is flat. By construction $U_x=Z_x=\Spec \kappa(v)$ so
$(U,v)\to (X,x)$ is an adic flat presentation. Moreover, $U_x=\Spec \kappa(v)\to
\stG_x\to \Spec \kappa(x)$ is \'etale so if the stabilizer group is smooth, then
$U_x\to \stG_x$ is smooth and $U\to X$ is smooth at $v$.
\end{proof}

\begin{corollary}\label{C:adic-artin}
Let $\stX$ be a noetherian algebraic stack. The following statements
are equivalent.
\begin{enumerate}
\item\label{CI:adic-artin-pres}
  There exists an artinian ring $A$ and a flat presentation
$p\colon \Spec A\to \stX$ which is adic at every point of $\stX$.
\item\label{CI:artin-pres}
  There exists an artinian ring $A$ and a flat presentation
$p\colon \Spec A\to \stX$.
\item\label{CI:artin}
  $\stX$ is artinian.
\end{enumerate}
\end{corollary}
\begin{proof}
The implications
\itemref{CI:adic-artin-pres}$\implies$\itemref{CI:artin-pres}$\implies$\itemref{CI:artin}
are trivial. The implication
\itemref{CI:artin}$\implies$\itemref{CI:adic-artin-pres} follows from the
proposition.
\end{proof}

\begin{remark}\label{R:miniversal}
Let $p$ be a smooth morphism $p\colon (U,u)\to (\stX,x)$. We say that $p$ is
\emph{miniversal} at $u$ if the induced morphism $T_{U,u}\to T_{\stX,x}$ on
tangent spaces is an isomorphism. Equivalently, $\Spec \hat{\oh}_{U,u}\to
\stX$ is a formal miniversal deformation space.
If the stabilizer at $x$ is smooth, then
$T_{\stX,x}$ identifies with the normal space $N_x$. Hence, $p$ is miniversal at $u$
if and only if $u$ is a connected component of $p^{-1}(\cG_x)$, that is, if and
only if $p$ is adic after restricting $U$ to a neighborhood of $u$. If the stabilizer at $x$ is
not smooth, then there does not exist smooth adic presentations but there
exists smooth miniversal presentations as well as flat adic presentations.
\end{remark}

If $\cX$ is an algebraic stack, $\cI \subseteq \oh_{\cX}$ is a sheaf of ideals and $\cF$ is a quasi-coherent sheaf, we set $\Gr_{\cI}(\cF) := \oplus_{n \ge 0} \cI^n \cF / \cI^{n+1} \cF$, which is a quasi-coherent sheaf of graded modules on the
closed substack defined by $\cI$.

\begin{proposition}\label{P:closed/iso-cond:infinitesimal}
Let $f\colon (\stX,x)\to (\stY,y)$ be a morphism of noetherian local stacks.
\begin{enumerate}
\item\label{PI:closed:infinitesimal}
  If $f^{[1]}$ is a closed immersion, then $f$ is adic and $f^{[n]}$ is
  a closed immersion for all $n\geq 0$.
\item\label{PI:iso:infinitesimal}
  If $f^{[1]}$ is a closed immersion and there
  exists an isomorphism $\varphi\colon \Gr_{\idm_x}(\oh_X)\to
  (f^{[0]})^*\Gr_{\idm_y}(\oh_Y)$ of graded $\oh_{\stX_x^{[0]}}$-modules,
  then $f^{[n]}$ is an isomorphism for all $n\geq 0$.
\end{enumerate}
\end{proposition}
\begin{proof}
Pick an adic flat
presentation $p\colon \Spec A\to \stY$. After pulling back $f$ along $p$, we
may assume that $\stY=\Spec A$ is a scheme. If $f^{[0]}$ is a closed immersion,
then $\stX_x^{[0]}$ is also a scheme, hence so is $\stX_x^{[n]}$ for all $n\geq 0$.
After replacing $f$ with $f^{[n]}$ for some $n$ we may thus assume that
$\stX=\Spec B$ and $\stY=\Spec A$ are affine and local artinian.
If
in addition $f^{[1]}$ is a closed immersion, then $\idm_A\to
\idm_B/\idm_B^2$ is surjective; hence so is $\idm_AB\to \idm_B$ by Nakayama's
Lemma. We conclude that $f$ is adic and that $A\to B$ is surjective (Nakayama's
Lemma again).

Assume that in addition we have an isomorphism $\varphi\colon
\Gr_{\idm_A}A\cong \Gr_{\idm_B} B$ of graded $k$-vector spaces where
$k=A/\idm_A=B/\idm_B$. Then $\dim_k \idm_A^n/\idm_A^{n+1}=\dim_k
\idm_B^n/\idm_B^{n+1}$. It follows that the surjections
$\idm_A^n/\idm_A^{n+1}\to \idm_B^n/\idm_B^{n+1}$ induced by $f$ are
isomorphisms. It follows that $f$ is an isomorphism.
\end{proof}


\begin{definition}\label{D:complete-local-stack}
Let $\stX$ be an algebraic stack. We say that $\stX$ is a
\emph{complete local stack} if
\begin{enumerate}
\item $\stX$ is local with closed point $x$,
\item $\stX$ is excellent with affine stabilizers, and
\item $\stX$ is coherently complete along the residual gerbe $\cG_x$.
\end{enumerate}
Recall
from Definition \ref{D:cc} that (3) means that the natural functor
\[
  \Coh(\cX)  \to  \ilim_n \Coh\bigl(\cX_x^{[n]}\bigr)
\]
is an equivalence of categories.
\end{definition}

\begin{proposition}\label{P:closed/iso-cond:complete}
Let $f\colon (\stX,x)\to (\stY,y)$ be a morphism of complete local stacks.
\begin{enumerate}
\item\label{PI:closed:complete}
  $f$ is a closed immersion if and only if $f^{[1]}$ is a closed immersion.
\item\label{PI:iso:complete}
  $f$ is an isomorphism if and only if $f^{[1]}$ is a closed immersion and there
  exists an isomorphism $\varphi\colon \Gr_{\idm_x}(\oh_X)\to
  (f^{[0]})^*\Gr_{\idm_y}(\oh_Y)$ of graded $\oh_{\stX_x^{[0]}}$-modules.
\end{enumerate}
\end{proposition}
\begin{proof}
The conditions are clearly necessary. Conversely, if $f^{[1]}$ is a closed
immersion, then $f$ is adic and $f^{[n]}$ is a closed immersion for all $n\geq
0$ by
Proposition~\ref{P:closed/iso-cond:infinitesimal}~\itemref{PI:closed:infinitesimal}. We thus obtain a system of closed immersions $f^{[n]}\colon \stX_x^{[n]}\inj
\stY_y^{[n]}$ which is compatible in the sense that $f^{[m]}$ is the pull-back of $f^{[n]}$ for every $m\leq n$. Since $\stY$ is coherently complete, we obtain a unique closed
substack $\stZ\inj \stY$ such that $\stX_x^{[n]}=\stZ\times_\stY \stY_y^{[n]}$ for all $n$.
If there exists an isomorphism
$\varphi$ as in the second statement, then $f^{[n]}$ is an isomorphism for all
$n\geq 0$ by Proposition~\ref{P:closed/iso-cond:infinitesimal}~\itemref{PI:iso:infinitesimal} and $\stZ=\stY$.
Finally, since
$\stX$, $\stY$ and $\stZ$ are complete local stacks, it follows by Tannaka duality (Theorem \ref{T:tannakian})
that we have an isomorphism $\stX\to \stZ$ over $\stY$ and the result follows.
\end{proof}


\subsection{Artin approximation}
Artin's original approximation theorem applies to the henselization of an
algebra of finite type over a field or an excellent Dedekind
domain~\cite[Thm.~1.12]{artin-approx}. This is sufficient for the main body of
this article but for the generality of this appendix we need Artin approximation
over arbitrary excellent schemes. It is well-known that this follows from
Popescu's theorem (general N\'eron desingularization), see
e.g.\ \cite[Thm.~1.3]{popescu-general}, \cite[Thm.~11.3]{spivakovsky_popescus_theorem} and \cite[Tag
\spref{07QY}]{stacks-project}. We include a proof here for completeness.

\begin{theorem}[Popescu]
A regular homomorphism $A\to B$ between noetherian rings is a
filtered colimit of smooth homomorphisms.
\end{theorem}

Here regular means flat with geometrically regular fibers.
See~\cite[Thm.~1.8]{popescu-general} for the original proof and~\cite{swan_popescus_theorem,spivakovsky_popescus_theorem} or \cite[Tag \spref{07BW}]{stacks-project} for more recent proofs.

\begin{theorem}[Artin approximation]\label{T:artin-approximation}
Let $S=\Spec A$ be the spectrum of a G-ring (e.g., excellent), let $s\in S$ be
a point and let $\widehat{S}=\Spec \widehat{A}$ be the completion at $s$.  Let
$F\colon (\Sch/S)^\op\to \Sets$ be a functor locally of finite
presentation. Let $\overline{\xi}\in F(\widehat{S})$ and let $N\geq 0$ be an
integer. Then there exists an \'etale neighborhood $(S',s')\to (S,s)$ and an
element $\xi'\in F(S')$ such that $\xi'$ and $\overline{\xi}$ coincide in
$F(S^{[N]}_s)$.
\end{theorem}
\begin{proof}
We may replace $A$ by the localization at the prime ideal $\idp$ corresponding
to the point $s$. Since $A$ is a G-ring, the morphism $A\to \widehat{A}$ is
regular and hence a filtered colimit of smooth homomorphisms $A\to A_\lambda$
(Popescu's theorem). Since $F$ is locally of finite presentation, we can thus
find a factorization $A\to A_1\to \widehat{A}$, where $A\to A_1$ is smooth, and
an element $\xi_1\in F(A_1)$ lifting $\overline{\xi}$. After replacing $A_1$ with a
localization $(A_1)_f$ there is a factorization $A\to A[x_1,x_2,\dots,x_n]\to
A_1$ where the second map is \'etale~\cite[IV.17.11.4]{EGA}. Choose a lift
$\varphi\colon A[x_1,x_2,\dots,x_n]\to A$ of
\[
\varphi_N\colon A[x_1,x_2,\dots,x_n]\to A_1\to \widehat{A}\to A/\idp^{N+1}.
\]
Let $A'=A_1\otimes_{A[x_1,x_2,\dots,x_n]} A$ and
let $\xi'\in F(A')$ be the image of $\xi_1$. By construction we have an
$A$-algebra homomorphism $\varphi'_N\colon A'\to A/\idp^{N+1}$ such that the
images of $\xi'$ and $\overline{\xi}$ are equal in $F(A/\idp^{N+1})$.  Since $A\to
A'$ is \'etale the result follows with $S'=\Spec A'$.
\end{proof}


\subsection{Formal versality}

\begin{definition}
Let $\stW$ be a noetherian algebraic stack, let $w\in |\stW|$ be a closed point and let
$\stW_w^{[n]}$ denote the $n$th infinitesimal neighborhood of $w$. Let $\stX$ be
a category fibered in groupoids and let $\eta\colon \stW\to \stX$ be a
morphism. We say that $\eta$ is \emph{formally versal} (resp.\ \emph{formally
  universal}) at $w$ if the following lifting condition holds. Given a
$2$-commutative diagram of solid arrows
\[
\xymatrix{
\stW_w^{[0]}\ar@{(->}[r]^{\iota}
  & \stZ\ar[r]^{f}\ar@{(->}[d]^{g}
  & \stW\ar[d]^{\eta} \\
& \stZ'\ar[r]\ar@{-->}[ur]^{f\mathrlap{'}} & \stX
}
\]
where $\stZ$ and $\stZ'$ are local artinian stacks and $\iota$ and $g$ are
closed immersions, there exists a morphism (resp.\ a unique morphism) $f'$
and $2$-isomorphisms such that the whole diagram is $2$-commutative.
\end{definition}

\begin{proposition}\label{P:formal-versality-criterion}
Let $\eta\colon (\stW,w)\to (\stX,x)$ be a morphism of noetherian algebraic
stacks.  Assume that $w$ and $x$ are closed points.
\begin{enumerate}
\item If $\eta^{[n]}$ is \'etale for every $n$, then
$\eta$ is formally universal at $w$.
\item If $\eta^{[n]}$ is smooth for every $n$ and the stabilizer $G_w$ is
linearly reductive, then $\eta$ is formally versal at $w$.
\end{enumerate}
\end{proposition}
\begin{proof}
We begin with the following observation: if $(\cZ,z)$ is a local artinian stack
and $h\colon (\cZ,z) \to (\mathcal{Q},q)$ is a morphism of algebraic stacks,
where $q$ is a closed point, then there exists an $n$ such that $h$ factors
through $\mathcal{Q}_q^{[n]}$.  Now, if we are given a lifting problem, then
the previous observation shows that we may assume that $\cZ$ and $\cZ'$ factor
through some $\cW_w^{[n]} \to \cX_x^{[n]}$.
The first part is now clear from descent.
For the second part, the obstruction to the existence of a lift belongs to the group $\Ext^1_{\oh_{\cZ}}(f^*L_{\cW_w^{[n]}/\cX_x^{[n]}},I)$, where $I$ is the square zero ideal defining the closed immersion $g$. When $\eta^{[n]}$
is representable, this follows directly from \cite[Thm.~1.5]{olsson-defn}. In
general, this follows from the fundamental exact triangle of the cotangent
complex for $\cZ\to \cW_w^{[n]}\times_{\cX_x^{[n]}} \cZ'\to \cZ'$ and two
applications of \cite[Thm.~1.1]{olsson-defn}.
But $\cZ$ is cohomologically affine and $L_{\cW_w^{[n]}/\cX_x^{[n]}}$ is a
perfect complex of Tor-amplitude $[0,1]$, so the $\Ext$-group vanishes. The
result follows.
\end{proof}


\subsection{Refined Artin--Rees for algebraic stacks}
The results in this section are a generalization of~\cite[\S3]{conrad-dejong} (also
see \cite[Tag~\spref{07VD}]{stacks-project}) from rings to algebraic stacks.

\begin{definition}\label{D:AR}
Let $\stX$ be a noetherian algebraic stack and let $\stZ\inj \stX$ be a closed
substack defined by the ideal $\sI\subseteq \oh_\stX$. Let $\varphi\colon
\sE\to \sF$ be a homomorphism of coherent sheaves on $\stX$. Let $c\geq 0$ be
an integer. We say that $\AR{c}$ holds for $\varphi$ along $\stZ$ if
\[
\varphi(\sE)\cap \sI^n\sF \subseteq \varphi(\sI^{n-c}\sE),\quad \forall n\geq c .
\]
\end{definition}

When $\stX$ is a scheme, $\AR{c}$ holds for all sufficiently large $c$ by the
Artin--Rees lemma. If $\pi\colon U\to \stX$ is a flat presentation, then
$\AR{c}$ holds for $\varphi$ along $\stZ$ if and only if $\AR{c}$ holds for
$\pi^*\varphi\colon\pi^*\sE\to \pi^*\sF$ along $\pi^{-1}(\stZ)$. In particular
$\AR{c}$ holds for $\varphi$ along $\stZ$ for all sufficiently large $c$.
If $f\colon \sE'\surj \sE$ is a surjective homomorphism, then $\AR{c}$ for
$\varphi$ holds if and only if $\AR{c}$ for $\varphi\circ f$ holds.

In the following section, we will only use the case when $|\stZ|$ is a closed
point.

\begin{theorem}\label{T:Artin-Rees}
Let $\sE_2\xrightarrow{\alpha} \sE_1\xrightarrow{\beta} \sE_0$
and $\sE'_2\xrightarrow{\alpha'} \sE'_1\xrightarrow{\beta'} \sE'_0$
be two complexes of coherent sheaves on a noetherian algebraic stack $\cX$.  Let $\stZ\inj \stX$ be a closed
substack defined by the ideal $\sI\subseteq \oh_\stX$.  Let $c$ be a positive integer. Assume
that
\begin{enumerate}
\item $\sE_0,\sE'_0,\sE_1,\sE'_1$ are vector bundles,
\item the sequences are isomorphic after tensoring with $\oh_{\cX}/\sI^{c+1}$,
\item the first sequence is exact, and
\item $\AR{c}$ holds for $\alpha$ and $\beta$ along $\stZ$.
\end{enumerate}
Then
\begin{enumerate}[label=(\alph*)]
\item the second sequence is exact in a neighborhood of $\stZ$;
\item $\AR{c}$ holds for $\beta'$ along $\stZ$; and
\item given an isomorphism $\varphi\colon \sE_0\to \sE'_0$, there exists a
  unique isomorphism $\psi$ of $\Gr_\sI(\oh_X)$-modules in the diagram
\[
\xymatrix@C+2mm{
\Gr_\sI(\sE_0)\ar@{->>}[r]^-{\Gr(\gamma)}\ar[d]^{\Gr(\varphi)}_{\cong}
 & \Gr_\sI(\coker \beta)\ar[d]^{\psi}_{\cong} \\
\Gr_\sI(\sE'_0)\ar@{->>}[r]^-{\Gr(\gamma')}
 & \Gr_\sI(\coker \beta')
}%
\]
where $\gamma\colon \sE_0\to \coker \beta$ and
$\gamma'\colon \sE'_0\to \coker \beta'$ denote the induced maps.
\end{enumerate}
\end{theorem}
\begin{proof}
Note that there exists an isomorphism $\psi$ if and only if $\ker
\Gr(\gamma)=\ker \Gr(\gamma')$.  All three statements can thus be checked after
pulling back to a presentation $U\to \stX$. We may also localize and assume
that $\stX=U=\Spec A$ where $A$ is a local ring. Then all vector bundles are
free and we may choose isomorphisms $\sE_i\cong\sE'_i$ for $i=0,1$ such that
$\beta=\beta'$ modulo $\sI^{c+1}$. We can also choose a surjection $\epsilon'\colon \oh_U^n\surj
\sE'_2$ and a lift $\epsilon\colon \oh_U^n\surj \sE_2$ modulo $\sI^{c+1}$, so that
$\alpha\circ\epsilon=\alpha'\circ\epsilon'$ modulo $\sI^{c+1}$. Thus, we may assume that $\sE_i=\sE'_i$ for
$i=0,1,2$ are free. The result then follows from~\cite[Lem.~3.1 and
  Thm.~3.2]{conrad-dejong} or \cite[Tags~\spref{07VE} and
  \spref{07VF}]{stacks-project}.
\end{proof}


\subsection{Equivariant algebraization}
We now consider the equivariant generalization of Artin's algebraization
theorem, see~\cite[Thm.~1.6]{artin-algebraization} and
\cite[Thm.~1.5, Rem.~1.7]{conrad-dejong}. In fact, we give a general
algebraization theorem for algebraic stacks.

\begin{theorem}\label{T:approximate-algebraization}
Let $S$ be an excellent scheme and let $T$ be a noetherian algebraic space over
$S$. Let $\cZ$ be an algebraic stack of finite presentation over $T$ and let
$z\in |\cZ|$ be a closed point such that $\cG_z\to S$ is of finite type.  Let $t \in T$ be the image of $z$.  Let
$\cX_1,\dots,\cX_n$ be categories fibered in groupoids over $S$, locally of finite presentation. Let $\eta\co \cZ \to
\cX=\cX_1\times_S\dots \times_S \cX_n$ be a morphism. Fix an integer $N\geq 0$. Then there exists
\begin{enumerate}
\item\label{TI:approx-first}
  an affine scheme $S'$ of finite type over $S$ and a closed point $s' \in S'$ mapping to the same point in $S$ as $t \in T$;
\item an algebraic stack $\cW \to S'$ of finite type;
\item a closed point $w\in |\cW|$ over $s'$;
\item a morphism $\xi\co \cW \to \cX$;
\item\label{TI:approx-last}
  an isomorphism $\cZ \times_T T^{[N]}_t \cong \cW \times_{S'} S'^{[N]}_{s'}$ over $\cX$ mapping $z$ to $w$; 
in particular, there is an isomorphism $\cZ^{[N]}_z\cong \cW^{[N]}_w$ over $\cX$; and
\item\label{TI:approx-graded-iso}
  an isomorphism $\Gr_{\fm_z}\oh_{\cZ}\cong \Gr_{\fm_w}\oh_{\cW}$ of
graded algebras over $\cZ^{[0]}_z\cong \cW^{[0]}_w$.
\end{enumerate}
Moreover, if $\cX_i$ is a quasi-compact algebraic stack and $\eta_i\co \cZ \to
\cX_i$ is affine for some $i$, then it can be arranged so that $\xi_i \co
\cW \to \cX_i$ is affine.
\end{theorem}
\begin{proof}
We may assume that $S=\Spec A$ is affine. Let $t\in T$ be the image of $z$. By replacing $T$ with the completion $\hat{T}=\Spec \hat{\oh}_{T,t}$
and $\cZ$ with $\cZ \times_T \hat{T}$, we may assume that $T=\hat{T}=\Spec
B$ where $B$ is a complete local ring. By standard limit methods, we have an
affine scheme $S_0=\Spec B_0$ and an algebraic stack $\cZ_0\to S_0$ of finite
presentation and a commutative diagram
\[
\xymatrix{%
\cZ\ar[r]\ar[d]\ar@/^2ex/[rr] & \cZ_0\ar[r]\ar[d] & \stX\ar[d] \\
T\ar[r] & S_0\ar[r]\ar@{}[ul]|\square & S
}%
\]
If $\stX_i$ is algebraic and quasi-compact and $\cZ\to \stX_i$ is
affine for some $i$, we may also arrange so that $\cZ_0\to \stX_i$ is
affine~\cite[Thm.~C]{rydh-noetherian}.

Since $\stG_z\to S$ is of finite type, so is $\Spec \kappa(t)\to S$. We may
thus choose a factorization $T\to S_1=\AA^n_{S_0}\to S_0$, such that
$T\to \hat{S}_1=\Spec \hat{\oh}_{S_1,s_1}$ is a
closed immersion; here $s_1\in S_1$ denotes the image of $t\in T$.
After replacing $S_1$ with an open neighborhood, we may assume that $s_1$ is a
closed point. Let
$\stZ_1=\stZ_0\times_{S_0} S_1$ and
$\hat{\stZ}_1=\stZ_1\times_{S_1}\hat{S}_1$. Consider the functor
$F\colon (\Sch/S_1)^\op\to \Sets$ where $F(U \to S_1)$ is the set of isomorphism
classes of complexes 
$$ \sE_2\xrightarrow{\alpha} \sE_1\xrightarrow{\beta}
  \oh_{\stZ_1 \times_{S_1}U}$$ 
  of finitely presented quasi-coherent $\oh_{\stZ_1 \times_{S_1} U} $-modules with $\cE_1$ locally free. By standard limit arguments, $F$ is locally of finite presentation.

We have an element $(\alpha,\beta)\in F(\hat{S}_1)$ such that $\im(\beta)$
defines $\stZ \inj \hat{\stZ}_1$. Indeed, choose a resolution
\[
\hat{\oh}_{S_1,s_1}^{\mathrlap{\oplus r}}\xrightarrow{\quad\tilde{\beta}\;\;}\hat{\oh}_{S_1,s_1}\surj B.
\]
After pulling back $\tilde{\beta}$ to $\hat{\stZ}_1$, we obtain a resolution
\[
\ker(\beta)\xhookrightarrow{\;\;\alpha\;\;} \oh_{\hat{\stZ}_1}^{\oplus r}
  \xrightarrow{\;\;\beta\;\;} \oh_{\hat{\stZ}_1}\surj \oh_{\stZ}.
\]
After increasing $N$, we may assume that $\AR{N}$ holds for $\alpha$
and $\beta$ at $z$.

Artin approximation (Theorem~\ref{T:artin-approximation}) gives
an \'etale neighborhood $(S',s') \to (S_1,s_1)$ and an
element $(\alpha',\beta')\in F(S')$ such that $(\alpha,\beta)=
(\alpha',\beta')$ in $F(S_{1,s_1}^{[N]})$. We let $\stW\inj \stZ_1\times_{S_1}
S'$ be the closed substack defined by $\im(\beta')$.  Then $\stZ \times_T T_t^{[N]}$ and
$\stW \times_{S'} S'^{[N]}_{s'}$ are equal as closed substacks of $\stZ_1\times_{S_1} S_{1,s_1}^{[n]}$ and \itemref{TI:approx-first}--\itemref{TI:approx-last}
follows. Finally~\itemref{TI:approx-graded-iso} follows from Theorem~\ref{T:Artin-Rees}.
\end{proof}

\begin{theorem}\label{T:algebraization}
Let $S$, $T$, $\stZ$, $\eta$, $N$, $\stW$ and $\xi$ be as in
Theorem~\ref{T:approximate-algebraization}.  If $\eta_1\colon \stZ\to \stX_1$ is formally
versal, then there are compatible isomorphisms $\varphi_n\colon \stZ_z^{[n]} \iso
\stW_w^{[n]}$ over $\stX_1$ for all $n\geq 0$. For $n\leq N$, the isomorphism
$\varphi_n$ is also compatible with $\eta$ and $\xi$.
\end{theorem}
\begin{proof}
We can assume that $N\geq 1$. By Theorem~\ref{T:approximate-algebraization}, we have an
isomorphism $\varphi_N\co \cZ_z^{[N]} \to \cW_w^{[N]}$  over
$\stX$. By formal versality and induction on $n \geq N$, we can extend $\psi_N=\varphi^{-1}_N$ to compatible morphisms
$\psi_n\colon \cW_w^{[n]}\to \cZ$ over $\stX_1$.  Indeed, formal versality allows us to find a dotted arrow such the diagram
$$\xymatrix@C+7mm{
\cW_w^{[n]} \ar[r]^-{\psi_n} \ar[d]		& \cZ \ar[d]^{\eta} \\
\cW_w^{[n+1]} \ar[r]_-{\xi_1|_{\cW_w^{[n+1]}}} \ar@{-->}[ur]^{\psi_{n+1}}	& \cX_1
}$$
is 2-commutative.  By
Proposition~\ref{P:closed/iso-cond:infinitesimal}~\itemref{PI:iso:infinitesimal}, $\psi_n$ induces an
isomorphism $\varphi_n\colon \stZ_z^{[n]}\to \stW_w^{[n]}$. 
\end{proof}

We now formulate the theorem above in a manner which is transparently an equivariant analogue of Artin algebraization \cite[Thm.~1.6]{artin-algebraization}.  It is this formulation that is directly applied to prove Theorem \ref{T:field}.

\begin{corollary}\label{C:equivariant-algebraization}
Let $H$ be a linearly reductive affine group scheme over an algebraically closed field $k$.
Let $\cX$ be a noetherian algebraic stack of
finite type over $k$ with affine stabilizers.
Let $\hat{\cH} = [\Spec C / H]$ be a noetherian algebraic stack over $k$. Suppose that $C^{H}$ is a complete local $k$-algebra.
 Let $\eta\colon \hat{\cH} \to \cX$ be a morphism
that is formally versal at a closed point $z\in |\hat{\cH} |$.  Let $N \ge 0$. Then there exists 
\begin{enumerate}
\item\label{CI:equialg:first}
  an algebraic stack $\cW = [\Spec A / H]$ of finite type over $k$;
\item a closed point $w \in |\cW|$;
\item\label{CI:equialg:third}
  a morphism
$f \co \cW \to \cX$;
\item
  a morphism
$\varphi \co (\hat{\cH},z) \to (\cW,w)$;
\item a $2$-isomorphism $\tau\co \eta\Rightarrow f\circ \varphi$; and
\item\label{CI:equialg:H-torsors-equal-over-N-truncation}
  a $2$-isomorphism $\nu_N\co \alpha^{[N]}\Rightarrow \beta^{[N]}\circ
  \varphi^{[N]}$ where $\alpha\co \hat{\cH}\to BH$ and $\beta\co \cW\to BH$
  denote the structure morphisms.
\end{enumerate}
such that for all $n$, the induced morphism $\varphi^{[n]} \co \hat{\cH}^{[n]}_z \to \cW^{[n]}_w$ is an isomorphism.
In particular, $\varphi$ induces an isomorphism $\hat{\varphi} \co  \hat{\cH} \to \hat{\cW}$ where $\hat{\cW}$ is the coherent completion of $\cW$ at $w$ (i.e., $\hat{\cW}= \cW \times_{W} \Spec \hat{\oh}_{W,w_0}$ where $W = \Spec A^H$ and $w_0 \in W$ is the image of $w $ under $\cW \to W$).
\end{corollary}
\begin{proof}
By Theorem \ref{T:good-finite-type}, the good moduli space
$\hat{\cH}\to \Spec C^H$ is of finite type.
If we apply Theorem~\ref{T:algebraization} with $S = \Spec k$, $T=\Spec C^H$,
$\cZ=\hat{\cH}$,
$\cX_1 = \cX$
and $\cX_2 = BH$, then we immediately
obtain~\itemref{CI:equialg:first}--\itemref{CI:equialg:third} together with
isomorphisms $\varphi_n \co \hat{\cH}_z^{[n]}\to \cW_w^{[n]}$, a compatible
system of $2$-isomorphisms $\{\tau_n \co \eta^{[n]}\Rightarrow f^{[n]}\circ
\varphi^{[n]}\}_{n\geq 0}$ for all $n$, and a $2$-isomorphism $\nu_N$ as in
\itemref{CI:equialg:H-torsors-equal-over-N-truncation}. Since $\hat{\cH}$ and
$\hat{\cW}$ are coherently complete (Theorem \ref{key-theorem}), the
isomorphisms $\varphi_n$ yield an isomorphism
$\hat{\varphi} \co \hat{\cH}\to \hat{\cW}$ and an induced morphism
$\varphi \co \hat{\cH}\to \cW$ by 
Tannaka duality (Corollary \ref{C:tannakian}).
Likewise, the system $\{\tau_n\}$ induces a $2$-isomorphism $\tau\co \eta
\Rightarrow f\circ \varphi$ by Tannaka duality (full faithfulness in
Corollary \ref{C:tannakian}).
\end{proof}

\begin{remark}
If $\cX$ is merely a category fibered in groupoids over $k$ that is locally of
finite presentation (analogously to the situation in
\cite[Thm.~1.6]{artin-algebraization}), then
Corollary~\ref{C:equivariant-algebraization} and its proof remain valid except
that instead of the $2$-isomorphism $\tau$ we only have the system
$\{\tau_n\}$.
\end{remark}


\bibliography{refs}

\providecommand{\MR}{\relax\ifhmode\unskip\space\fi MR }
\providecommand{\MRhref}[2]{%
  \href{http://www.ams.org/mathscinet-getitem?mr=#1}{#2}
}
\providecommand{\href}[2]{#2}
\begin{thebibliography}{SGA3\textsubscript{II}}

\bibitem[ACFW13]{acfw}
D.~Abramovich, C.~Cadman, B.~Fantechi, and J.~Wise, \emph{Expanded
  degenerations and pairs}, Comm. Algebra \textbf{41} (2013), no.~6,
  2346--2386.

\bibitem[AFS17]{afsw-good}
J.~Alper, M.~Fedorchuk, and D.~I. Smyth, \emph{Second flip in the
  {H}assett-{K}eel program: existence of good moduli spaces}, Compos. Math.
  \textbf{153} (2017), no.~8, 1584--1609.

\bibitem[AHR19]{ahr2}
J.~{Alper}, J.~{Hall}, and D.~{Rydh}, \emph{{The \'etale local structure of
  algebraic stacks}}, December 2019,
  \href{http://arXiv.org/abs/1912.06162}{\mbox{arXiv:1912.06162}}.

\bibitem[AK16]{alper-kresch}
J.~Alper and A.~Kresch, \emph{Equivariant versal deformations of semistable
  curves}, Michigan Math. J. \textbf{65} (2016), no.~2, 227--250.

\bibitem[Alp10]{alper-quotient}
J.~Alper, \emph{On the local quotient structure of {A}rtin stacks}, J. Pure
  Appl. Algebra \textbf{214} (2010), no.~9, 1576--1591.

\bibitem[Alp13]{alper-good}
J.~Alper, \emph{Good moduli spaces for {A}rtin stacks}, Ann. Inst. Fourier
  (Grenoble) \textbf{63} (2013), no.~6, 2349--2402.

\bibitem[Alp14]{alper-adequate}
J.~Alper, \emph{Adequate moduli spaces and geometrically reductive group
  schemes}, Algebr. Geom. \textbf{1} (2014), no.~4, 489--531.

\bibitem[AOV08]{tame}
D.~Abramovich, M.~Olsson, and A.~Vistoli, \emph{Tame stacks in positive
  characteristic}, Ann. Inst. Fourier (Grenoble) \textbf{58} (2008), no.~4,
  1057--1091.

\bibitem[Art69a]{artin-approx}
M.~Artin, \emph{Algebraic approximation of structures over complete local
  rings}, Inst. Hautes \'Etudes Sci. Publ. Math. (1969), no.~36, 23--58.

\bibitem[Art69b]{artin-algebraization}
M.~Artin, \emph{Algebraization of formal moduli. {I}}, Global {A}nalysis
  ({P}apers in {H}onor of {K}. {K}odaira), Univ. Tokyo Press, Tokyo, 1969,
  pp.~21--71.

\bibitem[Aus66]{auslander}
M.~Auslander, \emph{Coherent functors}, Proc. {C}onf. {C}ategorical {A}lgebra
  ({L}a {J}olla, {C}alif., 1965), Springer, New York, 1966, pp.~189--231.

\bibitem[BB73]{bb}
A.~Bia{\l}ynicki-Birula, \emph{Some theorems on actions of algebraic groups},
  Ann. of Math. (2) \textbf{98} (1973), 480--497.

\bibitem[Bri15]{brion-linearization}
M.~Brion, \emph{On linearization of line bundles}, J. Math. Sci. Univ. Tokyo
  \textbf{22} (2015), no.~1, 113--147.

\bibitem[BZFN10]{BZFN}
D.~Ben-Zvi, J.~Francis, and D.~Nadler, \emph{Integral transforms and {D}rinfeld
  centers in derived algebraic geometry}, J. Amer. Math. Soc. \textbf{23}
  (2010), no.~4, 909--966.

\bibitem[CF11]{MR2776372}
M.~Crainic and R.~L. Fernandes, \emph{A geometric approach to {C}onn's
  linearization theorem}, Ann. of Math. (2) \textbf{173} (2011), no.~2,
  1121--1139.

\bibitem[CJ02]{conrad-dejong}
B.~Conrad and A.~J. de~Jong, \emph{Approximation of versal deformations}, J.
  Algebra \textbf{255} (2002), no.~2, 489--515.

\bibitem[Con05]{conrad-gaga}
B.~Conrad, \emph{Formal {GAGA} for {A}rtin stacks},
  \url{http://math.stanford.edu/~conrad/papers/formalgaga.pdf}, 2005.

\bibitem[Con10]{mathoverflow_groups-over-dual-numbers}
B.~Conrad, \emph{Smooth linear algebraic groups over the dual numbers},
  MathOverflow question at
  \url{http://mathoverflow.net/questions/22078/smooth-linear-algebraic-groups-over-the-dual-numbers},
  2010.

\bibitem[CS13]{MR3185351}
M.~Crainic and I.~Struchiner, \emph{On the linearization theorem for proper
  {L}ie groupoids}, Ann. Sci. \'Ec. Norm. Sup\'er. (4) \textbf{46} (2013),
  no.~5, 723--746.

\bibitem[{Dri}13]{drinfeld}
V.~{Drinfeld}, \emph{{On algebraic spaces with an action of $\mathbb{G}_m$}},
  2013, \href{http://arXiv.org/abs/1308.2604}{\mbox{arXiv:1308.2604}}.

\bibitem[EGA]{EGA}
A.~Grothendieck, \emph{\'{E}l\'ements de g\'eom\'etrie alg\'ebrique}, I.H.E.S.
  Publ. Math. \textbf{4, 8, 11, 17, 20, 24, 28, 32} (1960, 1961, 1961, 1963,
  1964, 1965, 1966, 1967).

\bibitem[EHKV01]{ehkv}
D.~Edidin, B.~Hassett, A.~Kresch, and A.~Vistoli, \emph{Brauer groups and
  quotient stacks}, Amer. J. Math. \textbf{123} (2001), no.~4, 761--777.

\bibitem[Eis95]{eisenbud}
D.~Eisenbud, \emph{Commutative algebra}, Graduate Texts in Mathematics, vol.
  150, Springer-Verlag, New York, 1995, With a view toward algebraic geometry.

\bibitem[FMN10]{fantechi-mann-nironi}
B.~Fantechi, E.~Mann, and F.~Nironi, \emph{Smooth toric {D}eligne-{M}umford
  stacks}, J. Reine Angew. Math. \textbf{648} (2010), 201--244.

\bibitem[Gro17]{gross-resolution}
P.~Gross, \emph{Tensor generators on schemes and stacks}, Algebr. Geom.
  \textbf{4} (2017), no.~4, 501--522.

\bibitem[GS15]{GS-toric-stacks-2}
A.~Geraschenko and M.~Satriano, \emph{Toric stacks {II}: {I}ntrinsic
  characterization of toric stacks}, Trans. Amer. Math. Soc. \textbf{367}
  (2015), no.~2, 1073--1094.

\bibitem[GS19]{GS-toric-stacks-2-erratum}
A.~Geraschenko and M.~Satriano, \emph{Toric stacks {II}: {I}ntrinsic
  characterization of toric stacks: erratum}, in preparation (2019).

\bibitem[GY83]{MR706507}
S.~Goto and K.~Yamagishi, \emph{Finite generation of {N}oetherian graded
  rings}, Proc. Amer. Math. Soc. \textbf{89} (1983), no.~1, 41--44.

\bibitem[GZB15]{geraschenko-brown}
A.~Geraschenko and D.~Zureick-Brown, \emph{Formal {GAGA} for good moduli
  spaces}, Algebr. Geom. \textbf{2} (2015), no.~2, 214--230.

\bibitem[Hal14a]{hallj_coho_bc}
J.~Hall, \emph{Cohomology and base change for algebraic stacks}, Math. Z.
  \textbf{278} (2014), no.~1-2, 401--429.

\bibitem[{Hal}14b]{dhl-instability}
D.~{Halpern-Leistner}, \emph{{On the structure of instability in moduli
  theory}}, 2014,
  \href{http://arXiv.org/abs/1411.0627}{\mbox{arXiv:1411.0627}}.

\bibitem[Hal17]{hallj_openness_coh}
J.~Hall, \emph{Openness of versality via coherent functors}, J. Reine Angew.
  Math. \textbf{722} (2017), 137--182.

\bibitem[Har98]{hartshorne-coherent}
R.~Hartshorne, \emph{Coherent functors}, Adv. Math. \textbf{140} (1998), no.~1,
  44--94.

\bibitem[Hes81]{hesselink}
W.~H. Hesselink, \emph{Concentration under actions of algebraic groups}, Paul
  {D}ubreil and {M}arie-{P}aule {M}alliavin {A}lgebra {S}eminar, 33rd {Y}ear
  ({P}aris, 1980), Lecture Notes in Math., vol. 867, Springer, Berlin, 1981,
  pp.~55--89.

\bibitem[HLP14]{hlp}
D.~Halpern-Leistner and A.~Preygel, \emph{{Mapping stacks and categorical
  notions of properness}}, 2014,
  \href{http://arXiv.org/abs/1402.3204}{\mbox{arXiv:1402.3204}}.

\bibitem[HNR19]{hallj_neeman_dary_no_compacts}
J.~Hall, A.~Neeman, and D.~Rydh, \emph{One positive and two negative results
  for derived categories of algebraic stacks}, J. Inst. Math. Jussieu
  \textbf{18} (2019), no.~5, 1087--1111.

\bibitem[HR15]{hallj_dary_alg_groups_classifying}
J.~Hall and D.~Rydh, \emph{Algebraic groups and compact generation of their
  derived categories of representations}, Indiana Univ. Math. J. \textbf{64}
  (2015), no.~6, 1903--1923.

\bibitem[HR17]{perfect_complexes_stacks}
J.~Hall and D.~Rydh, \emph{Perfect complexes on algebraic stacks}, Compositio
  Math. \textbf{153} (2017), no.~11, 2318--2367.

\bibitem[HR19]{hallj_dary_coherent_tannakian_duality}
J.~Hall and D.~Rydh, \emph{Coherent {T}annaka duality and algebraicity of
  {$\Hom$}-stacks}, Algebra Number Theory \textbf{13} (2019), no.~7,
  1633--1675.

\bibitem[ILO14]{MR3309086}
L.~Illusie, Y.~Laszlo, and F.~Orgogozo (eds.), \emph{Travaux de {G}abber sur
  l'uniformisation locale et la cohomologie \'etale des sch\'emas
  quasi-excellents}, Soci\'et\'e Math\'ematique de France, Paris, 2014,
  S{\'e}minaire {\`a} l'{\'E}cole Polytechnique 2006--2008. [Seminar of the
  Polytechnic School 2006--2008], With the collaboration of Fr{\'e}d{\'e}ric
  D{\'e}glise, Alban Moreau, Vincent Pilloni, Michel Raynaud, Jo{\"e}l Riou,
  Beno{\^{\i}}t Stroh, Michael Temkin and Weizhe Zheng, Ast{\'e}risque No.
  363-364 (2014).

\bibitem[JS19]{js-bb}
J.~Jelisiejew and {\L}.~Sienkiewicz, \emph{Bia{\l}ynicki-{B}irula decomposition
  for reductive groups}, J. Math. Pures Appl. (9) \textbf{131} (2019),
  290--325.

\bibitem[KM97]{keel-mori}
S.~Keel and S.~Mori, \emph{Quotients by groupoids}, Ann. of Math. (2)
  \textbf{145} (1997), no.~1, 193--213.

\bibitem[Kon82]{konarski}
J.~Konarski, \emph{A pathological example of an action of {$k^{\ast} $}}, Group
  actions and vector fields ({V}ancouver, {B}.{C}., 1981), Lecture Notes in
  Math., vol. 956, Springer, Berlin, 1982, pp.~72--78.

\bibitem[Kra03]{MR2026723}
H.~Krause, \emph{Coherent functors and covariantly finite subcategories},
  Algebr. Represent. Theory \textbf{6} (2003), no.~5, 475--499.

\bibitem[Lie06]{lieblich-coherent}
M.~Lieblich, \emph{Remarks on the stack of coherent algebras}, Int. Math. Res.
  Not. (2006), Art. ID 75273, 12.

\bibitem[LMB]{MR1771927}
G.~Laumon and L.~Moret-Bailly, \emph{Champs alg\'ebriques}, Ergebnisse der
  Mathematik und ihrer Grenzgebiete. 3. Folge., vol.~39, Springer-Verlag,
  Berlin, 2000.

\bibitem[Lun73]{luna}
D.~Luna, \emph{Slices \'etales}, Sur les groupes alg\'ebriques, Bull. Soc.
  Math. France, M\'em., vol.~33, Soc. Math. France, Paris, 1973, pp.~81--105.

\bibitem[Mil17]{milne}
J.~S. Milne, \emph{Algebraic groups}, Cambridge Studies in Advanced
  Mathematics, vol. 170, Cambridge University Press, Cambridge, 2017, The
  theory of group schemes of finite type over a field.

\bibitem[Nag62]{nagata}
M.~Nagata, \emph{Complete reducibility of rational representations of a matric
  group.}, J. Math. Kyoto Univ. \textbf{1} (1961/1962), 87--99.

\bibitem[Nee96]{neeman_duality}
A.~Neeman, \emph{The {G}rothendieck duality theorem via {B}ousfield's
  techniques and {B}rown representability}, J. Amer. Math. Soc. \textbf{9}
  (1996), no.~1, 205--236.

\bibitem[Ols03]{olssonloggeometry}
M.~Olsson, \emph{Logarithmic geometry and algebraic stacks}, Ann. Sci. \'Ecole
  Norm. Sup. (4) \textbf{36} (2003), no.~5, 747--791.

\bibitem[Ols05]{olsson-proper}
M.~Olsson, \emph{On proper coverings of {A}rtin stacks}, Adv. Math.
  \textbf{198} (2005), no.~1, 93--106.

\bibitem[Ols06]{olsson-defn}
M.~Olsson, \emph{Deformation theory of representable morphisms of algebraic
  stacks}, Math. Z. \textbf{253} (2006), no.~1, 25--62.

\bibitem[Opr06]{oprea}
D.~Oprea, \emph{Tautological classes on the moduli spaces of stable maps to
  {$\mathbb{P}^r$} via torus actions}, Adv. Math. \textbf{207} (2006), no.~2,
  661--690.

\bibitem[Pop86]{popescu-general}
D.~Popescu, \emph{General {N}\'eron desingularization and approximation},
  Nagoya Math. J. \textbf{104} (1986), 85--115.

\bibitem[PPT14]{MR3259039}
M.~J. Pflaum, H.~Posthuma, and X.~Tang, \emph{Geometry of orbit spaces of
  proper {L}ie groupoids}, J. Reine Angew. Math. \textbf{694} (2014), 49--84.

\bibitem[Rom05]{romagny}
M.~Romagny, \emph{Group actions on stacks and applications}, Michigan Math. J.
  \textbf{53} (2005), no.~1, 209--236.

\bibitem[Ryd15]{rydh-noetherian}
D.~Rydh, \emph{Noetherian approximation of algebraic spaces and stacks}, J.
  Algebra \textbf{422} (2015), 105--147.

\bibitem[SGA1]{sga1}
A.~Grothendieck and M.~Raynaud, \emph{Rev{\^e}tements {\'e}tales et groupe
  fondamental}, S{\'e}minaire de G{\'e}om{\'e}trie Alg{\'e}brique, I.H.E.S.,
  1963.

\bibitem[SGA3\textsubscript{II}]{sga3ii}
\emph{Sch\'emas en groupes. {II}: {G}roupes de type multiplicatif, et structure
  des sch\'emas en groupes g\'en\'eraux}, S\'eminaire de G\'eom\'etrie
  Alg\'ebrique du Bois Marie 1962/64 (SGA 3). Dirig\'e par M. Demazure et A.
  Grothendieck. Lecture Notes in Mathematics, Vol. 152, Springer-Verlag,
  Berlin, 1970.

\bibitem[Sko13]{skowera}
J.~Skowera, \emph{Bia\l ynicki-{B}irula decomposition of {D}eligne-{M}umford
  stacks}, Proc. Amer. Math. Soc. \textbf{141} (2013), no.~6, 1933--1937.

\bibitem[Smy13]{smyth_towards-a-classification-modular-compactifications}
D.~I. Smyth, \emph{Towards a classification of modular compactifications of
  {$\mathcal{M}_{g,n}$}}, Invent. Math. \textbf{192} (2013), no.~2, 459--503.

\bibitem[Som82]{sommese}
A.~J. Sommese, \emph{Some examples of {${\bf C}^{\ast} $} actions}, Group
  actions and vector fields ({V}ancouver, {B}.{C}., 1981), Lecture Notes in
  Math., vol. 956, Springer, Berlin, 1982, pp.~118--124.

\bibitem[Spi99]{spivakovsky_popescus_theorem}
M.~Spivakovsky, \emph{A new proof of {D}. {P}opescu's theorem on smoothing of
  ring homomorphisms}, J. Amer. Math. Soc. \textbf{12} (1999), no.~2, 381--444.

\bibitem[Stacks]{stacks-project}
The Stacks~Project Authors, \emph{{\itshape Stacks Project}},
  \url{http://stacks.math.columbia.edu/}.

\bibitem[Sum74]{sumihiro}
H.~Sumihiro, \emph{Equivariant completion}, J. Math. Kyoto Univ. \textbf{14}
  (1974), 1--28.

\bibitem[Sum75]{sumihiro2}
H.~Sumihiro, \emph{Equivariant completion. {II}}, J. Math. Kyoto Univ.
  \textbf{15} (1975), no.~3, 573--605.

\bibitem[Swa98]{swan_popescus_theorem}
R.~G. Swan, \emph{N\'eron--{P}opescu desingularization}, Algebra and geometry
  (Taipei, 1995), Lect. Algebra Geom., vol.~2, Internat. Press, Cambridge, MA,
  1998, pp.~135--192.

\bibitem[Tod20]{toda-hallalg}
Y.~Toda, \emph{Hall algebras in the derived category and higher-rank {DT}
  invariants}, Algebr. Geom. \textbf{7} (2020), no.~3, 240--262.

\bibitem[Tot04]{totaro}
B.~Totaro, \emph{The resolution property for schemes and stacks}, J. Reine
  Angew. Math. \textbf{577} (2004), 1--22.

\bibitem[Wat60]{MR0118757}
C.~E. Watts, \emph{Intrinsic characterizations of some additive functors},
  Proc. Amer. Math. Soc. \textbf{11} (1960), 5--8.

\bibitem[Wei00]{weinstein_linearization}
A.~Weinstein, \emph{Linearization problems for {L}ie algebroids and {L}ie
  groupoids}, Lett. Math. Phys. \textbf{52} (2000), no.~1, 93--102, Conference
  Mosh{\'e} Flato 1999 (Dijon).

\bibitem[Zun06]{zung_proper_grpds}
N.~T. Zung, \emph{Proper groupoids and momentum maps: linearization, affinity,
  and convexity}, Ann. Sci. \'Ecole Norm. Sup. (4) \textbf{39} (2006), no.~5,
  841--869.

\end{thebibliography}
\bibliographystyle{dary}
\end{document}